\documentclass{article}
\usepackage{fullpage,cite}
\usepackage{mysty,bm}
\usepackage[capitalise]{cleveref}
\usepackage{amssymb}

\newcommand{\W}{\mathcal{W}_t}
\newcommand{\cond}[1]{\mathrm{cond}({#1})}

\begin{document}

\title{A Preconditioned Riemannian Gradient Descent Algorithm for Low-Rank Matrix Recovery}
\author{
Fengmiao Bian\thanks{Department of Mathematics, The Hong Kong University of Science and Technology, Hong Kong, China.
 E-mail: mafmbian@ust.hk} \and
		Jian-Feng Cai\thanks{Department of Mathematics, The Hong Kong University of Science and Technology, Hong Kong, China. E-mail: jfcai@ust.hk} \and
		Rui Zhang\thanks{Theory Lab, Central Research Institute, 2012 Labs, Huawei Technologies Co. Ltd., Hong Kong, China. E-mail: zhangrui191@huawei.com}
		}
\date{}
\maketitle

\begin{abstract}
The low-rank matrix recovery problem often arises in various fields, including signal processing, machine learning, and imaging science. The Riemannian gradient descent (RGD) algorithm has proven to be an efficient algorithm for solving this problem. In this paper, we present a preconditioned Riemannian gradient descent (PRGD) for low-rank matrix recovery. The preconditioner, noted for its simplicity and computational efficiency, is constructed by weighting the $(i,j)$-th entry of the gradient matrix according to the norms of the $i$-th row and the $j$-th column. We establish the theoretical recovery guarantee for PRGD under the restricted isometry property assumption. Experimental results indicate that PRGD can accelerate RGD by up to tenfold in solving low-rank matrix recovery problems such as matrix completion.

\end{abstract}


\begin{AMS}
{15A23 $\cdot$ 15A83 $\cdot$ 68Q25 $\cdot$ 90C26 $\cdot$ 90C53}
\end{AMS}

\section{Introduction}
In this paper, we consider the low-rank matrix recovery problem. Let $\bm{X} \in \mathbb{R}^{n_1 \times n_2}$ be an unknown matrix with $\textmd{rank}(\bm{X}) = r$. The goal of the low-rank matrix recovery problem is to recover the matrix $\bm{X}$ from a set of its linear measurements $\bm{y}\in\mathbb{R}^m$ obtained via
\begin{equation}\label{eq:y=AX}
\bm{y} = \mathcal{A}\bm{X}.
\end{equation}
Here $\mathcal{A}~:~\mathbb{R}^{n_1\times n_2}\mapsto\mathbb{R}^m$ is a linear operator defined as
\begin{equation}\label{eq:defA}
\mathcal{A}\bm{Y} = \begin{bmatrix} \langle \bm{A}_1, \bm{Y} \rangle \\ \langle \bm{A}_2, \bm{Y} \rangle \\   \vdots \\  \langle \bm{A}_m, \bm{Y} \rangle \end{bmatrix},\quad\forall~\bm{Y}\in\mathbb{R}^{n_1\times n_2},
\end{equation}
where $\bm{A}_i\in\mathbb{R}^{n_1\times n_2}$, $i = 1, 2, \dots, m$, are measurement matrices, and $\langle \bm{A}_i, \bm{Y} \rangle = \mathrm{trace}(\bm{A}_i^T\bm{Y})$ represents the standard inner product between two matrices $\bm{A}_i$ and $\bm{Y}$. Typically, the number of measurements is much smaller than the matrix dimension, i.e., $m\ll n_1n_2$.

Low-rank matrix recovery arises frequently in various applications across applied science and engineering, such as matrix completion \cite{GNOT}, phase retrieval \cite{H, MISE}, quantum state tomography \cite{G, KKEG}, and multi-task learning \cite{AFSU, AEP}. Different application scenarios often necessitate distinct types of measurement matrices. For instance, in matrix completion, entries of the unknown matrix are observed, and the measurement matrices are natural bases in the matrix space. In phase retrieval, the measurement matrices are rank-$1$, whereas in quantum state tomography, they are constructed by tensor products of Pauli matrices.

Given that $m\ll n_1n_2$, the linear system \eqref{eq:y=AX} is highly underdetermined. A direct inversion of \eqref{eq:y=AX} generally fails to yield a low-rank matrix solution. However, for an $n_1$-by-$n_2$ rank-$r$ matrix, the number of degrees of freedom is $(n_1 + n_2 - r) r$ \cite{V}, which can be much smaller than $n_1n_2$ when $r$ is small. Under the assumption that the unknown matrix is of low rank, it becomes possible to uniquely solve \eqref{eq:y=AX}. Indeed, under relatively mild assumptions on $\mathcal{A}$, a unique low-rank solution of \eqref{eq:y=AX} exists, enabling efficient low-rank matrix recovery.

In recent years, there has been significant interest in developing algorithms for low-rank matrix recovery. A straightforward approach to finding a low-rank matrix from \eqref{eq:y=AX} is to seek a matrix with the lowest rank, i.e., we solve
$$
\min_{\bm{Z} \in \mathbb{R}^{n_1 \times n_2}} \textmd{rank} (\bm{Z}) \quad \textmd{subject to} \quad \mathcal{A}(\bm{Z}) =\bm{y}.
$$
However, this rank minimization problem is non-convex and NP-hard, making it computationally intractable \cite{CLMW, CRT, D}. To overcome these issues, convex relaxation has emerged as an important method for recovering low-rank matrices. In particular, the nuclear norm $\| \bm{Z} \|_{*}$ of $\bm{Z}$, which is the sum of its singular values, can be substituted for the rank of $\bm{Z}$ in the optimization problem. This leads to the following relaxed convex optimization problem  
\begin{equation}\label{eq:nucmin}
\min_{\bm{Z} \in \mathbb{R}^{n_1 \times n_2}} \| \bm{Z} \|_{*} \quad \textmd{subject to} \quad \mathcal{A} (\bm{Z}) =\bm{y}.
\end{equation}
Under various settings of $\mathcal{A}$, it has been shown that the solution of the nuclear norm minimization \eqref{eq:nucmin} is the ground truth low-rank matrix $\bm{X}$ \cite{CR, CT1, G, R, RFP, CSV, L}. However, the non-smoothness of the nuclear norm function makes explicit gradient-type algorithms slow to converge. Therefore, implicit gradient-type algorithms or proximal algorithms are preferred. It has been shown that the proximal operator of the nuclear norm is the singular value thresholding (SVT) operator \cite{CCS}. Therefore, SVT-based algorithms are popular in low-rank matrix recovery, including the forward-backward splitting algorithm \cite{MGC}, alternating direction method of multipliers \cite{TaoY, LCWM, CHY}, proximal alternating linearized minimization algorithm \cite{BST} and accelerated proximal gradient algorithm \cite{TY}. Other algorithms independent of the SVT can be found in \cite{MZWG, XYWZ}.

To improve computational efficiency, non-convex optimization-based algorithms have been developed. Many non-convex optimization algorithms are designed to solve the following rank minimization problem:
\begin{equation}\label{M-m}
\min_{\bm{Z} \in \mathbb{R}^{n_1 \times n_2}} \frac{1}{2} \| \mathcal{A} \bm{Z} - \bm{y} \|_2^2 \quad \textmd{subject to} \quad \textmd{rank}(\bm{Z}) = r.
\end{equation}
One prominent method is the iterative hard thresholding (IHT) method \cite{BTW, JMD, KC, TW}, which projects the iteration matrix onto the set of all rank-$r$ matrices using the hard thresholding operator at each step. To avoid the computation of large-scale singular value decomposition (SVD), some other non-convex methods adopt the matrix factorization $\bm{Z} = \bm{L}\bm{R}^T$ with $\bm{L} \in \mathbb{R}^{n_1 \times r},~\bm{R} \in \mathbb{R}^{n_2 \times r}$ to parametrize the unknown rank-$r$ matrix. Then \eqref{M-m} is recast into the following optimization problem:
\begin{equation}\label{eq:facmodel}
\min_{\bm{L} \in \mathbb{R}^{n_1 \times r},~\bm{R} \in \mathbb{R}^{n_2 \times r}} \| \mathcal{A} (\bm{L}\bm{R}^T) - \bm{y} \|_2^2.
\end{equation}
Although non-convex, it has been shown  \cite{TBSSR,JNS,ZL,SL, TWei, YPCC, CLS, TWF} that, when appropriately initialized (e.g., by the spectral initialization), simple algorithms such as gradient descent converge to the underlying low-rank matrix for various low-rank matrix recovery problems such as matrix sensing \cite{TBSSR, JNS}, matrix completion \cite{ZL, SL,JNS, TWei}, phase retrieval \cite{CLS, TWF}, and robust principal component analysis (robust PCA) \cite{YPCC}. However, the factorization of a low-rank matrix is redundant and non-unique, leading to over-parametrization of the unknown low-rank matrix. Consequently, regularization terms on the factors $\bm{L}$ and $\bm{R}$ have to be imposed to find the desired factorization \cite{ZL, SL}, and the convergence speed of these methods depends highly on the condition number of the unknown low-rank matrix. To avoid over-parametrization and further improve efficiency, Riemannian manifold optimization algorithms have been adopted. These algorithms represent the unknown low-rank matrix as an element on the quotient Riemannian manifold \cite{HGZ, HGZ1, ZHVZ, LLZ} or an element on the Riemannian manifold embedded in $\mathbb{R}^{n_1\times n_2}$ \cite{WCCL, V, LMC, CW-phase}. Riemannian manifold optimization algorithms lead to more efficient low-rank matrix recovery algorithms, and their theoretical recovery guarantee is also provided despite the non-convexity \cite{JNS, ZL, SL, WCCL, LMC, CW-phase}. 

In this paper, we focus on the Riemannian optimization algorithms on embedded manifolds for low-rank matrix recovery problems. It is well known that all rank-$r$ matrices form a smooth manifold $\mathcal{M}_r$ embedded in the matrix space $\mathbb{R}^{n_1\times n_2}$. The low-rank matrix recovery problem can then be reformulated as the following optimization problem
\begin{equation}\label{Model}
\min_{\bm{Z} \in \mathcal{M}_r} \frac{1}{2} \| \mathcal{A} \bm{Z} - \bm{y} \|_2^2.
\end{equation}
One simple yet efficient algorithm for solving \eqref{Model} is Riemannian gradient descent (RGD), which is the gradient descent with respect to the metric on the embedded Riemannian manifold $\mathcal{M}_r$. Using the canonical metric from the ambient space $\mathbb{R}^{n_1\times n_2}$, RGD has been proposed and studied \cite{CWW, WCCL,V, CW-phase, LMC, TMC} for low-rank matrix recovery problems. By exploiting the special structure of $\mathcal{M}_r$ with the canonical metric, the computational cost of RGD per iteration is significantly reduced --- there is no large-scale SVD computation involved, and the computational cost per step is in the same order as gradient descent for the factorization-based optimization \eqref{eq:facmodel}. Moreover, when initialized by the spectral method, it is theoretically shown that RGD converges linearly to the underlying low-rank matrix, and the convergence factor is independent of the condition number of the unknown matrix \cite{CWW, WCCL, CW-phase, TMC}. All those make RGD one of the most efficient algorithms for low-rank matrix recovery. The performance of RGD can be further improved by Riemannian conjugate gradient algorithms. 

The efficiency of RGD in low-rank matrix recovery is highly dependent on the metric used for the embedded manifold $\mathcal{M}_r$. All previously mentioned RGD have utilized the canonical metric inherited from the ambient space $\mathbb{R}^{n_1\times n_2}$. Thus, a natural question arises:
\begin{center}
\emph{Can we design a more suitable metric on $\mathcal{M}_r$ that leads to faster convergence of RGD?}
\end{center}
We answer this question in the affirmative by constructing a data-driven metric using preconditioning techniques. The resulting algorithm, which is RGD with the data-driven metric, is called preconditioned RGD (PRGD). We demonstrate through numerical experiments that PRGD is much more efficient than RGD with the canonical metric. Specifically, our experiments show that PRGD can be up to ten times faster than RGD for the matrix completion problem. On the theoretical side, we also prove that PRGD converges linearly to the underlying low-rank matrix when initialized by one step of IHT starting from the zero matrix, assuming that the sampling operator $\mathcal{A}$ satisfies the restricted isometry property (RIP). Moreover, the linear convergence factor can be a universal constant.

The main contributions of this paper can be summarized as follows: 
\begin{itemize}
\item \emph{Data-driven metric.}  We propose a simple, adaptive, and easy-to-compute metric from the measurement data $\bm{y}$, the sampling operator $\mathcal{A}$, and $t$-th iteration $\bm{X}_t$. Our metric is constructed by restricting an entrywise weighted metric of $\mathbb{R}^{n_1\times n_2}$ onto the tangent spaces of $\mathcal{M}_r$, and the $(i,j)$-th entry of the weighting matrix is calculated by the $2$-norm of the $i$-th row and the $j$-th column of the gradient matrix at $\bm{X}_t$. This weighting matrix is easy to compute, and the computational cost per iteration does not increase compared to RGD with the canonical metric. We apply RGD under this new metric to obtain a novel low-rank matrix recovery algorithm called the preconditioned Riemannian gradient descent (PRGD) algorithm.
\item \emph{Recovery Guarantee.}  We prove the local convergence of PRGD in terms of the restricted isometry property (RIP) of the sensing operator $\mathcal{A}$. This convergence result ensures that PRGD converges linearly to the underlying low-rank matrix when initialized by one step iterative hard thresholding algorithm starting from zero. When the sensing matrices $\bm{A}_i$, $i=1,\ldots,m,$ in $\mathcal{A}$ are random Gaussian matrices, our result shows that the required RIP constant is satisfied with high probability if the number of measurements $m\sim O(nr^2)$. This sample complexity is optimal for non-convex low-rank matrix recovery algorithms. 
\item \emph{Empirical studies.} We evaluate the performance of the PRGD algorithm on various low-rank matrix recovery problems, especially the matrix completion problem. Our numerical results demonstrate that PRGD reduces the number of iterations without adding too much extra computation at each iteration. As a result, PRGD is much more efficient than RGD with the canonical metric --- the PRGD algorithm can be up to 10 times faster than the RGD algorithm for matrix completion. We also demonstrate that our PRGD algorithm outperforms other algorithms like NIHT.
\end{itemize}
 
The rest of this paper is organized as follows. In \cref{sec:alg},  we describe our proposed PRGD algorithm in detail. In \cref{sec:convergence}, we establish the theoretical recovery guarantee of our algorithm. Specifically, we prove the local convergence of PRGD in terms of the restricted isometry property (RIP) of the sensing operator $\mathcal{A}$. In \cref{sec:numerical}, we demonstrate the efficiency of our PRGD algorithm and its superiority over existing algorithms through numerical experiments. In \cref{sec:conclusion}, we conclude the paper and discuss some future research directions.

Before delving into the main content of the paper, we introduce some notations that are used throughout. We use uppercase and lowercase letters to denote matrices and vectors, respectively, and calligraphic letters for operators. For any $\bm{Z} \in \mathbb{R}^{n_1 \times n_2}$, we denote $\sigma_{\max}(\bm{Z})$ and $\sigma_{\min}(\bm{Z})$ as its maximum and minimum singular values, respectively. The condition number of $\bm{Z}$ is denoted as $\textmd{cond}(\bm{Z}) = \frac{\sigma_{\max}(\bm{Z})}{ \sigma_{\min}(\bm{Z})}$. In this paper, we use $\kappa$ to denote the condition number of the ground truth low-rank matrix $\bm{X}$. We denote the Frobenius norm of $\bm{Z}$ as $\|\bm{Z}\|_{F}$, the spectral norm of $\bm{Z}$ as $\| \bm{Z} \|$, and the transpose of $\bm{Z}$ as $\bm{Z}^{T}$. We define a matrix norm $\| \cdot \|_{\vee}$ as
\begin{equation}\label{eq:vnorm}
 \| \bm{Z} \|_{ \vee } = \max\Big\{ \max_{i} \|\bm{Z}(i,:)\|_2 , \max_{j} \|\bm{Z}(:,j)\|_2  \Big\},
 \end{equation}
i.e., the maximum of the $2$-norms of all rows and columns of the matrix. For any $\bm{Z}, \bm{Y} \in \mathbb{R}^{n_1 \times n_2}$, the inner product of $\bm{Z}$ and $\bm{Y}$ is denoted by $\langle \bm{Z},~\bm{Y} \rangle$. For any vector $\bm{z} \in \mathbb{R}^m$, we denote $\| \bm{z} \|_2$ as the $l_2$ norm of the vector. We denote $\mathcal{I}$ as the identity operator. For the linear operator $\mathcal{A}: \mathbb{R}^{n_1 \times n_2} \to \mathbb{R}^m$, we denote the operator norm of $\mathcal{A}$ as $\| \mathcal{A} \|$. We use $\mathcal{A}^* : \mathbb{R}^{m} \to \mathbb{R}^{n_1 \times n_2}$ to represent the adjoint operator of $\mathcal{A}$ under the standard metric. In the matrix recovery problem, the adjoint operator $\mathcal{A}^*: \mathbb{R}^{m} \to \mathbb{R}^{n_1 \times n_2}$ has the following  form:
$$
\mathcal{A}^*(\bm{p}) = \sum_{i=1}^{m} p_i \bm{A}_i, \quad  \textmd{for~~any} \quad \bm{p} \in \mathbb{R}^m.
$$

\section{Algorithms}
\label{sec:alg}
In this section, we propose our novel preconditioned Riemannian gradient descent (PRGD) algorithm. Some related algorithms are reviewed as well.

\subsection{Riemannian Gradient Descent}
We consider the following general optimization on the Riemannian manifold  $\mathcal{M}_r$ embedded in $\mathbb{R}^{n_1\times n_2}$
\begin{equation}\label{GM-model}
\min_{\bm{Z} \in \mathcal{M}_r} F(\bm{Z}),
\end{equation}
where $F:\mathbb{R}^{n_1\times n_2}\to\mathbb{R}$ is an objective function. The Riemannian gradient descent (RGD) is a popular first-order algorithm for solving Riemannian optimization problems. Starting from an initial guess $\bm{X}_0$, RGD generates a sequence of iterations by the following update rule:
\begin{equation}\label{eq:RGDgeneral}
\bm{X}_{t+1} = \mathcal{R} \left( \bm{X}_t - \alpha_t \nabla_{\mathcal{M}_r} F(\bm{X}_t) \right),\quad t=0,1,2,\ldots,
\end{equation}
where $\mathcal{R}$ is a retraction operator that maps a matrix on the tangent space of $\mathcal{M}_r$ at $\bm{X}_t$ back to a matrix on $\mathcal{M}_r$, $\alpha_t > 0$ is the step size, and $\nabla_{\mathcal{M}_r} F(\bm{X}_t)$ is the gradient of $F$ at $\bm{X}_t$ with respect to the Riemannian metric of the manifold $\mathcal{M}_r$. At each iteration, starting from $\bm{X}_t\in\mathcal{M}_r$, RGD first performs one step of the standard gradient descent in the tangent space of $\mathcal{M}_r$ at $\bm{X}_t$, and then it retracts the iteration matrix from the tangent space back to $\mathcal{M}_r$ by the retraction operator $\mathcal{R}$. RGD has been extensively studied in the literature for solving Riemannian optimization problems.

We recall that the low-rank matrix recovery problem can be formulated as the Riemannian optimization problem \eqref{Model}, where the objective function is
\begin{equation}\label{eq:F}
F(\bm{Z})=\frac{1}{2} \| \mathcal{A} \bm{Z} - \bm{y} \|_2^2.
\end{equation}
To apply RGD to \eqref{Model} for solving the low-rank matrix recovery problem, we need to choose a Riemannian metric on $\mathcal{M}_r$ and a retraction operator $\mathcal{R}$.
\begin{itemize}
\item \emph{Canonical metric.} 
As the embedded manifold $\mathcal{M}_r$ is a subset of the matrix space $\mathbb{R}^{n_1\times n_2}$, it is natural to use the canonical metric on $\mathbb{R}^{n_1\times n_2}$ as the metric on $\mathcal{M}_r$. The canonical inner product between any two matrices $\bm{Y}, \bm{Z}\in\mathbb{R}^{n_1\times n_2}$ is given by $\langle\bm{Y}, \bm{Z}\rangle=\sum_{i,j}Y_{ij}Z_{ij}$, and the induced norm is the Frobenius norm $\|\bm{Y}\|_F=\langle\bm{Y},~\bm{Y}\rangle^{\frac12}$. We restrict this canonical inner product and norm onto tangent spaces of $\mathcal{M}_r$ to obtain the canonical metric on $\mathcal{M}_r$. Let $\mathbb{T}_{\bm{Z}}$ denote the tangent space of $\mathcal{M}_r$ at $\bm{Z}\in\mathcal{M}_r$. In the canonical metric of $\mathcal{M}_r$, the inner product on $\mathbb{T}_{\bm{Z}}$ and its induced norm are exactly the same as in $\mathbb{R}^{n_1\times n_2}$. Specifically, for any two matrices $\bm{Y}, \bm{W}\in\mathbb{T}_{\bm{Z}}$, the inner product is $\langle\bm{Y}, \bm{W}\rangle=\sum_{i,j}Y_{ij}W_{ij}$, and the induced norm is the Frobenius norm. With the canonical metric, the gradient of of $F$ on $\mathcal{M}_r$ can be easily calculated as follows:
$$
\nabla_{\mathcal{M}_r} F(\bm{Z})=\mathcal{P}_{\mathbb{T}_{\bm{Z}}}\big(\nabla F(\bm{Z})\big)=\mathcal{P}_{\mathbb{T}_{\bm{Z}}}\big(\mathcal{A}^*(\mathcal{A}\bm{Z}-\bm{y})\big),
$$ 
where $\mathcal{P}_{\mathbb{T}_{\bm{Z}}}$ is the orthogonal projector in $\mathbb{R}^{n_1\times n_2}$, and $\nabla F(\bm{Z})$ is the gradient of $F$ viewed as a function $\mathbb{R}^{n_1\times n_2}\to\mathbb{R}$. 
\item \emph{Retraction by truncated SVD.} The retraction operator $\mathcal{R}$ is used to retract a matrix on a tangent space $\mathbb{T}_{\bm{Z}}$ back to the manifold $\mathcal{M}_r$, such that the error between $\mathcal{R}(\bm{Y})$ and $\bm{Y}$ is controlled for $\bm{Y}\in\mathbb{T}_{\bm{Z}}$ near $\bm{Z}$. A common choice of the retraction operator is the projection operator onto $\mathcal{M}_r$, also known as the best rank-$r$ approximation operator or the truncated SVD operator. Specifically, for any $\bm{Y}\in\mathbb{T}_{\bm{Z}}$ with an SVD $\bm{Y}=\sum_{i}\sigma_i\bm{u}_i\bm{v}_i^T$ ($\sigma_1\geq\sigma_2\geq\ldots$), we choose
$$
\mathcal{R}(\bm{Y})=\mathcal{H}_r(\bm{Y}):=\sum_{i=1}^{r}\sigma_i\bm{u}_i\bm{v}_i^T.
$$
Note that there are other options for the retraction operator.
\end{itemize}
By choosing the canonical metric and the truncated SVD retraction, we can obtain an efficient RGD algorithm for low-rank matrix recovery as in the following:
\begin{equation}\label{eq:RGD}
\bm{X}_{t+1} = \mathcal{H}_r \Big( \bm{X}_t - \alpha_t \mathcal{P}_{\mathbb{T}_{\bm{X}_t}}\big(\mathcal{A}^*(\mathcal{A}\bm{X}_t-\bm{y})\big) \Big),\quad t=0,1,2,\ldots.
\end{equation}

The computation of \eqref{eq:RGD} can be done very efficiently by exploiting the structure of the tangent spaces of $\mathcal{M}_r$. Besides the application of $\mathcal{A}$ and $\mathcal{A}^*$, the most time-consuming operations involved in \eqref{eq:RGD} are the evaluations of $\mathcal{P}_{\mathbb{T}_{\bm{X}_t}}$ and $\mathcal{H}_r$. Their computation is described below. 
\begin{itemize}
\item \emph{Computation of the projection $\mathcal{P}_{\mathbb{T}}$.} Recall that the tangent space $\mathbb{T}_{\bm{Z}}$ at $\bm{Z}\in\mathcal{M}_r$ with a compact SVD $\bm{Z}=\bm{U}_{\bm{Z}}\bm{\Sigma}_{\bm{Z}}\bm{V}_{\bm{Z}}^T$ can be explicitly expressed as 
\begin{equation}\label{TS}
\mathbb{T}_{\bm{Z}} = \big\{\bm{U}_{\bm{Z}} \bm{Q}^T + \bm{P} \bm{V}_{\bm{Z}}^T~|~\bm{P} \in \mathbb{R}^{n_1 \times r},~\bm{Q} \in \mathbb{R}^{n_2 \times r} \big\}.
\end{equation}
Therefore, the orthogonal projector $\mathcal{P}_{\mathbb{T}_{\bm{Z}}}$ has a closed form given by
$$
\mathcal{P}_{\mathbb{T}_{\bm{Z}}}\bm{Y}=\bm{U}_{\bm{Z}}\bm{U}_{\bm{Z}}^T\bm{Y}+\bm{Y}\bm{V}_{\bm{Z}}\bm{V}_{\bm{Z}}^T-\bm{U}_{\bm{Z}}\bm{U}_{\bm{Z}}^T\bm{Y}\bm{V}_{\bm{Z}}\bm{V}_{\bm{Z}}^T.
$$
Thus, when a compact SVD of $\bm{X}_t$ is available, the computation of $\mathcal{P}_{\mathbb{T}_{\bm{X}_t}}$ in \eqref{eq:RGD} costs only $O(r)$ matrix-vector products.
\item \emph{Computation of the truncated SVD $\mathcal{H}_r$.} Note that $\mathcal{H}_r$ in \eqref{eq:RGD} is only applied to a matrix $\bm{Y}\in\mathbb{T}_{\bm{X}_t}$. From \eqref{TS}, we observe that a matrix in a tangent space has a rank at most $2r$. Therefore, we can write $\bm{Y}=\bm{L}\bm{R}^T$ for some $\bm{L}\in\mathbb{R}^{n_1\times 2r}$ and $\bm{R}\in\mathbb{R}^{n_2\times 2r}$. The best rank-$r$ approximation to the rank-$2r$ matrix $\bm{Y}$ can be computed efficiently as follows:
\begin{itemize}
\item We compute the QR decompositions $\bm{L}=\bm{Q}_1\bm{R}_1$ and $\bm{R}=\bm{Q}_2\bm{R}_2$. 
\item Then, $\bm{Y}=\bm{Q}_1\bm{R}_1\bm{R}_2^T\bm{Q}_2^T$. Since $\bm{Q}_1$ and $\bm{Q}_2$ are orthogonal, 
$$
\mathcal{H}_r(\bm{Y})=\mathcal{H}_r(\bm{Q}_1\bm{R}_1\bm{R}_2^T\bm{Q}_2^T)=\bm{Q}_1\cdot \mathcal{H}_r(\bm{R}_1\bm{R}_2^T)\cdot\bm{Q}_2^T.
$$ 
To obtain $\mathcal{H}_r(\bm{Y})$, we only need to compute an SVD of the $2r\times 2r$ matrix $\bm{R}_1\bm{R}_2^T$.
\end{itemize}
Therefore, the computation of  $\mathcal{H}_r$ in \eqref{eq:RGD} can be efficiently done by two QR decompositions of tall matrices of width $r$, one SVD of a $2r\times 2r$ matrix, and a few matrix-vector products.
\end{itemize}
Therefore, the computational cost of each iteration of \eqref{eq:RGD} is very low. Theoretical analysis \cite{WCCL,LMC,CW-phase} has shown that the RGD \eqref{eq:RGD} converges linearly to the underlying low-rank matrix when initialized by one step of IHT starting from $\bm{0}$, under suitable assumptions on $\mathcal{A}$. Moreover, the contraction factor in the linear convergence is universal and independent of the unknown low-rank solution, which means that the RGD \eqref{eq:RGD} can achieve a very accurate solution within very few iterations. Altogether, the RGD \eqref{eq:RGD} is one of the most efficient algorithms for low-rank matrix recovery.

The  RGD \eqref{eq:RGD} and its variants have been proposed and studied for various low-rank matrix recovery problems, such as general matrix sensing \cite{WCCL}, matrix completion \cite{LMC,V}, and phase retrieval \cite{CW-phase}. It has also been extended to some other related problems including low-rank tensor recovery \cite{CLX, KSV}, spectral compressed sensing \cite{CWW}, and robust principal component analysis \cite{ZY,CCW}. 

\subsection{Preconditioned Riemannian Gradient Descent}
The performance of the gradient descent for optimization on Hilbert spaces relies on the metric of the underlying Hilbert space. A common technique to improve the performance of iterative algorithms is to modify the metric. In solving symmetric positive definite linear systems, we often formulate them as quadratic optimization problems in $\mathbb{R}^n$, and the steepest descent algorithm is the gradient descent with an exact line search. To accelerate the steepest descent, we change the metric of $\mathbb{R}^n$ weighted by an approximation of the coefficient matrix, resulting in the preconditioned steepest descent. This preconditioning technique can be combined with more advanced algorithms to obtain state-of-the-art linear system solvers like preconditioned conjugate gradient (PCG) method. In nonlinear unconstrained optimization, we often accelerate the gradient descent algorithm by changing the metric (e.g., weighted by an approximate Hessian), giving rise to preconditioned gradient methods including Newton-type algorithms \cite{LO, N}.

We employ a similar idea to accelerate the RGD \eqref{eq:RGD} by altering the metric of the Riemannian manifold $\mathcal{M}_r$ to obtain a preconditioned RGD (PRGD). Since the objective function $F$ in \eqref{eq:F} is a least-squares fitting to the linear system \eqref{eq:y=AX}, and RGD linearizes $F$ on the tangent space $\mathbb{T}_{\bm{X}_t}$, the convergence speed depends on the condition number of the sensing operator $\mathcal{A}$ restricted on $\mathbb{T}_{\bm{X}_t}$. Let $\mathfrak{g}$ be the metric used on $\mathcal{M}_r$ (i.e., $\langle \cdot, \cdot \rangle_{\mathfrak{g}}$ is the inner product on tangent spaces of $\mathcal{M}_r$ and $\|\cdot\|_{\mathfrak{g}}$ is the induced norm). Then, the condition number of $\mathcal{A}$ restricted on the tangent space $\mathbb{T}_{\bm{X}_t}$ is $\mathrm{cond}_{\mathfrak{g},\mathbb{T}_{\bm{X}_t}}(\mathcal{A})=\frac{C_{l,\mathfrak{g}}}{C_{u,\mathfrak{g}}}$, where
$$
C_{l,\mathfrak{g}} \| \bm{Z} \|_{\mathfrak{g}}^2 \leq \| \mathcal{A} \bm{Z} \|_{2}^2 \leq C_{u,\mathfrak{g}} \| \bm{Z} \|_{\mathfrak{g}}^2,  \qquad  \forall~\bm{Z} \in \mathbb{T}_{\bm{X}_t}.
$$
The smaller $\mathrm{cond}_{\mathfrak{g},\mathbb{T}_{\bm{X}_t}}(\mathcal{A})$, the faster convergence. Our goal is to find a metric $\mathfrak{g}$ that minimizes $\mathrm{cond}_{\mathfrak{g},\mathbb{T}_{\bm{X}_t}}(\mathcal{A})$.

Using the canonical metric (denoted by $\mathfrak{g}_0$), for many low-rank matrix recovery problems, we can prove that
\begin{equation}\label{eq:RIPXt}
(1-\delta) \| \bm{Z} \|_{F}^2 \leq \| \mathcal{A} \bm{Z} \|_{2}^2 \leq (1+\delta) \| \bm{Z} \|_{F}^2,  \qquad  \forall~\bm{Z} \in \mathbb{T}_{\bm{X}_t}
\end{equation}
for some $\delta\in(0,1)$. In the matrix sensing and quantum state tomography problems, it can be shown that $\mathcal{A}$ satisfies the restricted isometry property (RIP), meaning that it is nearly isometric when restricted onto the set of all low-rank matrices. Since all matrices on any tangent spaces of $\mathcal{M}_r$ are of rank at most $2r$, \eqref{eq:RIPXt} is satisfied. In matrix completion and phase retrieval, when $\bm{X}_t$ is sufficiently close $\bm{X}$, it has been proved that \eqref{eq:RIPXt} holds with high probability when $\mathcal{A}$ satisfies some probabilistic models. The inequality \eqref{eq:RIPXt} is often a crucial condition in the proof of recovery guarantee of many convex and non-convex algorithms, including RGD, factorization-based gradient descent, and landscape analysis. Moreover, in many applications such as matrix completion, quantum state tomography, and matrix sensing, the constant $\delta$ in \eqref{eq:RIPXt} can be arbitrarily small if the number of measurements $m$ is sufficiently large. Usually, the dependency of $m$ and $\delta$ is $m\sim O(\delta^{-2})$. When \eqref{eq:RIPXt} is satisfied with a small $\delta$, the restricted condition number of $\mathcal{A}$ under the canonical metric $\mathfrak{g}_0$ is $\mathrm{cond}_{\mathfrak{g}_0,\mathbb{T}_{\bm{X}_t}}(\mathcal{A})=\frac{1+\delta}{1-\delta}$, which is small. Therefore, when $m$ is sufficiently large to make small $\delta$ in \eqref{eq:RIPXt}, the RGD \eqref{eq:RGD} under the canonical metric is one of the most efficient algorithms for low-rank matrix recovery problems.

However, in some other applications such as phase retrieval, the constant $\delta$ in \eqref{eq:RIPXt} cannot be very small even when $m$ is sufficiently large. Additionally, in practical scenarios, we may not have sufficiently large number of measurements $m$ to make $\delta$ in \eqref{eq:RIPXt} very small, resulting in a large restricted condition number $\mathrm{cond}_{\mathfrak{g}_0,\mathbb{T}_{\bm{X}_t}}(\mathcal{A})$. Moreover, \eqref{eq:RIPXt} is established under certain probabilistic models of $\mathcal{A}$ and holds with certain probabilities. In reality, we have only access to a particular realization of $\mathcal{A}$, and even when $m$ is not small, there is still a chance that the available realization $\mathcal{A}$ satisfies \eqref{eq:RIPXt} with a large $\delta$. In all those cases, the RGD \eqref{eq:RGD} may converge slowly. To address this issue, we propose to use a non-canonical metric $\mathfrak{g}$ that results in a small restricted condition number $\mathrm{cond}_{\mathfrak{g},\mathbb{T}_{\bm{X}_t}}(\mathcal{A})$ under $\mathfrak{g}$. This approach can potentially improve the performance of the RGD algorithm in scenarios where the constant $\delta$ in \eqref{eq:RIPXt} is large, and the RIP is not well satisfied.

\subsubsection{Data-Driven Metric Construction.}
Our metric $\mathfrak{g}$ is constructed from the measurement operator $\mathcal{A}$, the measurement vector $\bm{y}$, and the iteration matrices, making it a \emph{data-driven metric} that adapts to the measurement and algorithmic data. To simplify notation, we use $\mathbb{T}_{t}$ to denote the tangent space $\mathbb{T}_{\bm{X}_t}$ for all $t\in\mathbb{N}$. Since only metrics on $\mathbb{T}_{t}$ affect the RGD iteration \eqref{eq:RGDgeneral}, we only need to construct $\mathfrak{g}$ on $\mathbb{T}_{t}$. 

The construction of our metric $\mathfrak{g}$ on each tangent space $\mathbb{T}_{t}$ involves two steps. In the first step, we construct a metric on the ambient space $\mathbb{R}^{n_1\times n_2}$ using preconditioning techniques for the linear space optimization problem $\min_{\bm{Z}\in\mathbb{R}^{n_1\times n_2}}F(\bm{Z})$. In the second step, we restrict the metric of $\mathbb{R}^{n_1\times n_2}$ onto the tangent space $\mathbb{T}_t$ to obtain the metric $\mathfrak{g}$. The two steps are described in detail below. 

\begin{itemize}
\item \emph{Constructing a metric on $\mathbb{R}^{n_1\times n_2}$ via preconditioning.} To construct our metric $\mathfrak{g}$, we first construct a metric on the ambient space $\mathbb{R}^{n_1\times n_2}$ using preconditioning techniques for linear space matrix optimizations. Without the low-rank constraint, the optimization problem in \eqref{GM-model} becomes a typical smooth optimization in the linear space $\mathbb{R}^{n_1\times n_2}$:
$$
\min_{\bm{Z}\in\mathbb{R}^{n_1\times n_2}}F(\bm{Z}), \qquad F(\bm{Z})=\frac12\|\mathcal{A}\bm{Z}-\bm{y}\|_2^2.
$$ 
To solve this problem, preconditioned gradient methods have emerged as a powerful class of algorithms, including Newton and quasi-Newton methods \cite{LO, N}. However, it is too expensive to construct the preconditioner in Newton and quasi-Newton methods. To overcome this, efficient and effective preconditioned algorithms have also been proposed, driven by modern machine learning, especially the training of deep neural networks.

Let $\bm{G}_t$ be the gradient of $F$ in $\mathbb{R}^{n_1\times n_2}$, i.e.,
$$
\bm{G}_t = \mathcal{A}^*(\mathcal{A}(\bm{X}_t) - \bm{y}).
$$
AdaGrad  \cite{DHS} constructs its preconditioner by accumulating outer products of historical gradients. It has a heavy weight on historical gradients, discouraging the algorithm from updating along historical directions again. This makes AdaGrad a highly efficient method in practice. However, the computation of the outer product of two matrices is still expensive, making the AdaGrad preconditioner unsuitable for our matrix optimization. To overcome this, the Shampoo preconditioner \cite{GKS} has been proposed. Shampoo approximates the AdaGrad preconditioner using a tensor product. Specifically, it defines
\begin{equation}\label{eq:ShamPrecon}
\bm{L}_t = \epsilon \bm{I}_{n_1} +  \sum_{i=1}^{t}\bm{G}_{i} \bm{G}_{i}^T,\quad\mbox{and}\quad
\bm{R}_t = \epsilon \bm{I}_{n_2} + \sum_{i=1}^{t}\bm{G}_{i}\bm{G}_{i}^T.
\end{equation}
The preconditioner is then $\bm{R}_t^{\frac14}\otimes\bm{L}_{t}^{\frac14}$, where the rows and columns are weighted separately by $\bm{L}_{t}^{\frac14}$ and $\bm{R}_{t}^{\frac14}$ in the weighted inner product. Additionally, the rows (resp. columns) are preconditioned by the accumulation of outer products of rows (resp. columns) of historical gradients. 

Shampoo is still too expensive for low-rank matrix recovery, especially for large-scale problems. Indeed, in matrix completion, the matrix is so large that a moderate computer cannot store one full dense matrix; however, Shampoo needs the full dense matrices $\bm{L}_t$ and $\bm{R}_t$. Therefore, to further reduce the temporal and spatial complexity, we further simplify the Shampoo preconditioner. First, we only use the gradient at $\bm{X}_t$ instead of all historical gradients involved in Shampoo. Second, we use only the diagonal part of the preconditioner matrix in Shampoo to reduce computation and memory. Specifically, we replace $\bm{L}_t$ and $\bm{R}_t$ in Shampoo by their diagonal approximation
\begin{equation}\label{eq:LtRt}
\bm{L}_t = \epsilon \bm{I}_{n_1} +  \textmd{diag}(\bm{G}_{t} \bm{G}_{t}^T)\quad\mbox{and}\quad
\bm{R}_t = \epsilon \bm{I}_{n_2} + \textmd{diag}(\bm{G}_{t}^T\bm{G}_t).
\end{equation}
Then, the rows and columns are weighted separately by $\bm{L}_{t}^{\frac14}$ and $\bm{R}_{t}^{\frac14}$ in the weighted inner product, resulting in our metric in $\mathbb{R}^{n_1\times n_2}$ --- for any $\bm{Z}, \bm{Y} \in \mathbb{R}^{n_1 \times n_2}$, 
\begin{equation}\label{eq:weightedinnerprod}
\big\langle \bm{Z}, \bm{Y} \big\rangle_{\mathcal{W}_t} = \big\langle \mathcal{W}_t \bm{Z}, \bm{Y} \big\rangle = \big\langle \bm{L}_t^{\frac14} \bm{Z} \bm{R}_t^{\frac14}, \bm{Y} \big\rangle.
\end{equation}

Obviously, the weighting operator $\mathcal{W}_t$ in our metric is applied entrywise. It is noteworthy that each diagonal element of the diagonal matrix $\diag (\bm{G}_t\bm{G}_t^T)$ is equal to $\| \bm{G}_t(i, :)\|_2^2$, where $\bm{G}_t(i, : )$ represent the $i$-th row of the gradient $\bm{G}_t$. Similarly, the $j$-th diagonal element of $\diag (\bm{G}_t^T\bm{G}_t)$ is the square of the norm of the $j$-th column for the gradient $\bm{G}_t$. Consequently, $\mathcal{W}_t$ in our metric rebalances the $(i,j)$-th entry of a matrix based on the 2-norms of the $i$-th row and $j$-th column of the gradient $\bm{G}_t$.

\item \emph{Constructing the metric on $\mathbb{T}_t$ by restriction.} We simply restrict the previously constructed metric in $\mathbb{R}^{n_1 \times n_2}$ to the tangent space $\mathbb{T}_t$.  For any $\bm{Z}, \bm{Y} \in \mathbb{T}_t$, the inner product can be expressed as follows:
\begin{equation}\label{eq:weightedinnerprod}
\big\langle \bm{Z}, \bm{Y} \big\rangle_{\mathcal{W}_t} = \big\langle \mathcal{W}_t \bm{Z}, \bm{Y} \big\rangle = \big\langle \bm{L}_t^{\frac14} \bm{Z} \bm{R}_t^{\frac14}, \bm{Y} \big\rangle,
\end{equation}
where $\bm{L}_t$ and $\bm{R}_t$ are defined in \eqref{eq:LtRt}. The norm induced by the inner product  $\langle \cdot, \cdot \rangle_{\mathcal{W}_t}$ is denoted by $\| \cdot \|_{\mathcal{W}_t}$. 

\end{itemize}

\subsubsection{Riemannian Gradient Descent under the new metric.}
To perform RGD \eqref{eq:RGDgeneral} using the new metric, we need to compute the Riemannian gradient $\nabla_{\mathcal{M}_r}F(\bm{X}_t)$ under the new metric. For this purpose, let $\bm{X}_t(s)$, $s\in\mathbb{R}$, be a smooth curve on $\mathcal{M}_r$ with $\bm{X}_t(0)=\bm{X}_t$. Let $\widetilde{\mathcal{P}}_{\mathbb{T}_t}$ be the orthogonal projector onto $\mathbb{T}_t$ in $\mathbb{R}^{n_1\times n_2}$ under the metric \eqref{eq:weightedinnerprod}. Differentiating $F(\bm{X}_t(s))$ with respect to $s$ yields
$$
\frac{d}{ds}F(\bm{X}_t(s))\Big|_{s=0}=\langle\dot{\bm{X}_t}(0), \bm{G}_t\rangle= \langle\dot{\bm{X}_t}(0), \mathcal{W}_t^{-1}\bm{G}_t\rangle_{{\mathcal{W}_t}}=\langle\dot{\bm{X}_t}(0), \widetilde{\mathcal{P}}_{\mathbb{T}_t}\mathcal{W}_t^{-1}\bm{G}_t\rangle_{{\mathcal{W}_t}},
$$
where in the last term, $\dot{\bm{X}_t}(0)$ and $ \widetilde{\mathcal{P}}_{\mathbb{T}_t}\mathcal{W}_t^{-1}\bm{G}_t$ are both in $\mathbb{T}_t$, and the inner product is the one we constructed in \eqref{eq:weightedinnerprod}. Therefore, the Riemannian gradient under the new metric is given by
$$
\nabla_{\mathcal{M}_r}F(\bm{X}_t)= \widetilde{\mathcal{P}}_{\mathbb{T}_t}\mathcal{W}_t^{-1}\bm{G}_t= \widetilde{\mathcal{P}}_{\mathbb{T}_t} (\bm{L}_t^{-\frac14} \bm{G}_t \bm{R}_t^{-\frac14} ).
$$
We will still use the truncated SVD as the retraction operator. Combining all of the above, we obtain our proposed preconditioned RGD (PRGD) for low-rank matrix recovery. The complete PRGD algorithm is summarized in \cref{PRGD}.

\begin{center}
\begin{minipage}{0.95\linewidth}
	
\begin{algorithm}[H]
\caption{Preconditioned Riemannian Gradient Descent (PRGD).}
\label{PRGD}
\begin{algorithmic}
\STATE{ Initialize $\bm{X}_0 = \mathcal{H}_r(\mathcal{A}^* \bm{y})$. 
}
\vspace{0.2cm}
\STATE{ {\bf While} stopping criteria is not satisfied {\bf do}\\
\vspace{0.2cm}
{ \qquad\quad~ {\bf Compute gradient:}} $\bm{G}_t = \mathcal{A}^*(\mathcal{A}(\bm{X}_t) - \bm{y})$.\\
\vspace{0.2cm}
{ \qquad\quad~ {\bf Choose $\epsilon_t$ and $\alpha_t$.}}\\
\vspace{0.2cm}
{ \qquad\quad~ {\bf Update preconditioners:}}
\vspace{0.2cm}
$$
\bm{L}_{t} = \epsilon_t \bm{I}_{n_1} +  \textmd{diag}(\bm{G}_{t} \bm{G}_{t}^T).
$$
$$
\bm{R}_t = \epsilon_t \bm{I}_{n_2} + \textmd{diag}(\bm{G}_{t}^T\bm{G}_t).
$$
\vspace{0.2cm}
{\qquad\quad{\bf Update intermediate parameters:}}
$$
 \bm{W}_{t} = \bm{X}_t - \alpha_t  \widetilde{\mathcal{P}}_{\mathbb{T}_t} (\bm{L}_t^{-\frac14} \bm{G}_t \bm{R}_t^{-\frac14} ).
$$
{ \qquad\quad~{\bf  Update X:}}
$$
\bm{X}_{t+1} = \mathcal{H}_{r}(\bm{W}_{t}).
$$
}
\STATE{ Output $\bm{X}_{t+1}$ when the stopping criteria is met.}
\end{algorithmic}
\end{algorithm}
\end{minipage}
\end{center}

\subsubsection{Computation in PRGD.}
\label{sec:comPRGD}
Similar to RGD \eqref{eq:RGD}, the main computations of PRGD, as described in \cref{PRGD}, involve the evaluation of the projection operator $\widetilde{\mathcal{P}}_{\mathbb{T}_t}$ and the retraction operator $\mathcal{H}_{r}$, as well as the calculation of $\bm{G}_t$. By utilizing the structure of the Riemannian manifold $\mathcal{M}_r$, both operators $\widetilde{\mathcal{P}}_{\mathbb{T}_t}$ and $\mathcal{H}_{r}$ can be evaluated efficiently.

\begin{itemize}
\item \emph{Computation of the projection $\widetilde{\mathcal{P}}_{\mathbb{T}_t}$.}  Recall that $\widetilde{\mathcal{P}}_{\mathbb{T}_t}$ denotes the projection onto the tangent space $\mathbb{T}_t$ with respect to the inner product $\langle \cdot, \cdot \rangle_{\mathcal{W}_t}$ of $\mathbb{R}^{n_1\times n_2}$. Let $\bm{X}_t = \bm{U}_t \bm{\Sigma}_t \bm{V}_t^T$ be a compact SVD with $\bm{U}_t = [\bm{u}_t^1, \bm{u}_t^2, \dots, \bm{u}_t^r] \in \mathbb{R}^{n_1 \times r}$ and $\bm{V}_t =[ \bm{v}_t^1, \bm{v}_t^2, \dots, \bm{v}_t^r] \in \mathbb{R}^{n_2 \times r}$. We define the weighted inner products 
\begin{equation}\label{eq:weightedinnerprodRn1}
\langle \bm{x}, \bm{y} \rangle_{\bm{L}_t^{\frac14}} = \langle \bm{L}_t^{\frac14}\bm{x}, \bm{y} \rangle ~~~ \textmd{in}  ~~~ \mathbb{R}^{n_1}  ~~~~ \textmd{and} ~~~~ \langle \bm{x}, \bm{y} \rangle_{\bm{R}_t^{\frac14}} = \langle \bm{R}_t^{\frac14}\bm{x}, \bm{y} \rangle ~~~ \textmd{in} ~~~ \mathbb{R}^{n_2},
\end{equation}
and orthogonalize $\big\{\bm{u}_t^i\big\}_{i =1}^r$ and $\big\{\bm{v}_t^i\big\}_{i =1}^r$ under these inner products to obtain orthonormal bases 
 \begin{equation}\nonumber
 \begin{aligned}
\widetilde{\bm{U}}_t = \bm{U}_t \big( \bm{U}_t^T \bm{L}_t^{\frac14} \bm{U}_t \big)^{-\frac12} := \left[\tilde{\bm{u}}_t^1, \tilde{\bm{u}}_t^2, \dots, \tilde{\bm{u}}_t^r \right] \in \mathbb{R}^{n_1 \times r},\\
\widetilde{\bm{V}}_t = \bm{V}_t \big( \bm{V}_t^T \bm{R}_t^{\frac14} \bm{V}_t \big)^{-\frac12} := \left[\tilde{\bm{v}}_t^1, \tilde{\bm{v}}_t^2, \dots, \tilde{\bm{v}}_t^r \right] \in \mathbb{R}^{n_2 \times r}.
\end{aligned}
 \end{equation}
We further extend $\widetilde{\bm{U}}_t$ and $\widetilde{\bm{V}}_t$ to full orthonormal bases of $(\mathbb{R}^{n_1},\langle\cdot,\cdot\rangle_{\bm{L}_t^{\frac14}})$ and $(\mathbb{R}^{n_2},\langle\cdot,\cdot\rangle_{\bm{R}_t^{\frac14}})$ respectively. Then, an orthonormal basis of $\mathbb{T}_t$ with respect to $\langle \cdot, \cdot \rangle_{\mathcal{W}_t}$ is given by
$\big\{\tilde{\bm{u}}_t^i (\tilde{\bm{v}}_t^j)^T \big\}_{\min\{i, j\} \leq r}$.
Using this orthonormal basis, we can calculate the projection $\widetilde{\mathcal{P}}_{\mathbb{T}_t}$ onto $\mathbb{T}_t$ directly. Specifically, for any $\bm{Z} \in \mathbb{R}^{n_1 \times n_2}$, we have
\begin{equation}\label{proj}
\begin{split}
\widetilde{\mathcal{P}}_{\mathbb{T}_t} \bm{Z} 
&=\sum_{(i,j):\min\{i,j\}\leq r}\langle\bm{Z},\tilde{\bm{u}}_t^i (\tilde{\bm{v}}_t^j)^T\rangle_{\mathcal{W}_t}\cdot\tilde{\bm{u}}_t^i (\tilde{\bm{v}}_t^j)^T =  \sum_{(i,j):\min\{i,j\}\leq r}\langle\bm{Z},\bm{L}_t^{\frac14}\tilde{\bm{u}}_t^i (\tilde{\bm{v}}_t^j)^T\bm{R}_t^{\frac14}\rangle\cdot\tilde{\bm{u}}_t^i (\tilde{\bm{v}}_t^j)^T\cr
&=\sum_{(i,j):\min\{i,j\}\leq r}(\tilde{\bm{u}}_t^i)^T\bm{L}_t^{\frac14}\bm{Z}\bm{R}_t^{\frac14}\tilde{\bm{v}}_t^j\cdot\tilde{\bm{u}}_t^i (\tilde{\bm{v}}_t^j)^T\cr
&= \widetilde{\bm{U}}_t \widetilde{\bm{U}}_t^T \bm{L}_t^{\frac14} \bm{Z} + \bm{Z} \bm{R}_t^{\frac14} \widetilde{\bm{V}}_t \widetilde{\bm{V}}_t^T - \widetilde{\bm{U}}_t \widetilde{\bm{U}}_t \bm{L}_t^{\frac14} \bm{Z} \bm{R}_t^{\frac14} \widetilde{\bm{V}}_t\widetilde{\bm{V}}_t^T.
\end{split}
\end{equation}
\item \emph{Computation of the truncated SVD $\mathcal{H}_r$.}
\label{2.1.2}
In Algorithm \ref{PRGD}, it is clear that $\bm{W}_t\in\mathbb{T}_t$. Using the form of tangent space $\mathbb{T}_t$ in \eqref{TS}, we can see that the rank of $\bm{W}_t$ is at most $2r$. This allows us to efficiently compute $\mathcal{H}_r(\bm{W}_t)$ in Algorithm \ref{PRGD} without resorting to a large-scale SVD. To see this, we can directly compute $\bm{W}_t$ as follows: 
\begin{equation*}
\begin{aligned}
\bm{W}_t  &= \bm{X}_t + \widetilde{\mathcal{P}}_{\mathbb{T}_t} (\bm{Z}_t)\\
&= \bm{U}_t \bm{\Sigma}_t \bm{V}_t^T +  \widetilde{\bm{U}}_t \widetilde{\bm{U}}^T_t \bm{L}_t^{\frac14} \bm{Z}_t + \bm{Z}_t \bm{R}_t^{\frac14} \widetilde{\bm{V}}_t \widetilde{\bm{V}}^T_t - \widetilde{\bm{U}}_t \widetilde{\bm{U}}^T_t \bm{L}_t^{\frac14} \bm{Z}_t \bm{R}_t^{\frac14} \widetilde{\bm{V}}_t \widetilde{\bm{V}}^T_t\\
&= \bm{U}_t \bm{\Sigma}_t \bm{V}_t^T + \bm{U}_t \big(\underbrace{\bm{U}_t^T \bm{L}_t^{\frac14} \bm{U}_t}_{:=\bm{M}_1}\big)^{-1} \bm{U}_t^T \bm{L}_t^{\frac14} \bm{Z}_t + \bm{Z}_t \bm{R}_t^{\frac14} \bm{V}_t \big(\underbrace{\bm{V}_t^T \bm{R}_t^{\frac14} \bm{V}}_{:=\bm{M}_2} \big)^{-1} \bm{V}_t^T \\
&\quad - \bm{U}_t \big(\bm{U}_t^T \bm{L}_t^{\frac14} \bm{U}_t \big)^{-1} \bm{U}_t^T \bm{L}_t^{\frac14} \bm{Z}_t  \bm{R}_t^{\frac14} \bm{V}_t \big(\bm{V}_t^T \bm{R}_t^{\frac14} \bm{V}_t \big)^{-1} \bm{V}_t^T\\
&=\bm{U}_t \big(\underbrace{\bm{\Sigma}_t + \bm{M}_1^{-1} \bm{U}_t^T\bm{L}_t^{\frac14} \bm{Z}_t \bm{V}_t + \big( \bm{U}_t^T  - \bm{M}_1^{-1} \bm{U}_t^T\bm{L}_t^{\frac14} \big) \bm{Z}_t \bm{R}_t^{\frac14} \bm{V}_t \bm{M}_2^{-1}}_{:=\bm{K}_0}  \big) \bm{V}_t^T\\
&\quad +\bm{U}_t \underbrace{\bm{M}_1^{-1} \bm{U}_t^T \bm{L}_t^{\frac14} \bm{Z}_t \left( \bm{I} - \bm{V}_t\bm{V}_t^T \right)}_{:=\bm{Y}_1^T}  + \underbrace{ \left(\bm{I} - \bm{U}_t  \bm{U}_t^T \right) \bm{Z}_t \bm{R}_t^{\frac14} \bm{V}_t \bm{M}_2^{-1}}_{:=\bm{Y}_2} \bm{V}_t^T\\
\end{aligned}
\end{equation*}
We can then perform a QR factorization of the matrices $\bm{Y}_1$ and $\bm{Y}_2$ to obtain $\bm{Y}_1 = \bm{Q}_1\bm{K}_1$ and $\bm{Y}_2 = \bm{Q}_2\bm{K}_2$, respectively. It is obvious that $\bm{U}_t^T\bm{Q}_2 =\bm{0}$ and $\bm{V}_t^T\bm{Q}_1=\bm{0}$. Using the resulting factorizations, we can rewrite $\bm{W}_t$ as follows: 
\begin{equation*}
\begin{aligned}
\bm{W}_t &= \bm{U}_t\bm{K}_0 \bm{V}_t^T + \bm{U}_t \bm{K}_1^T\bm{Q}_1^T + \bm{Q}_2 \bm{K}_2 \bm{V}_t^T=\big[\bm{U}_t \quad \bm{Q}_2 \big] \underbrace{\begin{bmatrix}\bm{K}_0 & \bm{K}_1^T \\ \bm{K}_2 & \bm{0} \end{bmatrix}}_{\bm{M}_t}  \begin{bmatrix} \bm{V}_t^T \\ \bm{Q}_1^T \end{bmatrix},
\end{aligned}
\end{equation*}
where $\bm{M}_t$ is a $2r \times 2r$ matrix. Since $\big[\bm{U}_t \quad \bm{Q}_2 \big]$ and $\big[\bm{V}_t \quad \bm{Q}_1 \big]$ are both orthogonal matrices, the SVD of $\bm{W}_t$ can be obtained from the SVD of $\bm{M}_t$, which can be computed using $O(r^3)$ flops. Therefore, the truncated SVD $\mathcal{H}_r(\bm{W}_t)$ can be done efficiently without any large-scale SVD.
\end{itemize}
Compared to the computation of RGD \eqref{eq:RGD} with the canonical metric, the proposed PRGD involves the calculation of additional matrices $\bm{L}_t$, $\bm{R}_t$, $\bm{K}_1^{-1}$, and $\bm{K}_2^{-1}$, which are either diagonal or in small size. Thus, PRGD incurs only minimal extra computation. As we demonstrate later, PRGD can significantly reduce the number of iterations required compared to RGD. Ultimately, PRGD is a much more efficient algorithm than RGD.


\section{Recovery Guarantee}
\label{sec:convergence}
In this section, we provide a recovery guarantee for the PRGD algorithm (i.e., \cref{PRGD}), assuming that the sensing operator $\mathcal{A}$ satisfies the restricted isometry property (RIP) \cite{RFP}. The analysis involves modifying the proof in \cite{WCCL} for RGD under the canonical metric to the non-canonical metric used in PRGD. When RIP is not satisfied, the recovery guarantee theory of PRGD can still be established by combining techniques from this paper and the corresponding proofs \cite{LMC,CW-phase} for RGD with the canonical metric.

RIP for low-rank matrix recovery was first introduced in \cite{RFP}. It can be traced back to \cite{CT}, where RIP for compressed sensing was developed. For low-rank matrix recovery, RIP is defined as follows.
\begin{definition}[Restricted Isometry Property (RIP)]
The linear operator $\mathcal{A} : \mathbb{R}^{n_1 \times n_2} \to \mathbb{R}^m$ satisfies the restricted isometry property of order $s$ if there exists a constant $\delta_s \in (0,1)$ such that
\begin{equation}\label{RIP}
(1 - \delta_{s}) \| \bm{Z} \|_{F}^2 \leq \| \mathcal{A}\bm{Z} \|_2^2 \leq (1+ \delta_{s}) \| \bm{Z} \|_{F}^2,
\qquad\forall~\bm{Z}~:~\mathrm{rank}(\bm{Z})\leq s.
\end{equation}
\end{definition}
When the sensing matrices $\bm{A}_i$, $1 \leq i \leq m$, in \eqref{eq:defA} contains i.i.d. random entries that follow a mean-$0$ and variance-$\frac{1}{m}$ sub-Gaussian distribution, the sensing operator $\mathcal{A}$ satisfies RIP with high probability, provided $m\geq C(n_1+n_2)s\log(n_1n_2)$, where $C$ depends on $\delta_s$. RIP has been used extensively in analyzing convex and non-convex approaches for low-rank matrix recovery. Under the assumption of RIP, the nuclear norm minimization gives an exact recovery of the underlying low-rank matrix \cite{RFP}. In our previous work \cite{WCCL}, we have shown that, when $\mathcal{A}$ satisfies RIP, the RGD \eqref{eq:RGD} can recover $\bm{X}$ exactly when initialized by one step of IHT starting from $\bm{0}$. There are also many other results showing that RIP implies exact low-rank matrix recovery by, e.g., singular value projection, factor gradient descent, alternating minimization, Gauss-Newton on quotient manifold, and global landscape analysis.

With the help of RIP, we now present our main result in the following theorem.
\begin{theorem}[Recovery Guarantee of PRGD]\label{convergence}
Assume that $\mathcal{A}$ satisfies RIP of order $3r$ with constant $\delta_{3r}$. Let $\bm{X}\in\mathbb{R}^{n_1\times n_2}$ have rank $r$ and condition number $\kappa$. Set the input data $\bm{y}=\mathcal{A}\bm{X}$ and the initial guess $\bm{X}_0 = \mathcal{H}_r (\mathcal{A}^*(\bm{y}))$ in Algorithm \ref{PRGD}, and choose $\alpha_t,\epsilon_t>0$ be some positive parameters that are computed from $\bm{G}_t$ for all $t\in\mathbb{N}$. Then, if  
$$
\delta_{3r} \leq \frac{1}{50 + 100 \sqrt{r} \kappa},
$$
the sequence $\{\bm{X}_t\}_{t\in\mathbb{N}}$ generated by Algorithm \ref{PRGD} satisfies
$$
\| \bm{X}_{t+1} - \bm{X} \|_{F} \leq \mu \| \bm{X}_t - \bm{X} \|_{F},
$$
where $\mu\leq 0.995$ is a positive constant.
\end{theorem}

\cref{convergence} demonstrates that the RIP constant $\delta_{3r}\lesssim\frac{1}{\sqrt{r}\kappa}$ is a sufficient condition for exact recovery of PRGD. When the entries of the sensing matrices follow i.i.d. sub-Gaussian distribution with mean $0$ and variance $\frac1m$, this requirement on the RIP constant is met if $m\gtrsim \kappa^2 r^2 \max\{n_1,n_2\}$. Therefore, PRGD requires $O(\kappa^2 r^2 \max\{n_1,n_2\})$ samples for exact recovery, which is optimal in the matrix dimension and is of the same order as many non-convex algorithms such as RGD and IHT. 

The rest of this section provides the proof of \cref{convergence}. First, we present some supporting lemmas in \cref{sec:lemma}, followed by the main proof in \cref{sec:proofmain}.

\subsection{Supporting Lemmas}\label{sec:lemma}

Our analysis is based on the weighted inner product $\langle \cdot, \cdot \rangle_{\mathcal{W}_t}$ defined in \eqref{eq:weightedinnerprod} and its induced norm $\| \cdot \|_{\mathcal{W}_t}$. In our analysis, we require some basic inequalities related to this weighted inner product and norm, which play a key role in our analysis.  

We first establish an equivalence relation between the weighted norm $\| \cdot \|_{\mathcal{W}_t}$ and the standard norm $\| \cdot \|_{F}$ in the following lemma. 
\begin{lemma}
\label{lem3.1}
For any $t\in\mathbb{N}$ and $\bm{Z} \in \mathbb{R}^{n_1 \times n_2}$, we have
\begin{equation}\label{Euqi}
\nu_t \| \bm{Z} \|_{F}^2  \leq \| \bm{Z} \|_{\mathcal{W}_t}^2 \leq  \mu_t \| \bm{Z} \|_{F}^2,  
\end{equation}
where $\nu_t = \epsilon_t^{\frac12} $ and $\mu_t = \left( \epsilon_t + \| \bm{G}_t \|_{ \vee }^2 \right)^{\frac12} $ with $\| \cdot \|_{\vee}$ defined in \eqref{eq:vnorm}.
\end{lemma}
\begin{proof}
Let $l_{t,i}$ and $r_{t,i}$ be the $i$-th diagonal of $\bm{L}_t$ and $\bm{R}_t$, respectively. By the definitions of $\bm{L}_t = \epsilon_t \bm{I} + \diag(\bm{G}_t\bm{G}_t^{T})$ and $\bm{R}_t = \epsilon_t \bm{I} + \diag(\bm{G}_t^{T}\bm{G}_t)$, we have
$$
l_{t,i}=\epsilon_t+\|\bm{G}_{t}(i,:)\|_2^2,\quad r_{t,i}=\epsilon_t+\|\bm{G}_{t}(:,i)\|_2^2,\qquad\forall~i,
$$
where $\bm{G}_{t}(i,:)$ and $\bm{G}_{t}(:,i)$ are the $i$-th row and $i$-th column of $\bm{G}_t$, respectively. Furthermore, from the definition of $\| \cdot \|_{\mathcal{W}_t}$, we know
\begin{equation}\label{eq:defWtnorm}
\| \bm{Z} \|_{\mathcal{W}_t}^2 = \big\langle \bm{L}_t^{\frac14} \bm{Z} \bm{R}_t^{\frac14},  \bm{Z} \big\rangle
= \sum_{i, j} l_{t,i}^{\frac14} r_{t,j}^{\frac14}  Z_{ij}^2  \quad \textmd{for~any} \quad \bm{Z}\in\mathbb{R}^{n_1\times n_2}.
\end{equation}
To obtain $\nu_t$ and $\mu_t$ in \eqref{Euqi}, we only need to find lower and upper bounds of  $l_{t,i}$ and $r_{t,i}$ for all $i$.
\begin{itemize}
\item For the lower bound $\nu_t$, we use the trivial lower bounds $l_{t,i}\geq\epsilon_t$ and $r_{t,i}\geq \epsilon_t$ for all $i$. Thus, from \eqref{eq:defWtnorm}, we have
$$
\| \bm{Z}\|_{\mathcal{W}_t}^2\geq \epsilon_t^{\frac12}\|\bm{Z}\|_F^2.
$$

\item For the upper bound $\mu_t$, we use
$$
\max\{l_{t,i}, r_{t,i}\}\leq \epsilon_t + \| \bm{G}_t \|_{\vee}^2,   \quad  \forall ~i.
$$
Plugging this into \eqref{eq:defWtnorm}, we obtain 
$$
\| \bm{Z}\|_{\mathcal{W}_t}^2\leq\left( \epsilon_t + \| \bm{G}_t \|_{\vee}^2 \right)^{\frac12}\|\bm{Z}\|_F^2.
$$
\end{itemize}
\end{proof}

In order to complete our convergence proof, we require estimates for the weighted inner product $\langle\cdot,\cdot\rangle_{\bm{L}_t^{\frac14}}$ and norm $\|\cdot\|_{\bm{L}_t^{\frac14}}$ (and similarly for $\langle\cdot,\cdot\rangle_{\bm{R}_t^{\frac14}}$ and $\|\cdot\|_{\bm{R}_t^{\frac14}}$) in $\mathbb{R}^{n_1}$, as defined in \eqref{eq:weightedinnerprodRn1}. We introduce the matrix operator norm $\| \bm{A}\|_{\bm{L}_t^{\frac14}} = \sup_{\| \bm{x} \| \neq 0} \frac{\| \bm{Ax}\|_{\bm{L}_t^{\frac14}}}{\| \bm{x}\|_{\bm{L}_t^{\frac14}}}$ for any matrix in $\mathbb{R}^{n_1\times n_1}$ (and similarly for $\|\cdot\|_{\bm{R}_t^{\frac14}}$). We state the following two lemmas for these two norms.
\begin{lemma}
\label{lem3.3}
Let $\bm{A},\bm{B} \in \mathbb{R}^{n_1 \times n_1}$. Then, we have
\begin{equation*}
\| \bm{AB}\|_{\bm{L}_t^{\frac14}} \leq \cond{\bm{L}_t^{\frac18}} \cdot \| \bm{A}\|_{\bm{L}_t^{\frac14}}\| \bm{B}\|_2.
\end{equation*} 
\end{lemma}
\begin{proof}
Using the definition of $\| \cdot \|_{\bm{L}_t^{\frac14}}$, we get
\begin{equation*}
\begin{aligned}
\| \bm{A} \|_{\bm{L}_t^{\frac14}} &= \sup_{\| \bm{x}  \| \neq 0} \frac{\|\bm{Ax}\|_{\bm{L}_t^{\frac14}}}{\|\bm{x}\|_{\bm{L}_t^{\frac14}}}
= \sup_{\| \bm{x} \| \neq 0} \frac{\|\bm{L}_t^{\frac18} \bm{Ax}\|_2}{\|\bm{L}_t^{\frac18} \bm{x}\|_2}
=\sup_{\| \bm{y} \| \neq 0} \frac{\|\bm{L}_t^{\frac18} \bm{A} \bm{L}_t^{-\frac18} \bm{y}\|_2}{\| \bm{y}\|_2}
=\| \bm{L}_t^{\frac18} \bm{A} \bm{L}_t^{-\frac18} \|_2.
\end{aligned}
\end{equation*}
Thus, a direct calculation gives 
\begin{equation*}
\begin{aligned}
\| \bm{AB} \|_{\bm{L}_t^{\frac14}} &= \| \bm{L}_t^{\frac18} \bm{AB} \bm{L}_t^{-\frac18} \|_2
\leq \| \bm{L}_t^{\frac18} \bm{A} \bm{L}_t^{-\frac18} \|_2 \| \bm{L}_t^{\frac18} \|_2 \| \bm{B}\|_2 \| \bm{L}_t^{-\frac18} \|_2
= \cond{\bm{L}_t^{\frac18}} \cdot\| \bm{A} \|_{\bm{L}_t^{\frac14}} \| \bm{B} \|_2.
\end{aligned}
\end{equation*}
This completes the proof.
\end{proof}
\begin{lemma}
\label{lem3.4}
Let $\bm{A} \in \mathbb{R}^{n_1 \times n_1}$, $\bm{B} \in \mathbb{R}^{n_1 \times n_2}$, and $\bm{C} \in \mathbb{R}^{n_2 \times n_2}$. Then, we have
\begin{equation}\label{Le3.11}
\| \bm{AB} \|_{\mathcal{W}_t} \leq \| \bm{A} \|_{\bm{L}_t^{\frac14}} \| \bm{B}\|_{\mathcal{W}_t} \quad \textmd{and} \quad \| \bm{BC} \|_{\mathcal{W}_t} \leq \cond{\bm{R}_t^{-\frac14}}\cdot \| \bm{B} \|_{\mathcal{W}_t}  \| \bm{C} \|_{\bm{R}_t^{\frac14}}.
\end{equation} 
\end{lemma}
\begin{proof}
Using the definition of $\|\cdot\|_{\mathcal{W}_t}$ and a property of the Frobenius norm, we get
\begin{equation*}
\begin{aligned}
\|\bm{AB} \|_{\mathcal{W}_t} &= \| \bm{L}_t^{\frac18} \bm{AB} \bm{R}_t^{\frac18}\|_{F}
\leq \| \bm{L}_t^{\frac18} \bm{A} \bm{L}_t^{-\frac18} \|_2 \| \bm{L}_t^{\frac18} \bm{B} \bm{R}_t^{\frac18}\|_{F}
= \| \bm{L}_t^{\frac18} \bm{A} \bm{L}_t^{-\frac18} \|_2 \| \bm{B} \|_{\mathcal{W}_t}
= \| \bm{A} \|_{\bm{L}_t^{\frac14}} \| \bm{B} \|_{\mathcal{W}_t},
\end{aligned}
\end{equation*}
which is the first inequality in \eqref{Le3.11}. For the second inequality, we note that
\begin{equation*}
\begin{aligned}
\| \bm{BC}  \|_{\mathcal{W}_t} &= \| \bm{L}_t^{\frac18} \bm{BC} \bm{R}_t^{\frac18} \|_{F}
\leq \| \bm{B} \|_{\mathcal{W}_t} \| \bm{R}_t^{-\frac18} \bm{C} \bm{R}_t^{\frac18} \|_{2}.
\end{aligned}
\end{equation*}
Denoting $\bm{C}_0 = \bm{R}_t^{\frac18} \bm{C} \bm{R}_t^{-\frac18}$, we have $\|\bm{C}_0\|_2=\| \bm{C}  \|_{\bm{R}_t^{\frac14}}$ and
$$
\| \bm{R}_t^{-\frac18} \bm{C} \bm{R}_t^{\frac18} \|_{2} = \| \bm{R}_t^{-\frac14} \bm{C}_0 \bm{R}_t^{\frac14} \|_{2}\leq \| \bm{R}_t^{-\frac14}\|_2\|\bm{C}_0\|_2\|\bm{R}_t^{\frac14} \|_{2}= \cond{\bm{R}_t^{-\frac14}}\cdot \| \bm{C}_0 \|_{2}.
$$
Therefore, we obtain
\begin{equation*}
\begin{aligned}
\| \bm{BC}  \|_{\mathcal{W}_t} \leq \cond{\bm{R}_t^{-\frac14}}\cdot \| \bm{B} \|_{\mathcal{W}_t} \| \bm{R}_t^{\frac18} \bm{C} \bm{R}_t^{-\frac18} \|_{2}
= \cond{\bm{R}_t^{-\frac14}} \cdot\| \bm{B} \|_{\mathcal{W}_t} \| \bm{C}  \|_{\bm{R}_t^{\frac14}}.
\end{aligned}
\end{equation*}
\end{proof}

Based on the above inequalities, we can provide the following estimates for the column subspaces of two matrices: 
\begin{lemma}
\label{lem3.5}
Let $\bm{X}_t = \bm{U}_t \bm{\Sigma}_t \bm{V}_t^T$ and $\bm{X} = \bm{U} \bm{\Sigma} \bm{V}^T$ be compact SVDs of $\bm{X}_t$ and $\bm{X}$, respectively. Then, we have
\begin{equation}\label{eq5}
\| \widetilde{\bm{U}}_t \widetilde{\bm{U}}_t^T \bm{L}_t^{\frac14}  - \widetilde{\bm{U}} \widetilde{\bm{U}}^T \bm{L}_t^{\frac14} \|_{\bm{L}_t^{\frac14}} \leq \frac{\cond{\bm{L}_t^{\frac38}}}{\sigma_{\min}(\bm{X})} \| \bm{X}_t - \bm{X}\|_2,
\end{equation} 
where $\widetilde{\bm{U}}_t= \bm{U}_t \big( \bm{U}_t^T \bm{L}_t^{\frac14} \bm{U}_t \big)^{-\frac12}$ and $\widetilde{\bm{U}}= \bm{U} \big( \bm{U}^T \bm{L}_t^{\frac14} \bm{U} \big)^{-\frac12}$ are respectively orthonormalizations of $\bm{U}_t$ and $\bm{U}$ under the inner product $\langle\cdot,\cdot\rangle_{\bm{L}_t}$.
\end{lemma}
\begin{proof}
We follow the same arguments as  \cite[Lemma 4.2]{WCCL}. Since $\widetilde{\bm{U}}_t \widetilde{\bm{U}}_t^T \bm{L}_t^{\frac14} $ and $\widetilde{\bm{U}} \widetilde{\bm{U}}^T  \bm{L}_t^{\frac14}$  are orthogonal projections under the inner product $\langle \cdot, \cdot \rangle_{\bm{L}_t^{\frac14}}$ onto the column subspace of $\bm{X}_t$ and $\bm{X}$ respectively, \cite[Theorem 2.6.1]{GL} implies
$$
\big\| \widetilde{\bm{U}}_t \widetilde{\bm{U}}_t^T \bm{L}_t^{\frac14}  - \widetilde{\bm{U}} \widetilde{\bm{U}}^T  \bm{L}_t^{\frac14} \big\|_{\bm{L}_t^{\frac14}} = \big\| \widetilde{\bm{U}} \widetilde{\bm{U}}^T \bm{L}_t^{\frac14} \big( \bm{I} -  \widetilde{\bm{U}}_t \widetilde{\bm{U}}_t^T \bm{L}_t^{\frac14} \big) \big\|_{\bm{L}_t^{\frac14}} =  \big\|  \big( \bm{I} -  \widetilde{\bm{U}}_t \widetilde{\bm{U}}_t^T \bm{L}_t^{\frac14} \big) \widetilde{\bm{U}} \widetilde{\bm{U}}^T\bm{L}_t^{\frac14} \big\|_{\bm{L}_t^{\frac14}}.
$$
Using the definitions of $\widetilde{\bm{U}}$ and $\widetilde{\bm{U}}_t$, we have 
$$
\widetilde{\bm{U}} \widetilde{\bm{U}}^T=\bm{U}\big( \bm{U}^T \bm{L}_t^{\frac14} \bm{U} \big)^{-1} \bm{U}^T=\bm{X}\bm{V}\bm{\Sigma}^{-1}\big( \bm{U}^T \bm{L}_t^{\frac14} \bm{U} \big)^{-1} \bm{U}^T
$$
and
$$
\big(\bm{I} - \widetilde{\bm{U}}_t \widetilde{\bm{U}}_t^T \bm{L}_t^{\frac14} \big) \bm{X}_t =\bm{0}.
$$
Altogether, we obtain
\begin{equation*}
\begin{aligned} 
\big\| \widetilde{\bm{U}}_t \widetilde{\bm{U}}_t^T \bm{L}_t^{\frac14}  - \widetilde{\bm{U}} \widetilde{\bm{U}}^T  \bm{L}_t^{\frac14} \big\|_{\bm{L}_t^{\frac14}} 
&=\big\| \big(\bm{I} - \widetilde{\bm{U}}_t \widetilde{\bm{U}}_t^T \bm{L}_t^{\frac14} \big)\bm{X} \bm{V} \bm{\Sigma}^{-1} \big( \bm{U}^T \bm{L}_t^{\frac14} \bm{U} \big)^{-1} \bm{U}^T \bm{L}_t^{\frac14} \big\|_{\bm{L}_t^{\frac14}} \\
&= \big\| \big(\bm{I} - \widetilde{\bm{U}}_t \widetilde{\bm{U}}_t^T \bm{L}_t^{\frac14} \big) (\bm{X}- \bm{X}_t) \bm{V} \bm{\Sigma}^{-1} \big( \bm{U}^T \bm{L}_t^{\frac14} \bm{U} \big)^{-1} \bm{U}^T \bm{L}_t^{\frac14} \big\|_{\bm{L}_t^{\frac14}}\\
&\leq \cond{\bm{L}_t^{\frac18}}\cdot  \| \bm{I} - \widetilde{\bm{U}}_t \widetilde{\bm{U}}_t^T \bm{L}_t^{\frac14}  \|_{\bm{L}_t^{\frac14}}\| (\bm{X} - \bm{X}_t )\bm{V} \bm{\Sigma}^{-1} ( \bm{U}^T \bm{L}_t^{\frac14} \bm{U} )^{-1} \bm{U}^T \bm{L}_t^{\frac14} \|_{2}\\
&\leq \cond{\bm{L}_t^{\frac18}}\cdot \| \bm{X} - \bm{X}_t \|_2 \| \bm{V} \|_2 \| \bm{\Sigma}^{-1} \|_2 \| \big(\bm{U}^T \bm{L}_t^{\frac14} \bm{U} \big)^{-1} \|_2 \| \bm{U}^T \|_2 \| \bm{L}_t^{\frac14} \|_2\\
&\leq \cond{\bm{L}_t^{\frac18}}\cdot \frac{\| \bm{X} - \bm{X}_t \|_2}{\sigma_{\min}(\bm{X})} \frac{\lambda_{\max} (\bm{L}_t^{\frac14})}{\lambda_{\min}(\bm{L}_t^{\frac14})}= \frac{\cond{\bm{L}_t^{\frac38}}}{\sigma_{\min}(\bm{X})} \| \bm{X} - \bm{X}_t \|_2,
\end{aligned} 
\end{equation*}
where in the three inequalities we have respectively used Lemma \ref{lem3.3}, the fact that  $\bm{I} - \widetilde{\bm{U}}_t \widetilde{\bm{U}}_t^T \bm{L}_t^{\frac14}$ is an orthogonal projector under the weighted inner product, and the inequality $\| \big(\bm{U}^T \bm{L}_t^{\frac14} \bm{U} \big)^{-1} \|_2\leq\frac{1}{\lambda_{\min}(\bm{L}_t^{\frac14})}$ following from
$$
\begin{aligned}
\| \big(\bm{U}^T \bm{L}_t^{\frac14} \bm{U} \big)^{-1} \|_2 &=\frac{1}{\min_{\| \bm{x} \|_2 =1} \langle\bm{U}^T \bm{L}_t^{\frac14} \bm{U} \bm{x}, \bm{x}\rangle}
\leq \frac{1}{\lambda_{\min} (\bm{L}_t^{\frac14})\cdot \min_{\| \bm{x}\| =1} \langle \bm{U}\bm{x}, \bm{U}\bm{x} \rangle }= \frac{1}{\lambda_{\min} (\bm{L}_t^{\frac14})}.
\end{aligned}
$$
\end{proof}

The following lemma presents an important property of the projection $\widetilde{\mathcal{P}}_{\mathbb{T}_t}$, which incorporates second-order information of the smooth low-rank matrix manifold, as highlighted in \cite{WCCL}. This lemma is crucial for the proof of our main result.    

\begin{lemma}
\label{lem3.6}
We have
\begin{equation*}
\begin{aligned}
\|\big(\mathcal{I} - \widetilde{\mathcal{P}}_{\mathbb{T}_t}  \big) \bm{X} \|_{\mathcal{W}_t} 
\leq \frac{\mu_t^{\frac{5}{4}}}{\nu_t^{\frac{7}{4}}\cdot \sigma_{\min}(\bm{X})} \| \bm{X} - \bm{X}_t \|_{\mathcal{W}_t}^2,
\end{aligned}
\end{equation*}
where the constants $\nu_t$ and $\mu_t$ are from \eqref{Euqi}.
\end{lemma}
\begin{proof}
Let $\widetilde{\mathcal{P}}_{\mathbb{T}}$ be the orthogonal projector onto
$\mathcal{\mathbb{T}}$, the tangent space of $\mathcal{M}_r$ at $\bm{X}$, under the inner product $\langle \cdot, \cdot \rangle_{\W}$ defined in \eqref{eq:weightedinnerprod}. Then, for any $\bm{Z} \in \mathbb{R}^{n_1 \times n_2}$, we have
\begin{equation}\label{proj-0X}
\widetilde{\mathcal{P}}_{\mathbb{T}} (\bm{Z}) = \widetilde{\bm{U}} \widetilde{\bm{U}}^T \bm{L}_t^{\frac14} \bm{Z} + \bm{Z} \bm{R}_t^{\frac14}\widetilde{\bm{V}} \widetilde{\bm{V}}^T - \widetilde{\bm{U}} \widetilde{\bm{U}}^T \bm{L}_t^{\frac14} \bm{Z} \bm{R}_t^{\frac14} \widetilde{\bm{V}} \widetilde{\bm{V}}^T,
\end{equation}
where $\bm{X}=\bm{U}\bm{\Sigma}\bm{V}^T$ is a compact SVD of $\bm{X}$, $\widetilde{\bm{U}}= \bm{U} \big( \bm{U}^T \bm{L}_t^{\frac14} \bm{U} \big)^{-\frac12}$, and $\widetilde{\bm{V}} = \bm{V} \big( \bm{V}^T \bm{R}_t^{\frac14} \bm{V} \big)^{-\frac12}$. Obviously, $\widetilde{\mathcal{P}}_{\mathbb{T}}(\bm{X}) = \bm{X}$. 

Using the similar argument as in \cite[Lemma 4.1]{WCCL}, we obtain
\begin{equation}\label{eq6}
\begin{aligned}
( \mathcal{I} - \widetilde{\mathcal{P}}_{\mathbb{T}_t} ) \bm{X} =( \widetilde{\mathcal{P}}_{\mathbb{T}} - \widetilde{\mathcal{P}}_{\mathbb{T}_t} ) \bm{X}  =\big(\widetilde{\bm{U}} \widetilde{\bm{U}}^{T} \bm{L}_t^{\frac14} - \widetilde{\bm{U}}_t \widetilde{\bm{U}}_t^{T} \bm{L}_t^{\frac14} \big) ( \bm{X}- \bm{X}_t) \big( \bm{I} - \bm{R}_t^{\frac14} \widetilde{\bm{V}}_t \widetilde{\bm{V}}_t^{T} \big).
\end{aligned}
\end{equation}
This,  together with Lemmas \ref{lem3.1}, \ref{lem3.4}, and \ref{lem3.5}, yields that 
\begin{equation*}
\begin{aligned}
\|\big( \mathcal{I} - \widetilde{\mathcal{P}}_{\mathbb{T}_t} \big) \bm{X} \|_{\mathcal{W}_t} & \leq \cond{\bm{R}_t^{-\frac14}}\cdot \| \big( \widetilde{\bm{U}} \widetilde{\bm{U}}^{T} \bm{L}_t^{\frac14} - \widetilde{\bm{U}}_t \widetilde{\bm{U}}_t^{T} \bm{L}_t^{\frac14} \big) \|_{\bm{L}_t^{\frac14}} \| \bm{X}- \bm{X}_t \|_{\mathcal{W}_t} \| \bm{I} - \bm{R}_t^{\frac14} \widetilde{\bm{V}}_t \widetilde{\bm{V}}_t^{T} \|_{\bm{R}_t^{\frac14}}\\
&\leq \frac{\cond{\bm{R}_t^{-\frac14}}\cdot \cond{\bm{L}_t^{\frac38}}}{\sigma_{\min}(\bm{X})} \| \bm{X}_t - \bm{X} \|_2 \| \bm{X}_t - \bm{X} \|_{\mathcal{W}_t}\\
&\leq \frac{\cond{\bm{R}_t^{-\frac14}} \cdot\cond{\bm{L}_t^{\frac38}}}{\sigma_{\min}(\bm{X})} \| \bm{X}_t - \bm{X} \|_F \| \bm{X}_t - \bm{X} \|_{\mathcal{W}_t}\\
&\leq  \frac{ \cond{\bm{R}_t^{-\frac14}}\cdot\cond{\bm{L}_t^{\frac38}}}{ \nu_t^{\frac12}\cdot \sigma_{\min}(\bm{X})} \| \bm{X}_t - \bm{X} \|_{\mathcal{W}_t}^2
\leq\frac{\mu_t^{\frac{5}{4}}}{\nu_t^{\frac{7}{4}}\cdot\sigma_{\min}(\bm{X})} \| \bm{X} - \bm{X}_t \|_{\mathcal{W}_t}^2.
\end{aligned}
\end{equation*}
This completes the proof.
\end{proof}

The following two lemmas are related to RIP.
\begin{lemma}\label{Lem3.7}
Let $\bm{Z}_1$ and $\bm{Z}_2$ be two matrices of rank $r_1$ and  $r_2$, respectively, with $r_1+ r_2 \leq \min\{n_1, n_2\}$ and satisfying $\left\langle \bm{Z}_1,  \bm{Z}_2 \right\rangle_{\mathcal{W}_t} = 0$. Assume that $\mathcal{A}$ satisfies RIP \eqref{RIP} with constant $\delta_{r_1+r_2}$. Then, we have
\begin{equation*}
\big| \big\langle \mathcal{A} \bm{Z}_1, \mathcal{A} \bm{Z}_2 \big\rangle \big| \leq  \frac{1}{2}  \big((\nu_t^{-1} - \mu_t^{-1}) + (\nu_t^{-1} + \mu_t^{-1}) \delta_{r_1+ r_2} \big) \| \bm{Z}_1 \|_{\mathcal{W}_t} \| \bm{Z}_2 \|_{\mathcal{W}_t},
\end{equation*}
where $\nu_t$ and $\mu_t$ are from Lemma \ref{Euqi}. 
\end{lemma}
\begin{proof}
Without loss of generality, we assume $\| \bm{Z}_1 \|_{\mathcal{W}_t}= \| \bm{Z}_2 \|_{\mathcal{W}_t}=1$. According to the RIP and Lemma \ref{Euqi}, we have
\begin{equation*}
\begin{aligned}
\| \mathcal{A} (\bm{Z}_1 \pm \bm{Z}_2) \|_2^2 &\geq \big( 1 - \delta_{r_1+ r_2} \big) \| \bm{Z}_1 \pm \bm{Z}_2 \|_F^2
\geq \mu_t^{-1} \big( 1 - \delta_{r_1+ r_2} \big) \| \bm{Z}_1 \pm \bm{Z}_2 \|_{\mathcal{W}_t}^2.
\end{aligned}
\end{equation*}
Similarly,
\begin{equation*}
\begin{aligned}
\| \mathcal{A} (\bm{Z}_1 \pm \bm{Z}_2) \|_2^2 &\leq \big( 1 + \delta_{r_1+ r_2} \big) \| \bm{Z}_1 \pm \bm{Z}_2 \|_F^2
\leq \nu_t^{-1} \big( 1 + \delta_{r_1+ r_2} \big) \| \bm{Z}_1 \pm \bm{Z}_2 \|_{\mathcal{W}_t}^2.
\end{aligned}
\end{equation*}
Since  $\big\langle \bm{Z}_1, \bm{Z}_2\big\rangle_{\mathcal{W}_t} = 0$ implies $\| \bm{Z}_1 \pm \bm{Z}_2\|_{\mathcal{W}_t}^2 = 2$, the above two inequalities imply
\begin{equation*}
2 \mu_t^{-1}\big( 1 - \delta_{r_1+ r_2} \big) \leq \| \mathcal{A} (\bm{Z}_1 \pm \bm{Z}_2) \|_2^2 \leq 2 \nu_t^{-1} \big( 1 + \delta_{r_1+ r_2} \big).
\end{equation*}
Therefore, 
\begin{equation*}
\begin{aligned}
| \langle \mathcal{A} \bm{Z}_1, \mathcal{A} \bm{Z}_2 \rangle | &= \frac{1}{4} \big| \| \mathcal{A} (\bm{Z}_1 + \bm{Z}_2) \|_2^2 - \| \mathcal{A} (\bm{Z}_1 - \bm{Z}_2) \|_2^2 \big|
\leq \frac{1}{4} \big| 2\nu_t^{-1} ( 1 + \delta_{r_1+ r_2}) - 2\mu_t^{-1}( 1 - \delta_{r_1+ r_2}) \big|\\
&= \frac{1}{2} \left( (\nu_t^{-1} - \mu_t^{-1}) + ( \nu_t^{-1}+ \mu_t^{-1} ) \delta_{r_1+ r_2} \right).
\end{aligned}
\end{equation*}
\end{proof}

\begin{lemma}
\label{lem3.8}
Assume that $\mathcal{A}$ satisfies RIP \eqref{RIP} with constant $\delta_{3r}$. Then 
$$
\| \widetilde{\mathcal{P}}_{\mathbb{T}_t} \mathcal{W}_t^{-1} \mathcal{A}^*\mathcal{A} ( \mathcal{I} - \widetilde{\mathcal{P}}_{\mathbb{T}_t} ) \bm{X} \|_{\mathcal{W}_t} \leq \frac{1}{2} \big( (\nu_t^{-1} - \mu_t^{-1}) + ( \nu_t^{-1} + \mu_t^{-1} ) \delta_{3r} \big) ~~ \| ( \mathcal{I} - \widetilde{\mathcal{P}}_{\mathbb{T}_t} )\bm{X} \|_{\mathcal{W}_t}.
$$
\end{lemma}
\begin{proof} 
A direct calculation gives
\begin{equation*}
\begin{split}
\|\widetilde{\mathcal{P}}_{\mathbb{T}_t} \mathcal{W}_t^{-1} \mathcal{A}^* \mathcal{A} ( \mathcal{I} - \widetilde{\mathcal{P}}_{\mathbb{T}_t} ) \bm{X} \|_{\mathcal{W}_t}
&= \sup_{\| \bm{Z} \|_{\mathcal{W}_t}=1}  \big| \langle\widetilde{\mathcal{P}}_{\mathbb{T}_t} \mathcal{W}_t^{-1} \mathcal{A}^* \mathcal{A}( \mathcal{I} - \widetilde{\mathcal{P}}_{\mathbb{T}_t} )\bm{X}, \bm{Z}\rangle_{\mathcal{W}_t} \big| \\
&=\sup_{\| \bm{Z} \|_{\mathcal{W}_t} = 1} | \langle   \mathcal{A} ( \mathcal{I} - \widetilde{\mathcal{P}}_{\mathbb{T}_t} )\bm{X}, \mathcal{A}\widetilde{\mathcal{P}}_{\mathbb{T}_t} \bm{Z} \rangle |.
\end{split}
\end{equation*}
We now apply Lemma \ref{Lem3.7} with $\bm{Z}_1=( \mathcal{I} - \widetilde{\mathcal{P}}_{\mathbb{T}_t} )\bm{X}$ and $\bm{Z}_2=\widetilde{\mathcal{P}}_{\mathbb{T}_t} \bm{Z}$. It can be easily checked that $\mathrm{rank}(\bm{Z}_1)=r$, $\mathrm{rank}(\bm{Z}_2)=2r$, and $\langle\bm{Z}_1,\bm{Z}_2\rangle_{\mathcal{W}_t}=0$. Thus, Lemma \ref{Lem3.7} gives
\begin{equation*}
\begin{aligned}
\|\widetilde{\mathcal{P}}_{\mathbb{T}_t} \mathcal{W}_t^{-1} \mathcal{A}^* \mathcal{A} \big( \mathcal{I} - \widetilde{\mathcal{P}}_{\mathbb{T}_t} \big) \bm{X} \|_{\mathcal{W}_t}
&\leq \sup_{\| \bm{Z} \|_{\mathcal{W}_t} = 1}  \frac{1}{2} \big( (\nu_t^{-1} - \mu_t^{-1}) + (\nu_t^{-1} + \mu_t^{-1}) \delta_{3r} \big) \| ( \mathcal{I} - \widetilde{\mathcal{P}}_{\mathbb{T}_t} )\bm{X} \|_{\mathcal{W}_t} \| \widetilde{\mathcal{P}}_{\mathbb{T}_t} \bm{Z}  \|_{\mathcal{W}_t}\\
&\leq \frac{1}{2} \big( (\nu_t^{-1} - \mu_t^{-1}) + (\nu_t^{-1} + \mu_t^{-1}) \delta_{3r} \big) \| ( \mathcal{I} - \widetilde{\mathcal{P}}_{\mathbb{T}_t} ) \bm{X} \|_{\mathcal{W}_t}. 
\end{aligned}
\end{equation*}
\end{proof}

\subsection{Proof of \cref{convergence}}
\label{sec:proofmain}
Now we give the proof of our main result \cref{convergence}.
\begin{proof}[Proof of \cref{convergence}]
Since $\bm{X}_{t+1}$ is the best rank-$r$ approximation of $\bm{W}_{t}$ under the Frobinus norm, using the equivalence between the Frobinus norm and the $\| \cdot \|_{\mathcal{W}_t}$ norm, we obtain
\begin{equation}\nonumber
\begin{aligned}
\| \bm{X}_{t+1} - \bm{X} \|_{\mathcal{W}_t} &\leq \| \bm{X}_{t+1} - \bm{W}_t \|_{\mathcal{W}_t} + \| \bm{W}_t - \bm{X} \|_{\mathcal{W}_t}\leq \sqrt{\mu_t} \| \bm{X}_{t+1} - \bm{W}_t \|_{F} + \| \bm{W}_t - \bm{X} \|_{\mathcal{W}_t}\\
&\leq (1+ \sqrt{\mu_t/\nu_t}) \| \bm{W}_t - \bm{X} \|_{\mathcal{W}_t} := (1+ \sqrt{\rho_t}) \| \bm{W}_t - \bm{X} \|_{\mathcal{W}_t}.
\end{aligned}
\end{equation}
Here and in what follows, we use $\rho_t = \frac{\mu_t}{\nu_t}$. Therefore, it suffices to estimate $\| \bm{W}_t - \bm{X} \|_{\mathcal{W}_t}$. Using \cref{PRGD} and the definition of $\mathcal{W}_t$, we can express $\bm{W}_t$ as
$$
\bm{W}_t = \bm{X}_t - \alpha_t \widetilde{\mathcal{P}}_{\mathbb{T}_t} \big( \bm{L}_t^{-\frac14} \bm{G}_t \bm{R}_t^{-\frac14} \big) = \bm{X}_t - \alpha_t \widetilde{\mathcal{P}}_{\mathbb{T}_t} \mathcal{W}_t^{-1} \bm{G}_t.
$$
Using the definition of $\bm{G}_t$,  we get
\begin{equation}\nonumber
\begin{aligned}
\|\bm{W}_t - \bm{X} \|_{\mathcal{W}_t} &= \| \bm{X}_t - \alpha_t \widetilde{\mathcal{P}}_{\mathbb{T}_t} \mathcal{W}_t^{-1} \bm{G}_t - \bm{X} \|_{\mathcal{W}_t}
= \| \bm{X}_t - \bm{X} - \alpha_t \widetilde{\mathcal{P}}_{\mathbb{T}_t} \mathcal{W}_t^{-1} \mathcal{A}^* \mathcal{A} (\bm{X}_t - \bm{X}) \|_{\mathcal{W}_t}\\
&\leq \| (\mathcal{I} - \alpha_t \widetilde{\mathcal{P}}_{\mathbb{T}_t} \mathcal{W}_t^{-1} \mathcal{A}^* \mathcal{A} \widetilde{\mathcal{P}}_{\mathbb{T}_t}) (\bm{X}_t - \bm{X}) \|_{\mathcal{W}_t} + \alpha_t \| \widetilde{\mathcal{P}}_{\mathbb{T}_t} \mathcal{W}_t^{-1} \mathcal{A}^* \mathcal{A} ( \mathcal{I} - \widetilde{\mathcal{P}}_{\mathbb{T}_t} ) (\bm{X}_t - \bm{X})\|_{\mathcal{W}_t}\\
&\leq \underbrace{\| (\widetilde{\mathcal{P}}_{\mathbb{T}_t} - \alpha_t \widetilde{\mathcal{P}}_{\mathbb{T}_t} \mathcal{W}_t^{-1} \mathcal{A}^* \mathcal{A} \widetilde{\mathcal{P}}_{\mathbb{T}_t}) (\bm{X}_t - \bm{X}) \|_{\mathcal{W}_t}}_{I_1} + \underbrace{\| (\mathcal{I} - \widetilde{\mathcal{P}}_{\mathbb{T}_t}) ( \bm{X}_t - \bm{X}) \|_{\mathcal{W}_t}}_{I_2} \\
&\quad  + \alpha_t \underbrace{\| \widetilde{\mathcal{P}}_{\mathbb{T}_t} \mathcal{W}_t^{-1} \mathcal{A}^* \mathcal{A} ( \mathcal{I} - \widetilde{\mathcal{P}}_{\mathbb{T}_t} ) (\bm{X}_t - \bm{X})\|_{\mathcal{W}_t}}_{I_3}.
\end{aligned}
\end{equation}
Choose
\begin{equation}\label{alpha-rho}
\alpha_t  = \frac{2}{\mu_t^{-1}+ \nu_t^{-1}}. 
\end{equation}
We can estimate each of these three terms separately as follows:
\begin{itemize}
\item \emph{Estimation of $I_1$.}  Since the operator $ \widetilde{\mathcal{P}}_{\mathbb{T}_t} - \alpha_t \widetilde{\mathcal{P}}_{\mathbb{T}_t} \mathcal{W}_t^{-1} \mathcal{A}^* \mathcal{A} \widetilde{\mathcal{P}}_{\mathbb{T}_t}$ is self-adjoint under $\langle \cdot, \cdot \rangle_{\mathcal{W}_t}$, we have
\begin{equation}\label{eq7}
\begin{aligned}
\|  \widetilde{\mathcal{P}}_{\mathbb{T}_t} -  \alpha_t \widetilde{\mathcal{P}}_{\mathbb{T}_t} \mathcal{W}_t^{-1} \mathcal{A}^* \mathcal{A} \widetilde{\mathcal{P}}_{\mathbb{T}_t}  \|_{\mathcal{W}_t} & = \sup_{\| \bm{Z} \|_{\mathcal{W}_t} = 1} | \langle ( \widetilde{\mathcal{P}}_{\mathbb{T}_t} -  \alpha_t \widetilde{\mathcal{P}}_{\mathbb{T}_t} \mathcal{W}_t^{-1} \mathcal{A}^* \mathcal{A} \widetilde{\mathcal{P}}_{\mathbb{T}_t} ) \bm{Z}, \bm{Z}  \rangle_{\mathcal{W}_t} | \\
&= \sup_{\| \bm{Z} \|_{\mathcal{W}_t} = 1} \big|  \| \widetilde{\mathcal{P}}_{\mathbb{T}_t} \bm{Z}\|_{\mathcal{W}_t}^2 - \alpha_t \langle  \mathcal{A}^* \mathcal{A} \widetilde{\mathcal{P}}_{\mathbb{T}_t} \bm{Z}, \mathcal{W}_t^{-1} \widetilde{\mathcal{P}}_{\mathbb{T}_t}  \bm{Z}  \rangle_{\mathcal{W}_t} \big| \\
&= \sup_{\| \bm{Z} \|_{\mathcal{W}_t} = 1} \big|  \| \widetilde{\mathcal{P}}_{\mathbb{T}_t} \bm{Z}\|_{\mathcal{W}_t}^2 -  \alpha_t \| \mathcal{A} \widetilde{\mathcal{P}}_{\mathbb{T}_t} \bm{Z}\|_{2}^2\big|.
\end{aligned}
\end{equation}
Since $\mathrm{rank}( \widetilde{\mathcal{P}}_{\mathbb{T}_t} \bm{Z})\leq 2r$, RIP \eqref{RIP} and Lemma \ref{lem3.1} imply
\begin{equation}\label{eq8}
\begin{split}
\mu_t^{-1} (1 - \delta_{2r}) \| \widetilde{\mathcal{P}}_{\mathbb{T}_t} \bm{Z} \|_{\mathcal{W}_t}^2\leq  (1 - \delta_{2r})& \| \widetilde{\mathcal{P}}_{\mathbb{T}_t} \bm{Z} \|_{F}^2 \leq \| \mathcal{A} \widetilde{\mathcal{P}}_{\mathbb{T}_t} \bm{Z} \|_2^2\cr
&\leq (1+ \delta_{2r}) \| \widetilde{\mathcal{P}}_{\mathbb{T}_t} \bm{Z} \|_{F}^2 \leq \nu_t^{-1} (1+ \delta_{2r}) \| \widetilde{\mathcal{P}}_{\mathbb{T}_t} \bm{Z} \|_{\mathcal{W}_t}^2.
\end{split}
\end{equation}
Combining \eqref{eq7} with \eqref{eq8}, we get
\begin{equation}\label{eq9}
\begin{aligned}
\|  \widetilde{\mathcal{P}}_{\mathbb{T}_t} -  \alpha_t \widetilde{\mathcal{P}}_{\mathbb{T}_t} \mathcal{W}_t^{-1} \mathcal{A}^* \mathcal{A} \widetilde{\mathcal{P}}_{\mathbb{T}_t} \|_{\mathcal{W}_t} & 
\leq \max \left\{ \left| 1- \alpha_t \nu_t^{-1}(1+ \delta_{2r}) \right|, \left| 1-  \alpha_t \mu_t^{-1} (1- \delta_{2r}) \right| \right\}  \sup_{\| \bm{Z} \|_{\mathcal{W}_t} =1}  \|  \widetilde{\mathcal{P}}_{\mathbb{T}_t} \bm{Z} \|_{\mathcal{W}_t}^2\\
&\leq \max \left\{ \left|1- \alpha_t \nu_t^{-1}(1+ \delta_{2r}) \right|, \left|1-  \alpha_t \mu_t^{-1} (1- \delta_{2r})\right| \right\}\\
&= 1-  \frac{2}{1 +  \rho_t} (1- \delta_{2r}),
\end{aligned}
\end{equation}
where in the last equality we have used \eqref{alpha-rho}.
\item \emph{Estimation of $I_2$.}
Since $( \mathcal{I} - \widetilde{\mathcal{P}}_{\mathbb{T}_t})\bm{X}_t = \bm{0}$, Lemma \ref{lem3.6} implies
\begin{equation}\label{eq10}
\begin{aligned}
\|\big( \mathcal{I} - \widetilde{\mathcal{P}}_{\mathbb{T}_t} \big) (\bm{X}_t - \bm{X}) \|_{\mathcal{W}_t} \leq \frac{\rho_t^{\frac{5}{4}}}{ \nu_t^{\frac12}\sigma_{\min}(\bm{X})} \| \bm{X}_t - \bm{X} \|_{\mathcal{W}_t}^2.
\end{aligned}
\end{equation}

\item \emph{Estimation of $I_3$.}
Finally, we derive a bound for $\alpha_t \| \widetilde{\mathcal{P}}_{\mathbb{T}_t} \mathcal{W}_t^{-1} \mathcal{A}^* \mathcal{A} \big( \mathcal{I} - \widetilde{\mathcal{P}}_{\mathbb{T}_t} \big) (\bm{X}_t - \bm{X})\|_{\mathcal{W}_t}$. Using Lemma \ref{lem3.8} and equation \eqref{alpha-rho}, we obtain
\begin{equation}\label{eq11}
\begin{aligned}
\alpha_t \| \widetilde{\mathcal{P}}_{\mathbb{T}_t} \mathcal{W}_t^{-1} \mathcal{A}^* \mathcal{A} \big( \mathcal{I} - \widetilde{\mathcal{P}}_{\mathbb{T}_t} \big) \bm{X} \|_{\mathcal{W}_t} &\leq \left( \frac{\rho_t - 1}{\rho_t +1} + \delta_{3r} \right) \| \bm{X}_t - \bm{X} \|_{\mathcal{W}_t}.
\end{aligned}
\end{equation}
\end{itemize}

By combining \eqref{eq9}, \eqref{eq10}, and \eqref{eq11}, we obtain 
\begin{equation}\nonumber
\begin{aligned}
&\| \bm{X}_{t+1} - \bm{X}\|_{\mathcal{W}_t} \leq ( 1+ \sqrt{\rho_t}) \|  \bm{W}_t - \bm{X} \|_{\mathcal{W}_t}\\
&\leq ( 1+ \sqrt{\rho_t} ) \left( 1-  \frac{2}{1 +  \rho_t} (1- \delta_{2r}) +  \frac{\rho_t - 1}{\rho_t +1} + \delta_{3r}  + \frac{\rho_t^{\frac{5}{4}}}{ \nu_t^{\frac12}\sigma_{\min}(\bm{X})}  \| \bm{X}_t - \bm{X}\|_{\mathcal{W}_t} \right) \| \bm{X}_t - \bm{X} \|_{\mathcal{W}_t}\\
&\leq ( 1+ \sqrt{\rho_t} ) \left( 1-  \frac{2}{1 +  \rho_t} (1- \delta_{2r}) +  \frac{\rho_t - 1}{\rho_t +1} + \delta_{3r}  + \frac{\rho_t^{\frac{7}{4}}}{ \sigma_{\min}(\bm{X})}  \| \bm{X}_t - \bm{X}\|_{F} \right) \| \bm{X}_t - \bm{X} \|_{\mathcal{W}_t}.
\end{aligned}
\end{equation}
Using Lemma \ref{lem3.1}, we further have
\begin{equation}\label{eq100}
\begin{aligned}
&\| \bm{X}_{t+1} - \bm{X}\|_{F}\\
&\leq  ( \rho_t + \sqrt{\rho_t} ) \left( 1-  \frac{2}{1 +  \rho_t} (1- \delta_{2r}) +  \frac{\rho_t - 1}{\rho_t +1} + \delta_{3r}  + \frac{\rho_t^{\frac{7}{4}}}{ \sigma_{\min}(\bm{X})}  \| \bm{X}_t - \bm{X}\|_{F} \right) \| \bm{X}_t - \bm{X} \|_{F}.
\end{aligned}
\end{equation}
Now we set $\epsilon_t  = \| \bm{G}_t \|_{\vee}^2$ and use Lemma \ref{lem3.1} to obtain $\rho_t = \sqrt{2}$. Substituting this into \eqref{eq100} and using the fact $\delta_{2r} \leq \delta_{3r}$, we can simplify \eqref{eq100} as
\begin{equation}\label{eq101}
\begin{aligned}
&\| \bm{X}_{t+1} - \bm{X}\|_{F}\leq \left( 0.895  + 5\delta_{3r} + \frac{2^{\frac{7}{8}}}{ \sigma_{\min}(\bm{X})}  \| \bm{X}_t - \bm{X}\|_{F} \right) \| \bm{X}_t - \bm{X} \|_{F}.
\end{aligned}
\end{equation}
In Algorithm \ref{PRGD}, we set the initial guess as $\bm{X}_0 = \mathcal{H}_{r}(\mathcal{A}^*(\bm{y}))$. Using the analysis in \cite{WCCL}, we have
\begin{equation}\label{FF00}
\| \bm{X}_0 - \bm{X} \|_{F} \leq 2 \delta_{2r} \| \bm{X} \|_{F} \leq 2 \delta_{3r} \sqrt{r} \sigma_{\max}(\bm{X}).
\end{equation}
 Denote
$$
\mu = 0.895 + 5\delta_{3r} + 10 \delta_{3r}\kappa \sqrt{r}.
$$
If $\mu < 1$, then we use induction to show that
\begin{equation}\label{eq200}
\| \bm{X}_{k+1} - \bm{X} \|_{F} \leq \mu \| \bm{X}_k - \bm{X} \|_{F},
\end{equation}
for all $k$.  Combining \eqref{eq101} with \eqref{FF00}, we obtain \eqref{eq200} holds for $k=0$. Suppose \eqref{eq200} holds for $k=0, 1, \dots t-1$. Then we have 
$$
\| \bm{X}_{t} - \bm{X} \|_{F} \leq \mu^t \| \bm{X}_0 - \bm{X} \|_{F} \leq  2 \delta_{3r} \sqrt{r} \sigma_{\max}(\bm{X}).
$$
By substituting this into \eqref{eq101}, we get \eqref{eq200} for $k=t$. Therefore, the induction on $k$ gives that \eqref{eq200} holds for any $k$.
 Furthermore, if
$$
\delta_{3r} \leq \frac{1}{50 + 100 \sqrt{r} \kappa},
$$
then $\mu \leq 0.995$. This completes the proof.
\end{proof}

\section{ Numerical experiments}
\label{sec:numerical} 

In this section, we apply the PRGD algorithm, as described in \cref{PRGD}, to solve low-rank matrix recovery problems such as low-rank matrix completion, low-rank matrix sensing, and phase retrieval. We compare the performance of PRGD with other state-of-the-art algorithms, including normalized iterative hard thresholding (NIHT) \cite{TW}, Riemannian gradient descent (RGD) with the canonical metric \cite{CWW, WCCL,V, CW-phase, LMC, TMC}, and the RGD with Shampoo preconditioner \cite{GKS} defined in \eqref{eq:ShamPrecon}. Our experimental results show that PRGD is highly efficient compared with other methods. We conducted all experiments using MATLAB R2020b on a desktop computer equipped with a 2.4GHz i5-10600T CPU and 16GB memory. The singular value decomposition (SVD) involved in the experiments is conducted using PROPACK in MATLAB R2020b. 

\subsection{Low-Rank Matrix Completion}
We begin our investigation of PRGD's performance by assessing its efficiency in solving the low-rank matrix completion problem. In this problem, the measurement matrices in \eqref{eq:defA} take the form of $\bm{A}_i=\bm{e}_{r_i}\bm{e}_{c_i}^T$ for all $(r_i,c_i)\in\Omega$, where $\Omega$ is a subset of $[n_1]\times [n_2]$ with $|\Omega|=m$. The goal of matrix completion is to recover the low-rank matrix $\bm{X}$ from its partially observed entries on $\Omega$.

\subsubsection{Simulated data}
\label{4.1}
We test the PRGD algorithm on simulated data. To generate the rank-$r$ matrix $\bm{X}$ of size $n_1\times n_2$, we sample the entries of $\bm{X}_L \in \mathbb{R}^{n_1 \times r}$ and $\bm{X}_R \in \mathbb{R}^{n_2 \times r}$  from a uniform distribution in $[0, 1]$. We then set $\bm{X} = \bm{X}_L \bm{X}_R^{T}$. The set of observed entries $\Omega$ is sampled uniformly at random among all subsets of cardinality $m$. We define the sampling ratio $p$ as $p = \frac{m}{n_1n_2}$ and the oversampling ratio as
\begin{equation}\label{sp}
q = \frac{(n_1 + n_2 - r)r}{m}.
\end{equation}
Given fixed parameters $(n_1, n_2, m, r)$, we obtain a fixed $p$ and $q$. For each experiment, we conduct $5$ random tests with the same $(n_1, n_2, m, r)$ and report the average result for comparison. 

\paragraph{Sensitivity to the Size of Matrix.} We first evaluate the sensitivity of PRGD to changes in matrix size. For this purpose, we conduct tests using varying matrix sizes $(n_1,n_2)$ and different combinations of $r$ and $q$.  We compare the CPU time (in seconds) and the number of iterations required by different algorithms to obtain an $\bm{X}_t$ satisfying $\| \bm{X} - \bm{X}_t \|_F \leq 10^{-4} \cdot \| \bm{X} \|_F$.

We present the results in Figure \ref{res1}.  Figure \ref{res1} indicates that PRGD generally requires significantly fewer iterations than all other algorithms. Moreover, as the matrix size increases, the number of iterations required by the other three algorithms increases noticeably, while the iterations of PRGD exhibit only a slight increase. Furthermore, as analyzed in Section \ref{sec:comPRGD}, each iteration of PRGD takes almost the same computation time as RGD and Shampoo, but much less than NIHT. As a result, the total computation time of PRGD is much less than that of the other three algorithms. Indeed, we can also see from Figure \ref{res1} that when the matrix size increases, the computation time required for the other three algorithms increases significantly. However, the computational time required for PRGD exhibits only a slow increase with increasing matrix size. Notably, when the matrix size is $20000$, the computation time of PRGD is significantly lower than the other three algorithms. Specifically, the computational time required for PRGD is only about $1/20$, $1/10$, and $1/5$ times that of NIHT, RGD, and shampoo, respectively. All these results suggest that choosing a suitable metric is an effective technique to accelerate RGD, and our data-driven metric in PRGD is much better than the metric used in Shampoo.

\begin{figure}[htbp]
	\centering
	\subfloat[{\tt CPU-time(s) $r=10$ and $q=5$}]{ \includegraphics[width=0.45\linewidth]{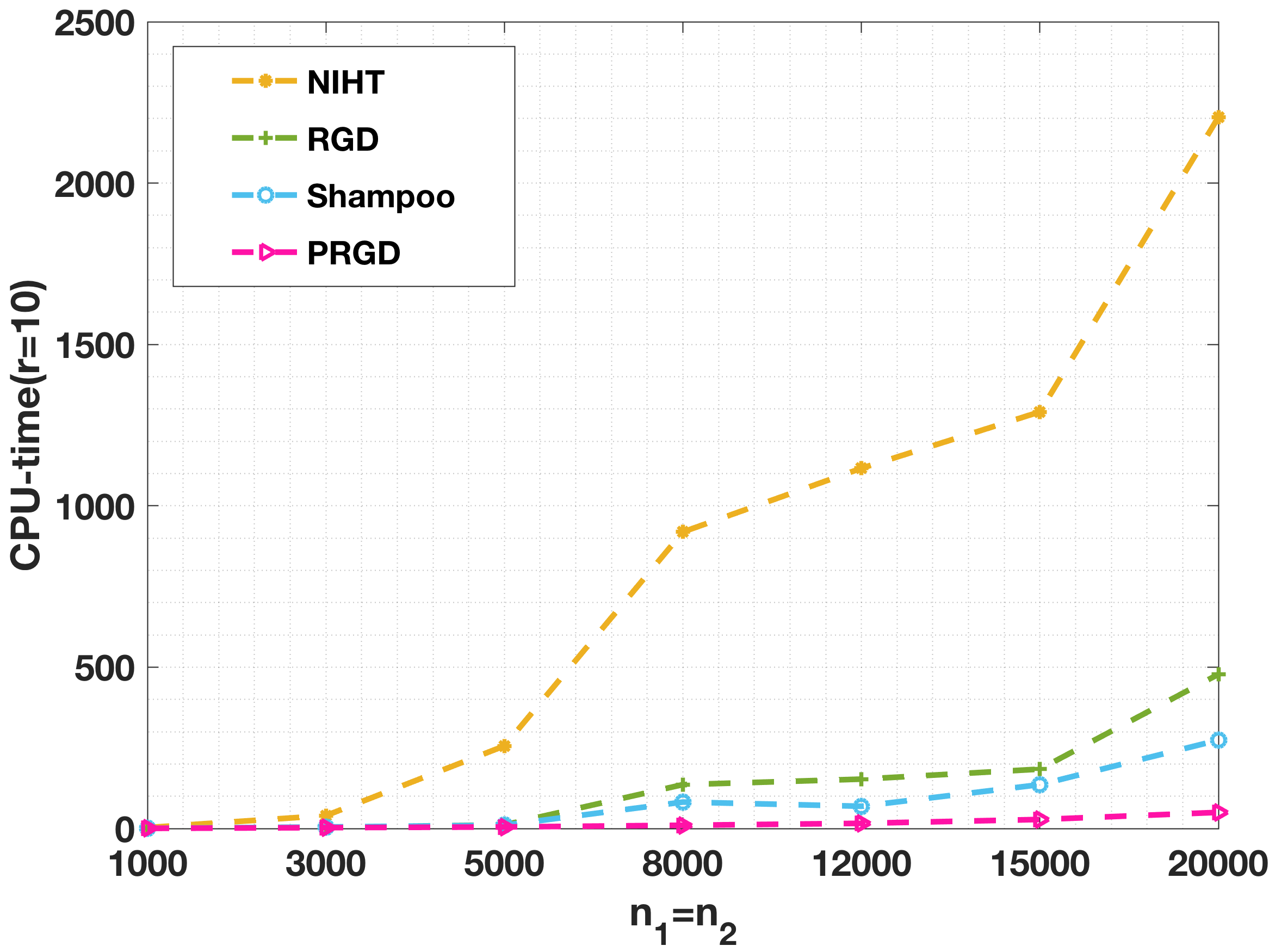} } 		
	\subfloat[{\tt $\sharp$iteration $r=10$ and $q=5$}]{ \includegraphics[width=0.45\linewidth]{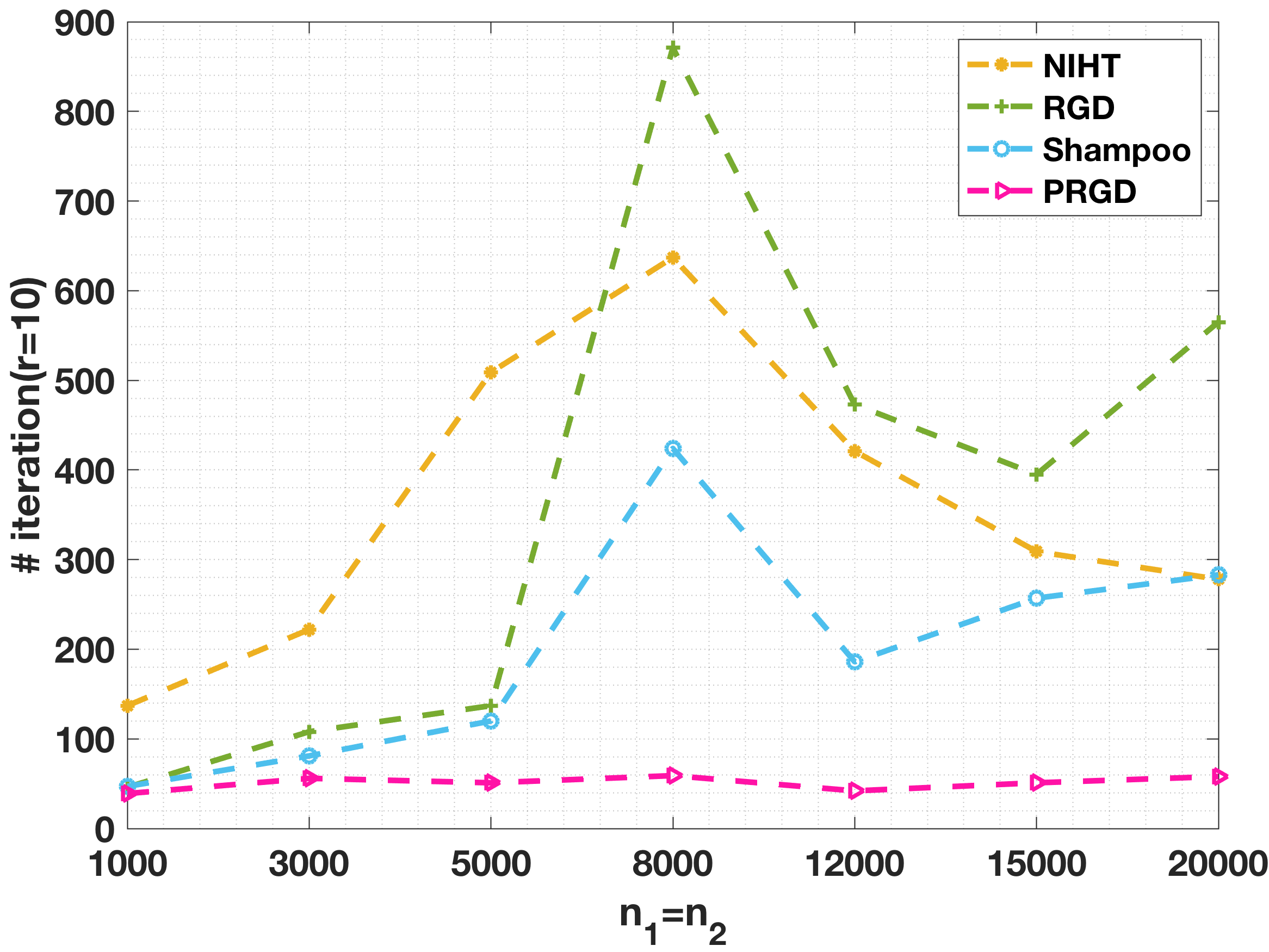} } 		\\[-2mm]	
	\subfloat[{\tt CPU-time(s) $r=50$ and $q=6$}]{ \includegraphics[width=0.45\linewidth]{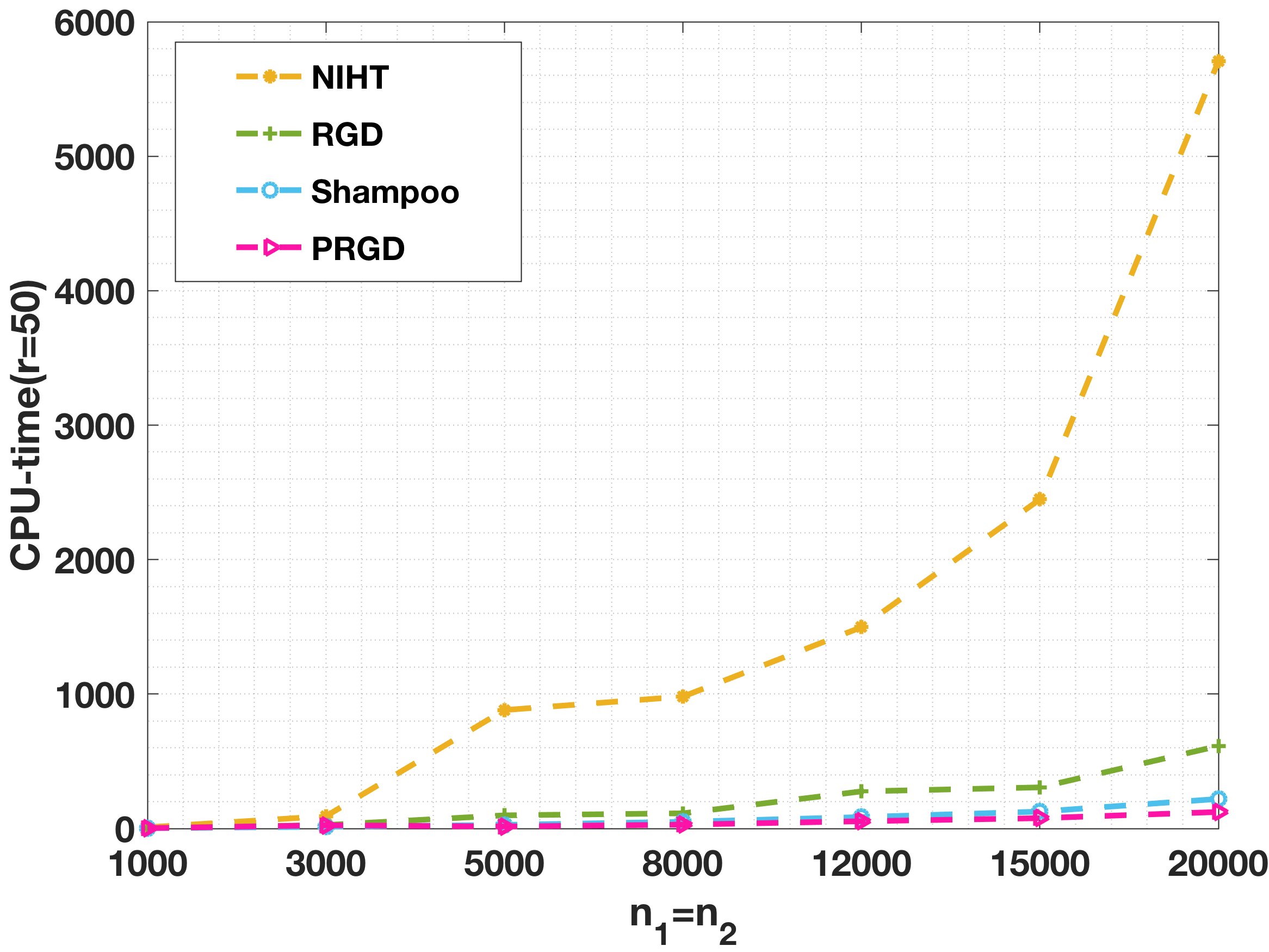} } 
	\subfloat[{\tt $\sharp$iteration $r=50$ and $q=6$}]{ \includegraphics[width=0.45\linewidth]{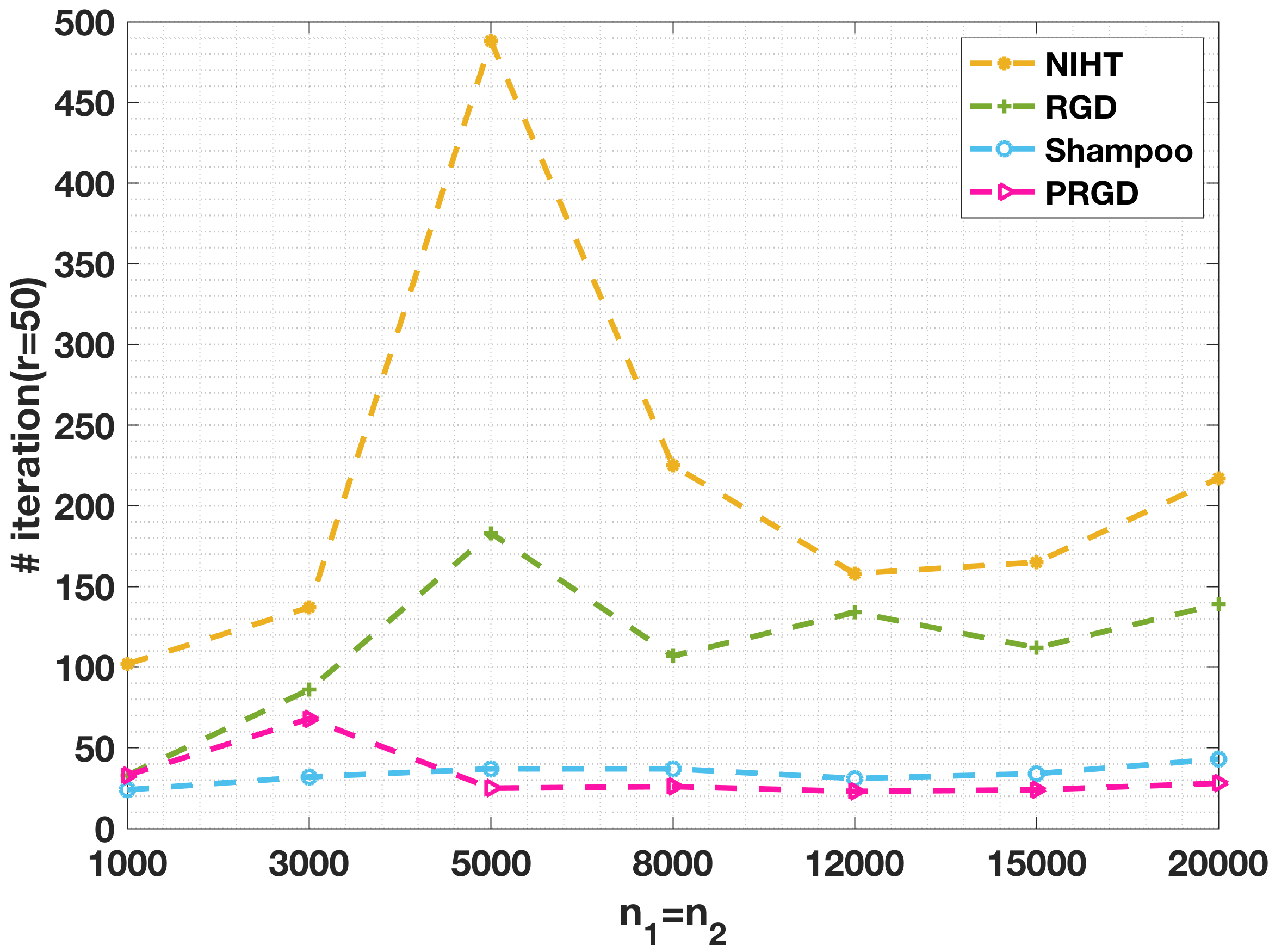} }         \\[-2mm]	
	\subfloat[{\tt CPU-time(s) $r=100$ and $q=7$}]{ \includegraphics[width=0.45\linewidth]{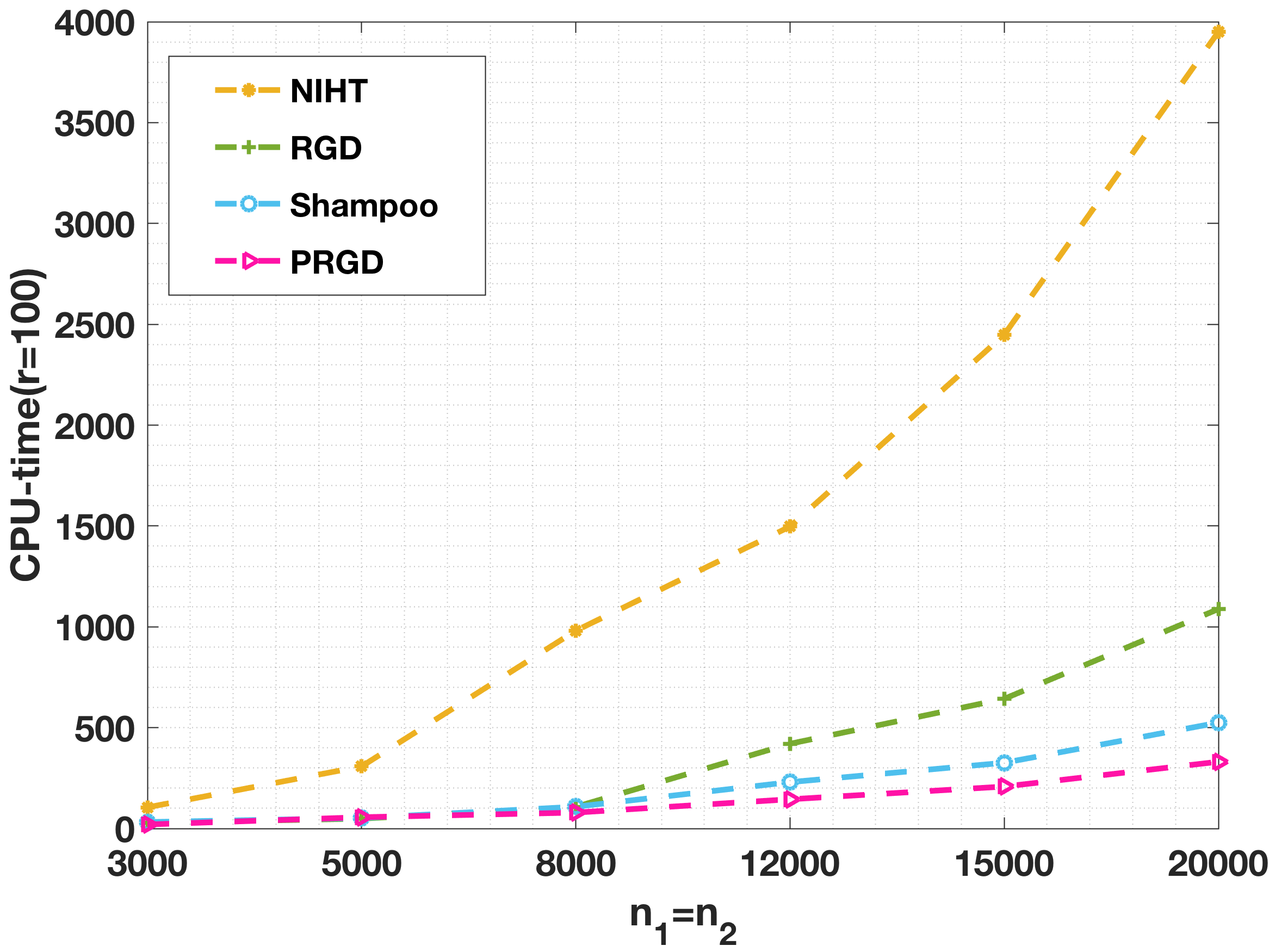} } 
	\subfloat[{\tt $\sharp$iteration $r=100$ and $q=7$}]{ \includegraphics[width=0.45\linewidth]{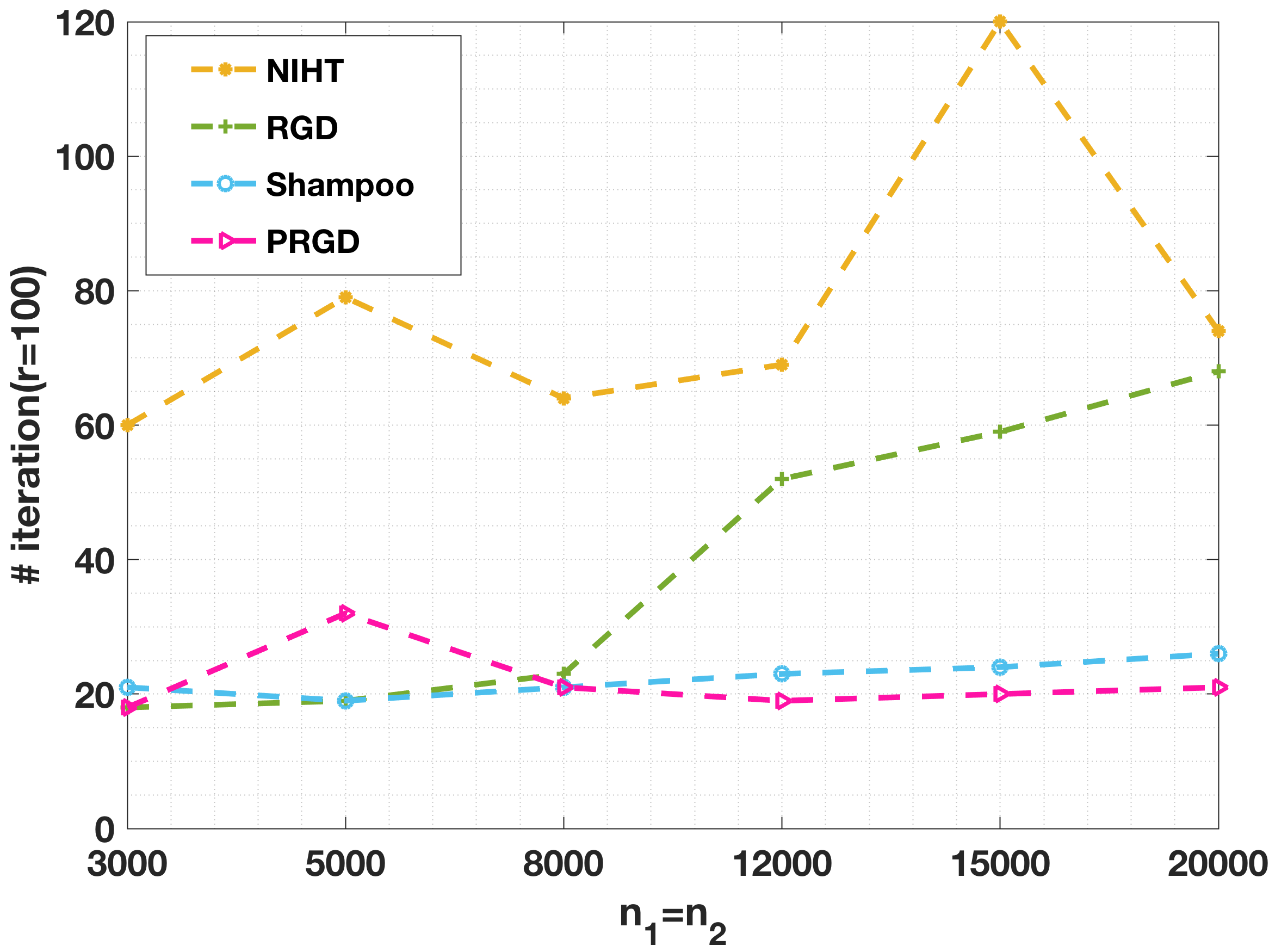} }         \\[-2mm]	
	\caption{ Results of CPU-time (in seconds) and the number of iterations for matrix completion on simulated data.}
    \label{res1}
\end{figure}

\paragraph{Sensitivity to Oversampling Ratio.} We investigate the sensitivity of PRGD to changes in the oversampling ratio $q$. As defined in \eqref{sp}, the oversampling ratio $q$ is the rate between the number of sampled entries and the ``true dimensionality" of an $n_1 \times n_2$ matrix of rank $r$. A smaller oversampling ratio implies a more challenging matrix completion problem since there are fewer observed entries available to estimate the unknown entries. We perform tests on matrices of size $10000 \times 10000$ with rank $10$ and vary the oversampling ratio in the set $\{1/3, 1/4, 1/5, 1/6, 1/7, 1/8, 1/9, 1/10, 1/11, 1/12, 1/13, 1/14, 1/15\}$. The results of the number of iterations and computational time to find an $\bm{X}_t$ satisfying $\| \bm{X} - \bm{X}_t \|_F \leq 10^{-4} \cdot \| \bm{X} \|_F$ are presented in Figure \ref{res2}. 

The results in Figure \ref{res2} demonstrate that for all cases, the PRGD algorithm outperforms all other algorithms in terms of both computational time and the number of iterations required. Moreover, as the oversampling ratio $q$ increases, i.e., the problem becomes more challenging, the advantage of PRGD in terms of the number of iterations and computational time becomes more pronounced. For instance, when the oversampling ratio $q$ is set to $1/3$, indicating a highly challenging matrix completion problem, the PRGD algorithm requires only about $1/10$ of the number of iterations and $1/8$ of the computational time required by RGD. These observations further confirm that the PRGD algorithm we constructed is highly efficient and outperforms RGD, Shampoo, and NIHT.

\begin{figure}[htbp]
	\centering
	\subfloat[{\tt CPU-time(s) }]{ \includegraphics[width=0.45\linewidth]{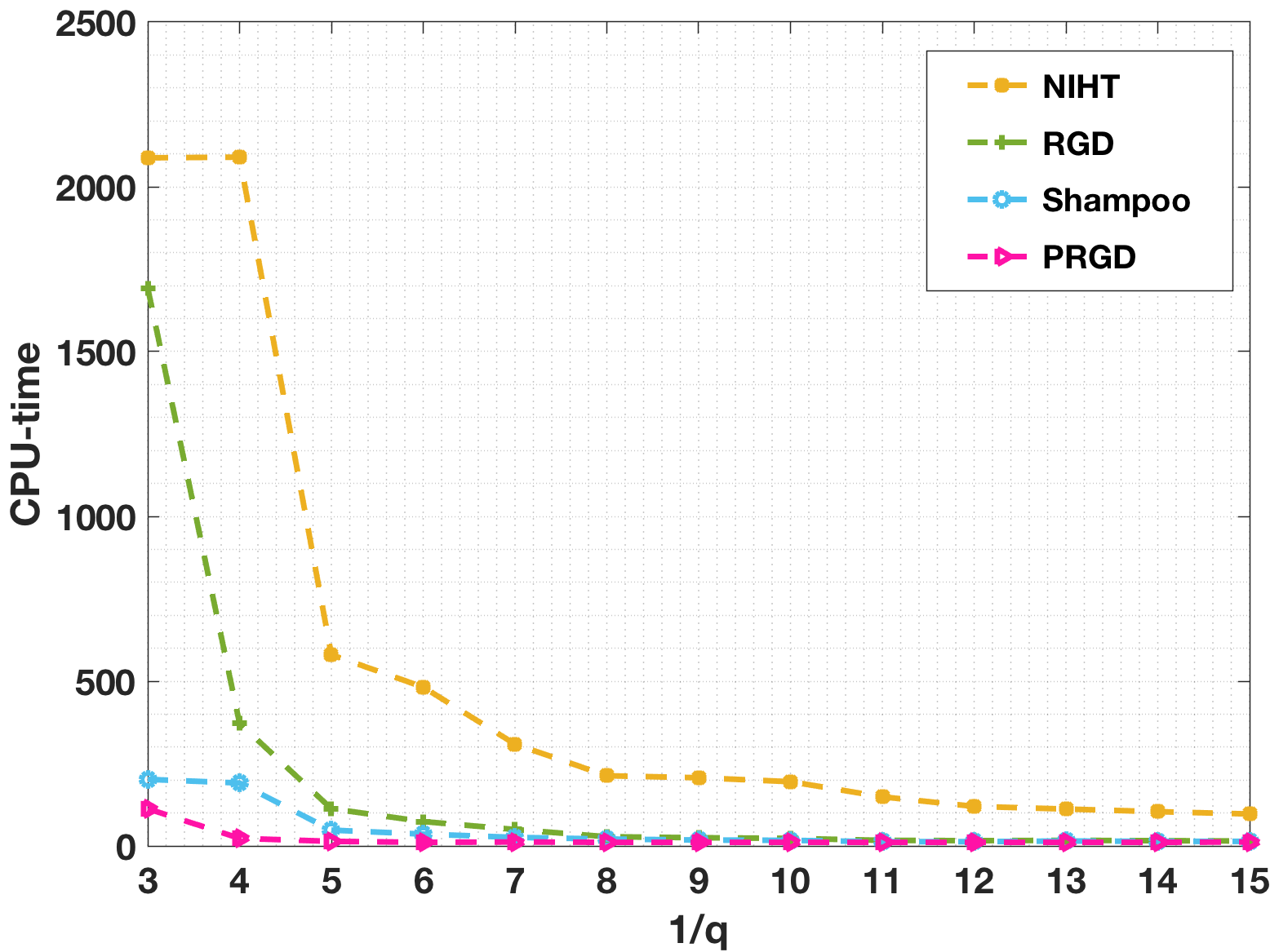} } 		
	\subfloat[{\tt The number of iterations }]{ \includegraphics[width=0.45\linewidth]{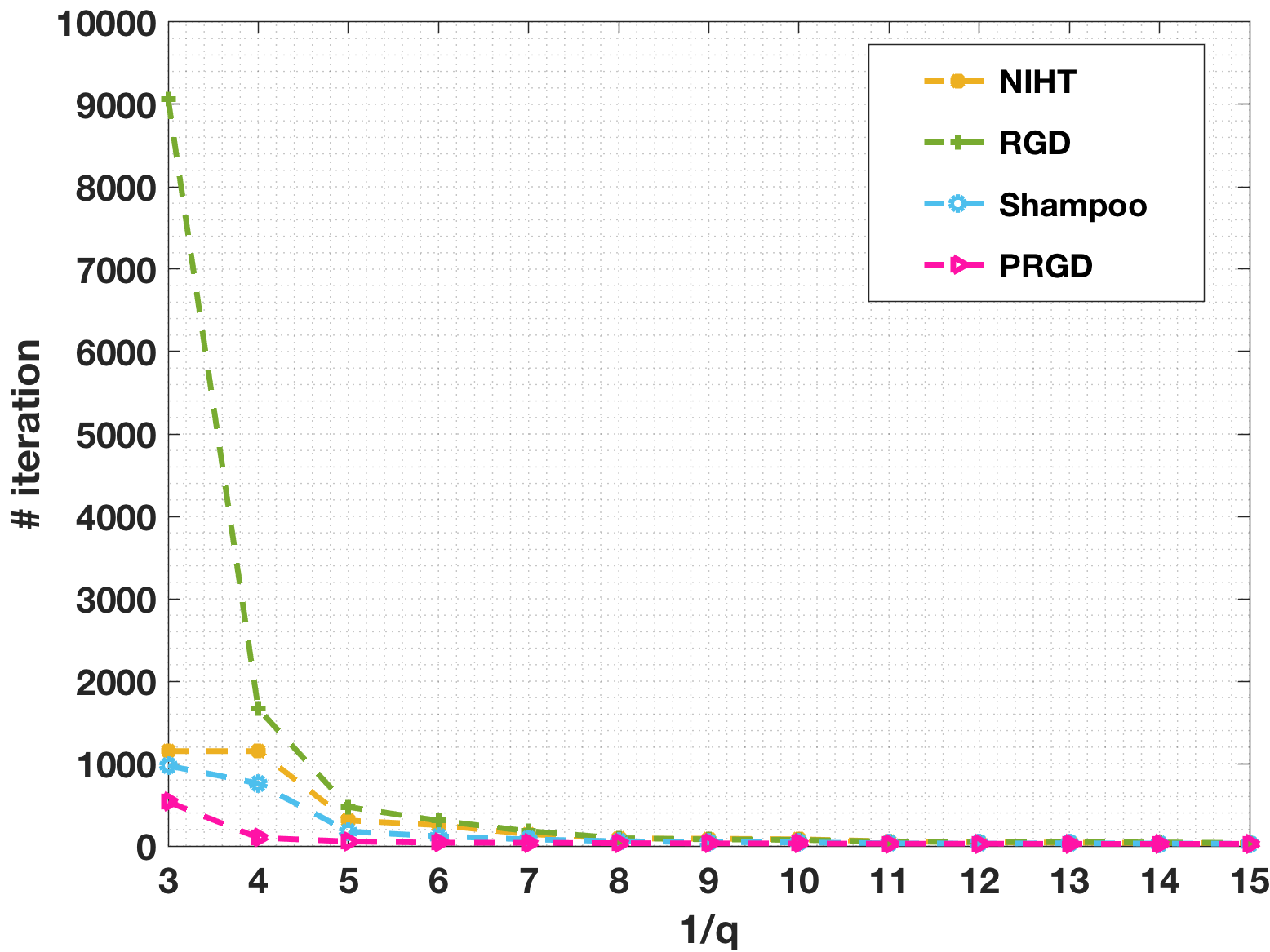} } 		\\[-2mm]	
	\subfloat[{\tt Zoom in of (a)}]{ \includegraphics[width=0.45\linewidth]{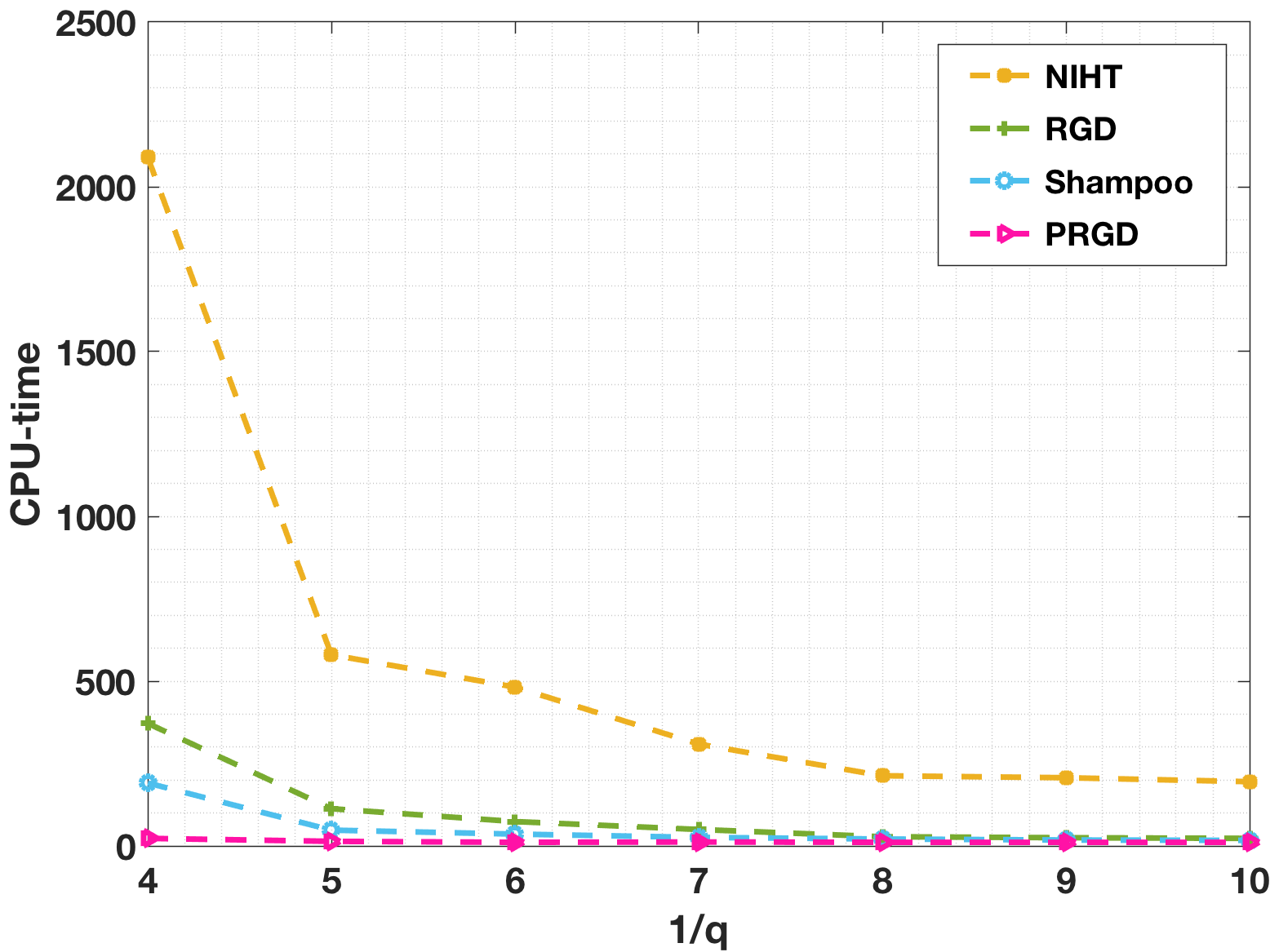} } 
	\subfloat[{\tt Zoom in of (b) }]{ \includegraphics[width=0.45\linewidth]{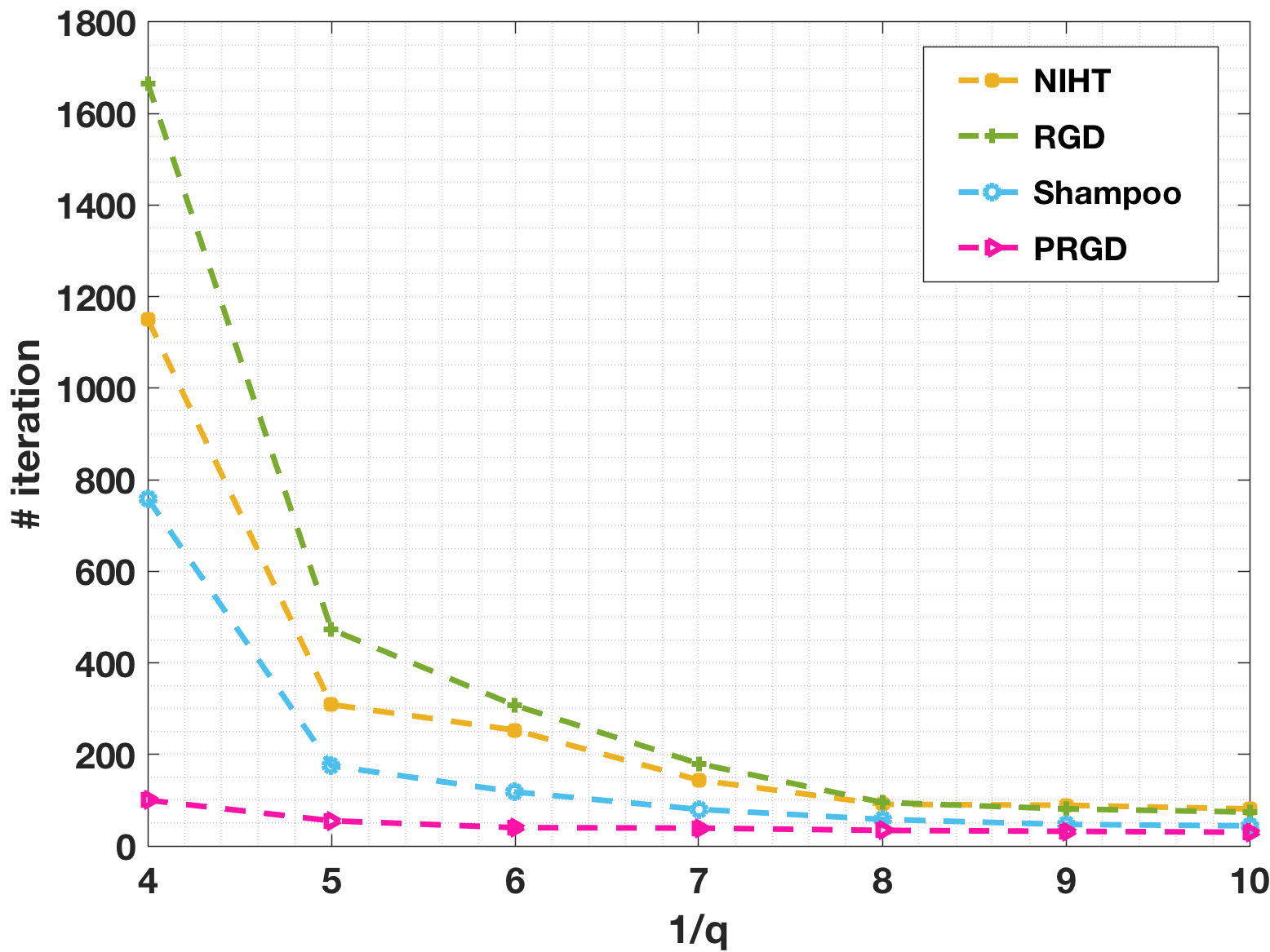} }         \\[-2mm]	
	\caption{ Results of CPU-time (in seconds) and the number of iterations for matrix completion on simulated data. The unknown matrix has size $10000 \times 10000$ and rank $r=10$.}
    \label{res2}
\end{figure}

\begin{figure}[htbp]
	\centering
	\includegraphics[width=0.45\linewidth]{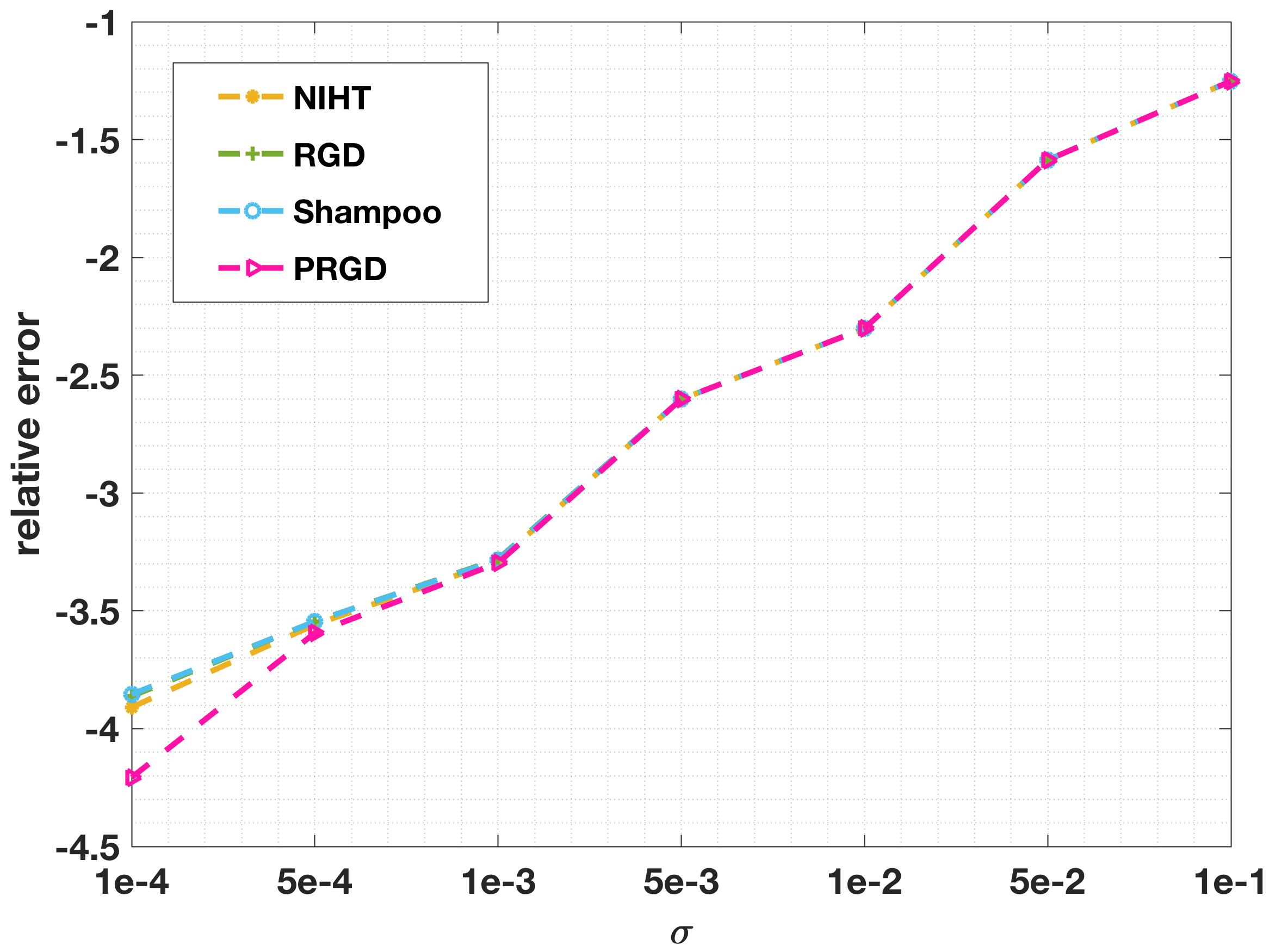} 	
   \caption{ Results of relative error of the recovered matrix for matrix completion on simulated data with noise. The unknown matrix has size $10000 \times 10000$ and rank $r=10$.}
    \label{robust}
\end{figure}

\paragraph{ Robustness to Additive Noise.}
To evaluate the robustness of the algorithms to additive noise,  we add different noise to the sampled entries using vector
\begin{equation}
e = \sigma \cdot \| \mathcal{P}_{\Omega}(\bm{X}) \|_{F} \cdot \frac{\bm{w}}{\| \bm{w} \|_2},
\end{equation}
where the entries of $\bm{w}$ are $i.i.d$ Gaussian random variables and $\sigma$ is referred to as noise level. Actually, $\sigma$ is also the relative error of observed data. We conduct tests with $8$ different values of $\sigma$ from $1\times 10^{-4}$ to $1 \times 10^{-1}$. We test the $10000 \times 10000$ matrix of rank $r = 10$ and fix the oversampling ratio $q=5$. We stop all algorithms when $\| \bm{X}_{t+1} - \bm{X}_{t}\|_{F} \leq 10^{-5}\cdot \max\{1, \| \bm{X}_t\|_{F} \}$. We present the results in Figures \ref{robust} and \ref{res3}. Figure \ref{robust} shows that all algorithms are robust to additive noise in the sense that the error of the recovered matrix is only proportional to the noise level. Moreover, PRGD is faster than the other three algorithms, as evident from Figure \ref{res3}. When the noise level becomes higher, the efficiency advantage is more noticeable. For example, when $\sigma= 0.1$, the PRGD algorithm requires only about $1/10$ of the number of iterations and $1/10$ of the computational time required by RGD.

\begin{figure}[htbp]
	\centering
	\subfloat[{\tt CPU-time(s)}]{ \includegraphics[width=0.43\linewidth]{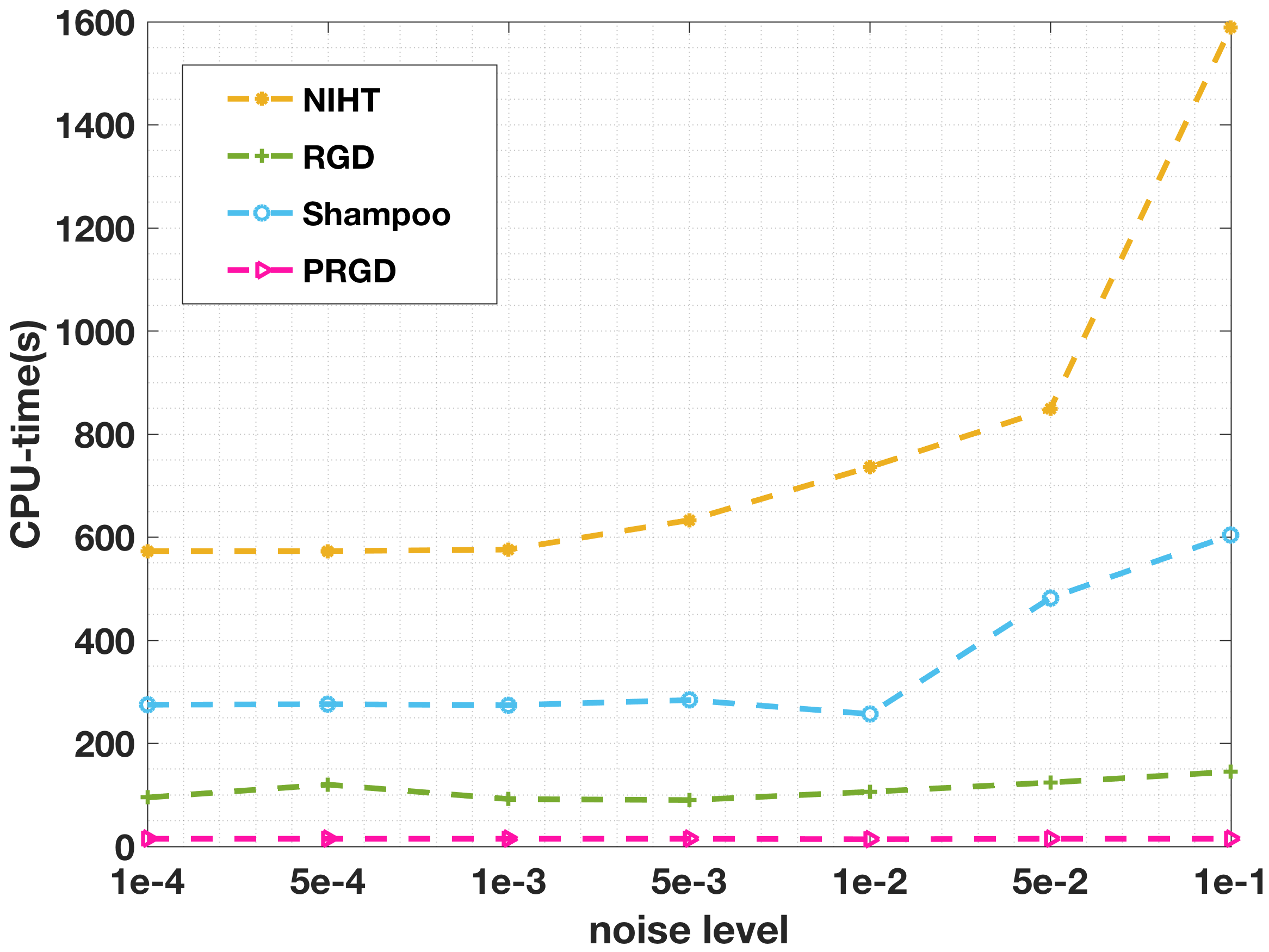} } 		
	\subfloat[{\tt The number of iterations }]{ \includegraphics[width=0.43\linewidth]{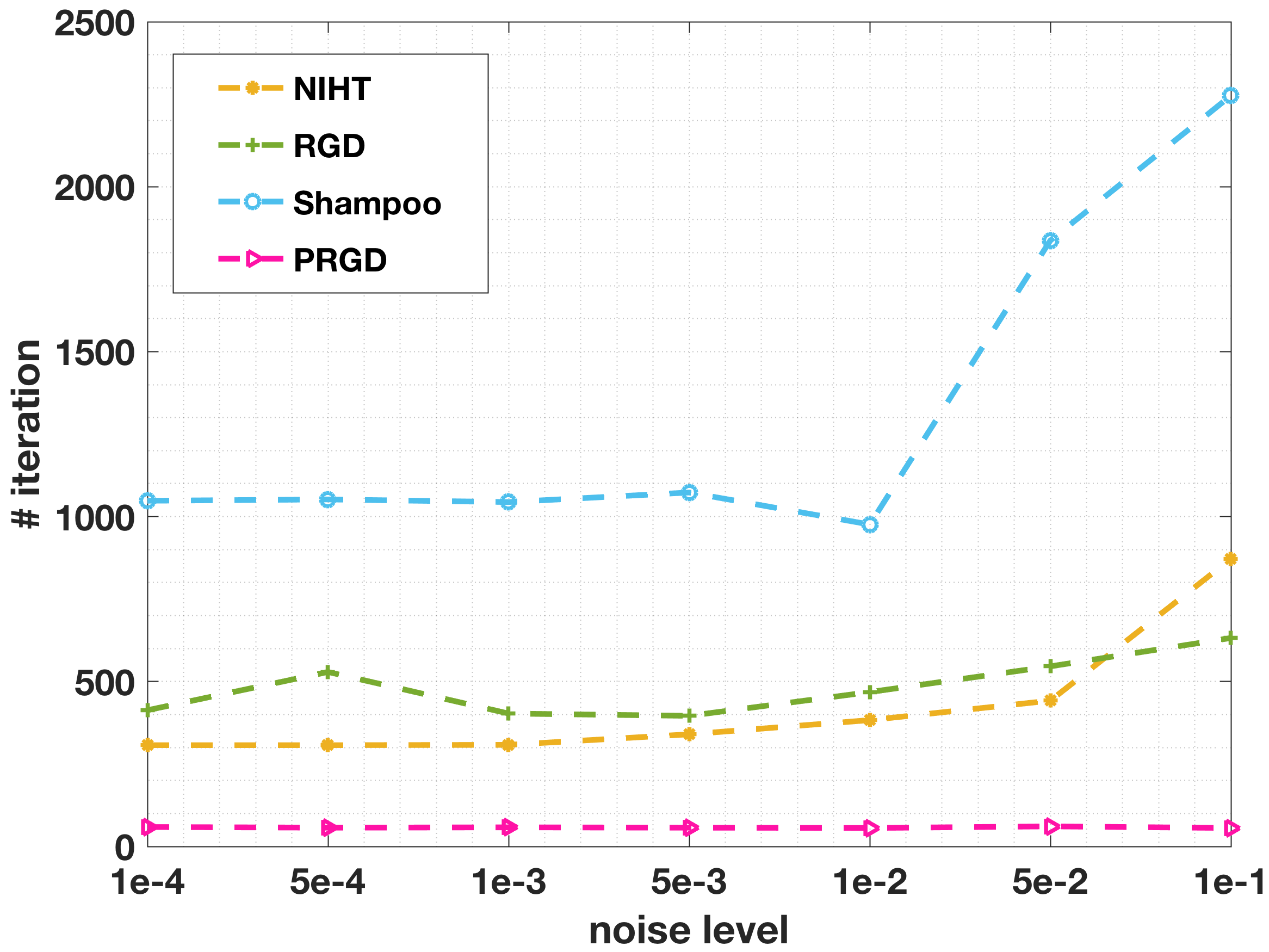} } 		\\[-2mm]	
	\caption{ Results of CPU-time (in seconds) and the number of iterations for matrix completion on simulated data with noise. The unknown matrix has size $10000 \times 10000$ and rank $r=10$.}
    \label{res3}
\end{figure}

\begin{figure}[htbp]
	\centering
	\subfloat[{\tt r=5 }]{ \includegraphics[width=0.43\linewidth]{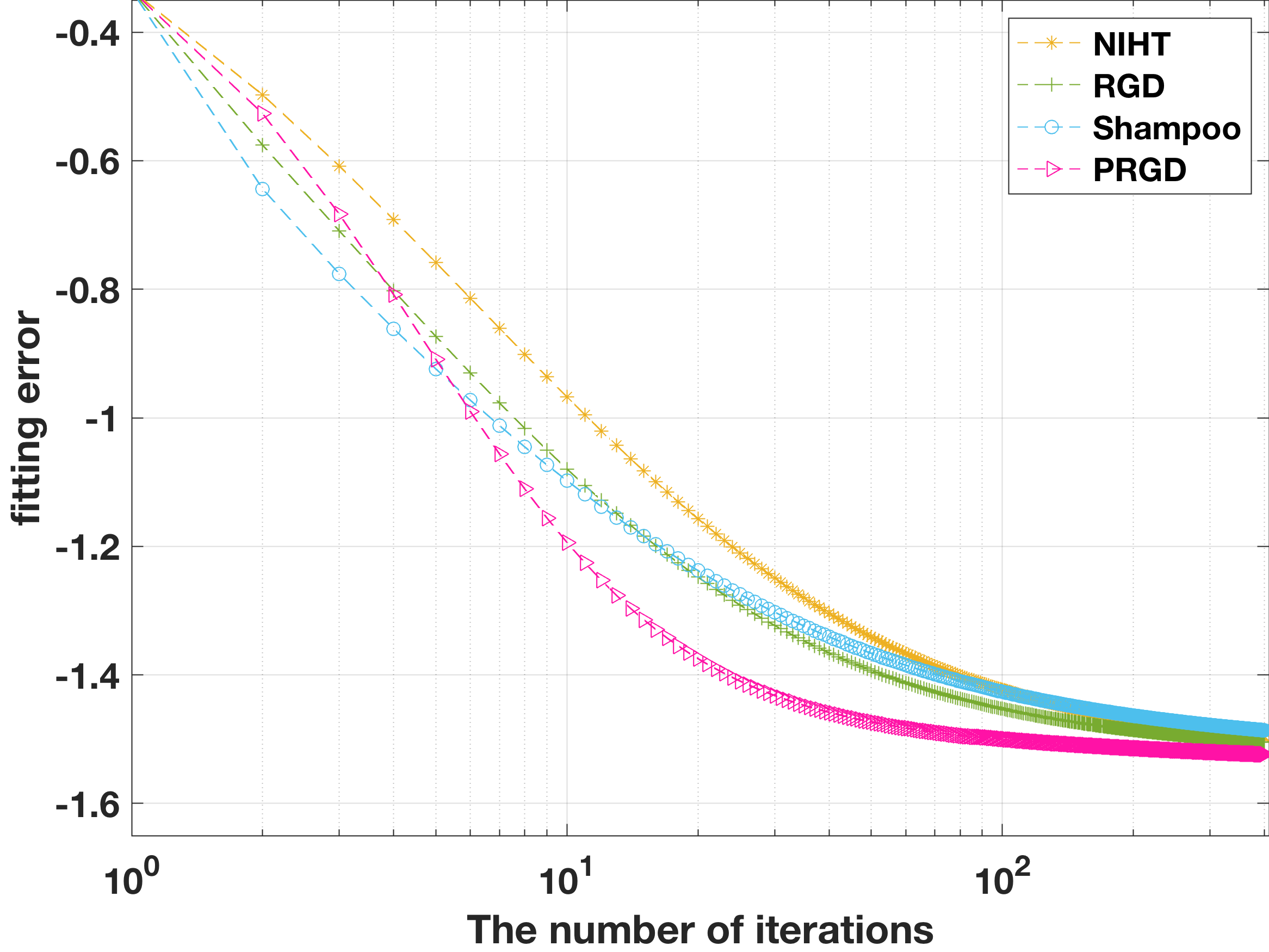} } 		
	\subfloat[{\tt r=10 }]{ \includegraphics[width=0.43\linewidth]{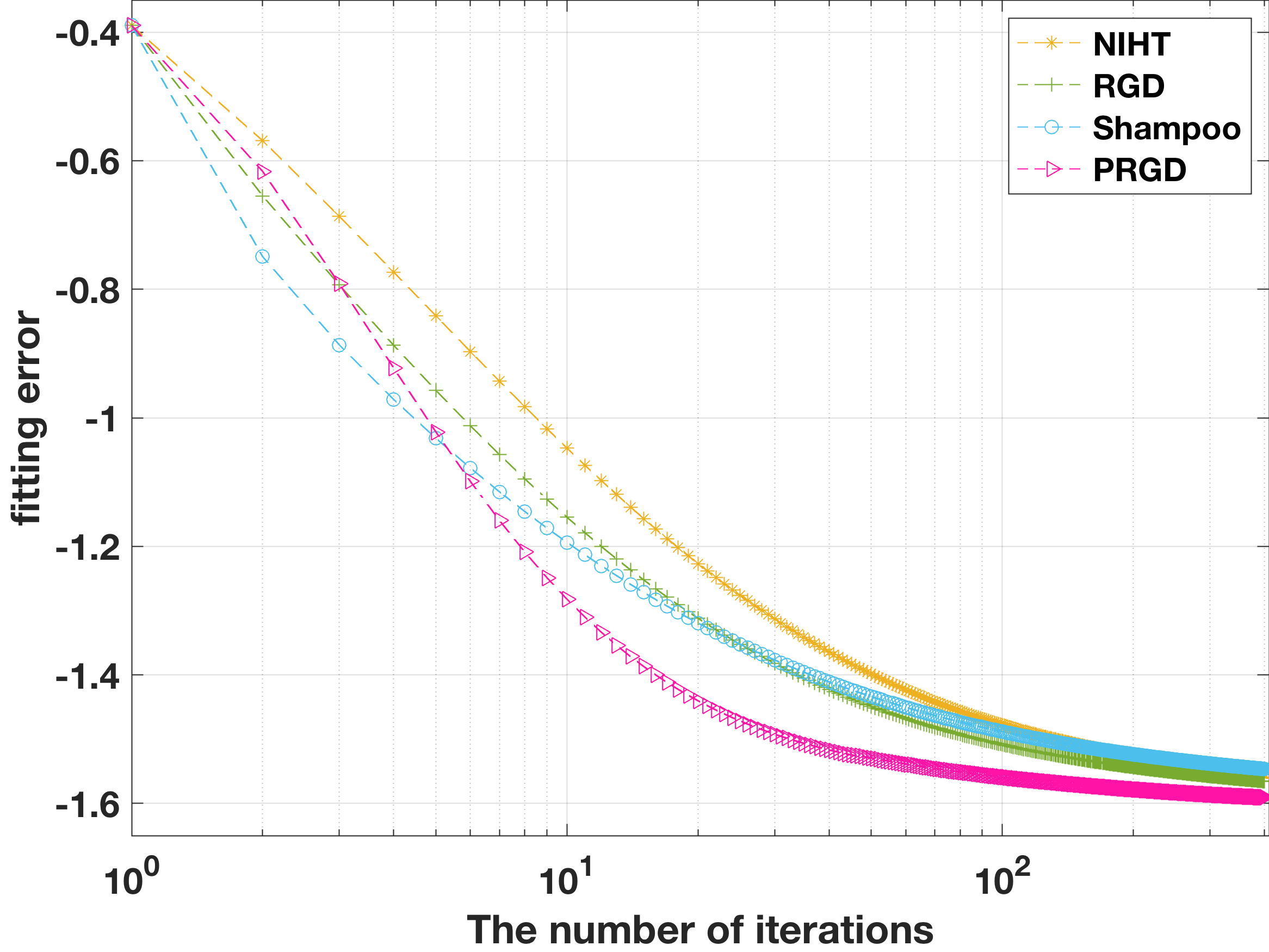} } 		\\[-2mm]	
	\caption{ Results of the fitting error for matrix completion on real data. We set the rank of the unknown matrix to be $r=5$ and $r=10$, respectively.}
    \label{res4}
\end{figure}

\subsubsection{Real Data}
We evaluate the performance of the PRGD algorithm on the Movielens dataset, a real dataset commonly used in matrix completion problems. Specifically, we use the Movielens 1M dataset \cite{HK}, which contains anonymous ratings for 3,952 movies from 6,040 people who joined Movielens in 2000. Since  the ground truth and its rank of this real dataset are not known, we set $r=5$ and $r=10$, respectively, and compare the fitting error
$$
\frac{\| \mathcal{P}_{\Omega}(\bm{X}-\bm{X}_t)\|_{F}}{\| \mathcal{P}_{\Omega}(\bm{X})\|_{F}}
$$
 for all algorithms by running the same number of iterations. We plot the fitting error versus the number of iterations in Figure \ref{res4}. From the figure, we observe that PRGD is the most efficient algorithm for fitting the Movielens dataset with a low-rank matrix. Specifically, PRGD outputs $\bm{X}_t$ with the smallest fitting error among all algorithms for all $t$, and it requires the fewest iterations to achieve the same fitting error. Notably, PRGD requires half the number of iterations as RGD to reduce the fitting error to $10^{-1.4}$.

\subsection{Low-Rank Matrix Sensing}
In this section, we further evaluate the ability of PRGD to recover a low-rank matrix in the matrix sensing problem, where the measurement matrices $\bm{A}_i, i=1,2, \dots, m,$ in the linear operator $\mathcal{A}$ defined in \eqref{eq:defA} are general. In this experiment, we generate the entries of $\bm{A}_i$ using i.i.d. standard Gaussian random variables. The underlying matrix $\bm{X}$ is synthesized as $\bm{X} = \bm{X}_{\bm L} \bm{X}_{\bm R}^{T}$, where  $\bm{X}_{\bm L} \in \mathbb{R}^{n_1 \times r}$ and $\bm{X}_{\bm R} \in \mathbb{R}^{n_2 \times r}$ have i.i.d. standard Gaussian entries.
The measurement vector $\bm{y} \in \mathbb{R}^{m}$ is generated as $\bm{y} = \mathcal{A} \bm{X}$. We define the sampling ratio as $p = \frac{m}{n_1n_2}$ and the oversampling ratio $q$ as defined in \eqref{sp}. 

In this experiment, RGD with a constant step size performs poorly and requires many iterations. Therefore, we do not include a comparison of RGD, but instead compare RGD with an adaptive steepest descent step size (called adaptive-RGD) and PRGD with an adaptive steepest descent step size (called adaptive-PRGD).

\begin{figure}[htbp]
	\centering
	\subfloat[{\tt  The number of iterations}]{ \includegraphics[width=0.45\linewidth]{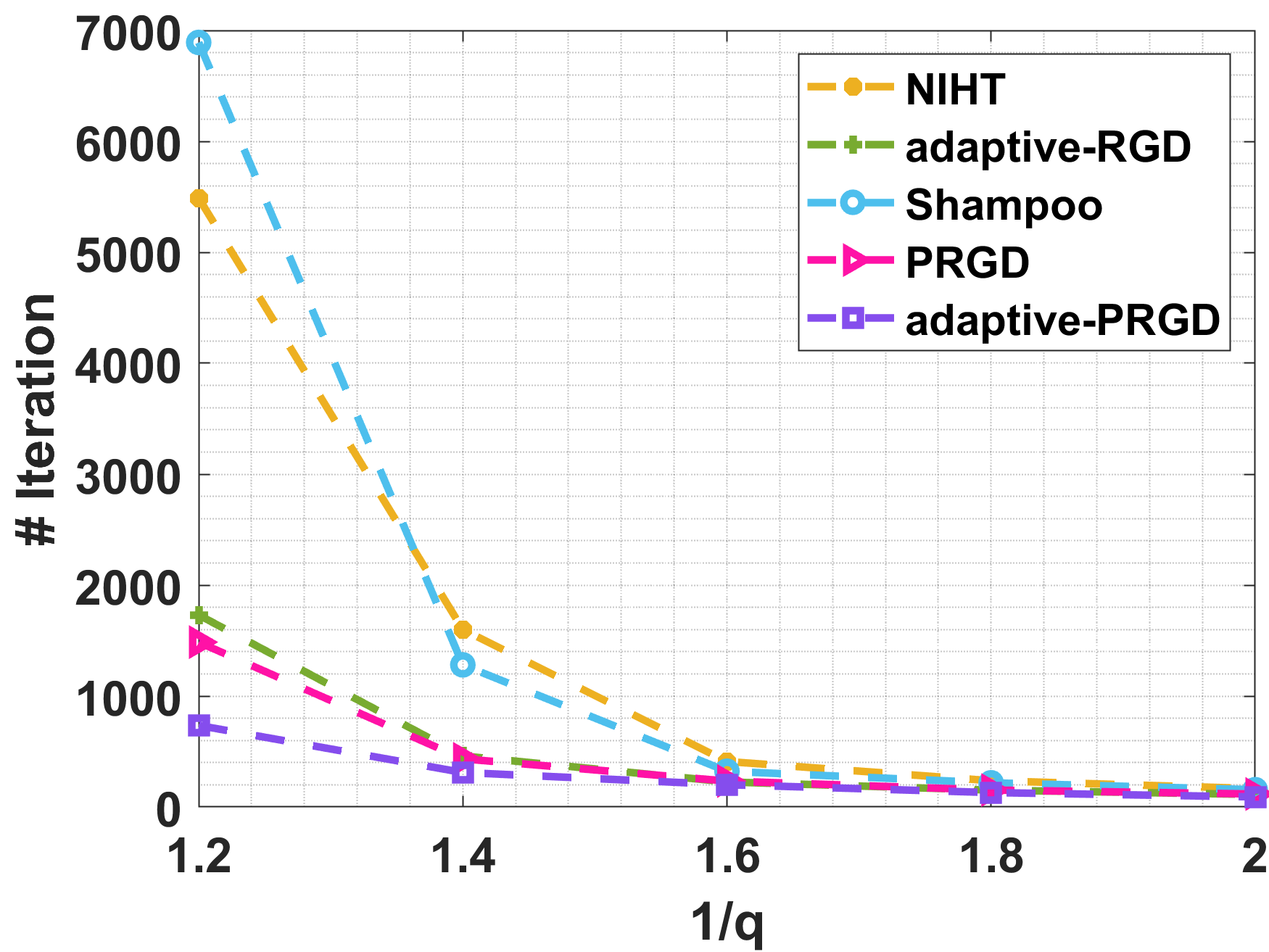} } 	
	\subfloat[{\tt CPU-time(s)}]{ \includegraphics[width=0.45\linewidth]{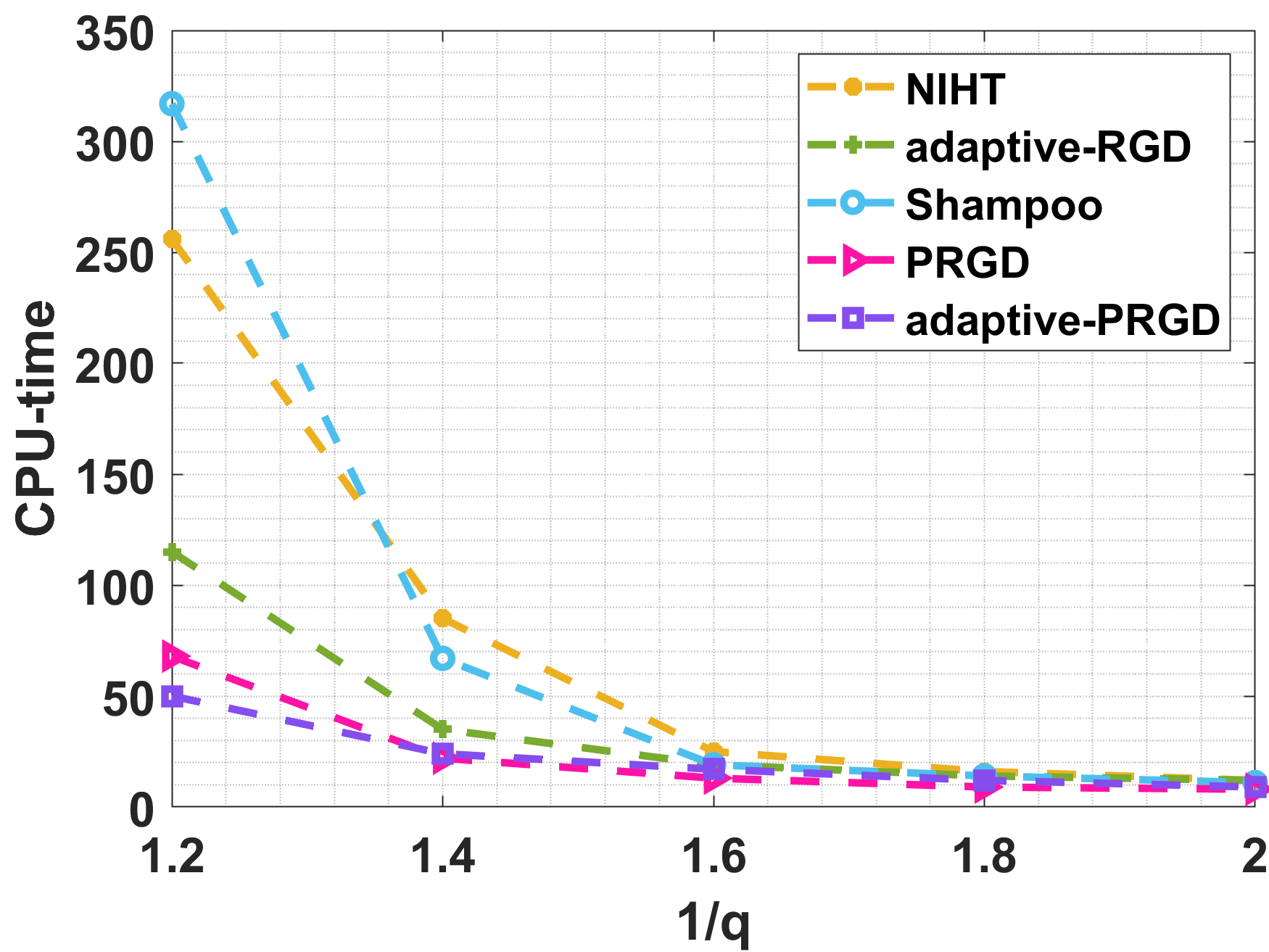} }  	

	\caption{Results of CPU-time (in seconds) and the number of iterations for matrix sensing on simulated data. The unknown matrix has size $150 \times 150$ and rank $r=10$.}
    \label{lmr_1}
\end{figure}

\paragraph{Sensitivity to Oversampling Ratio.} We investigate the sensitivity of PRGD to changes in the over-sampling ratio $q$, where a larger value of $q$ implies a more challenging problem. We fix the size of the measurement matrix at $n_1 =n_2 = 150$ and its rank at $r=10$, while varying the oversampling ratio $1/q \in \{1.2, 1.4, 1.6, 1.8, 2.0 \}$. The results for the number of iterations and computational time required to obtain an $\bm{X}_t$ satisfying $\|\bm{X}-\bm{X}_t\|_F\leq 10^{-4}\cdot \|\bm{X}\|_F$ are shown in Figure \ref{lmr_1}.

From Figure \ref{lmr_1}(a), it can be seen that adaptive-PRGD outperforms the other four algorithms in terms of the number of iterations. PRGD has a comparable number of iterations to adaptive-RGD. Specifically, when the oversampling ratio $1/q=1.2$, the number of iterations of adaptive-PRGD is about $1/2.35$, $1/7.5$, $1/9.4$, and $1/2$ of that of adaptive-RGD, NIHT, Shampoo, and PRGD, respectively. From Figure \ref{lmr_1}(b), it can be observed that the computational time of adaptive-PRGD is always less than that of adaptive-RGD, Shampoo, NIHT, and PRGD. When the oversampling ratio $1/q=1.2$, the computational time of adaptive-PRGD is about $1/2.3$, $1/5.1$, $1/6.3$, and $1/1.3$ of that of adaptive-RGD,  NIHT, Shampoo, and PRGD, respectively. The computational time of PRGD is less than that of adaptive-RGD. This is because although adaptive-RGD and PRGD have a similar number of iterations, adaptive-RGD requires additional time to compute the steepest descent step size.
\begin{figure}[htbp]
	\centering
	\subfloat[{\tt adaptive-RGD}]{ \includegraphics[width=0.3\linewidth]{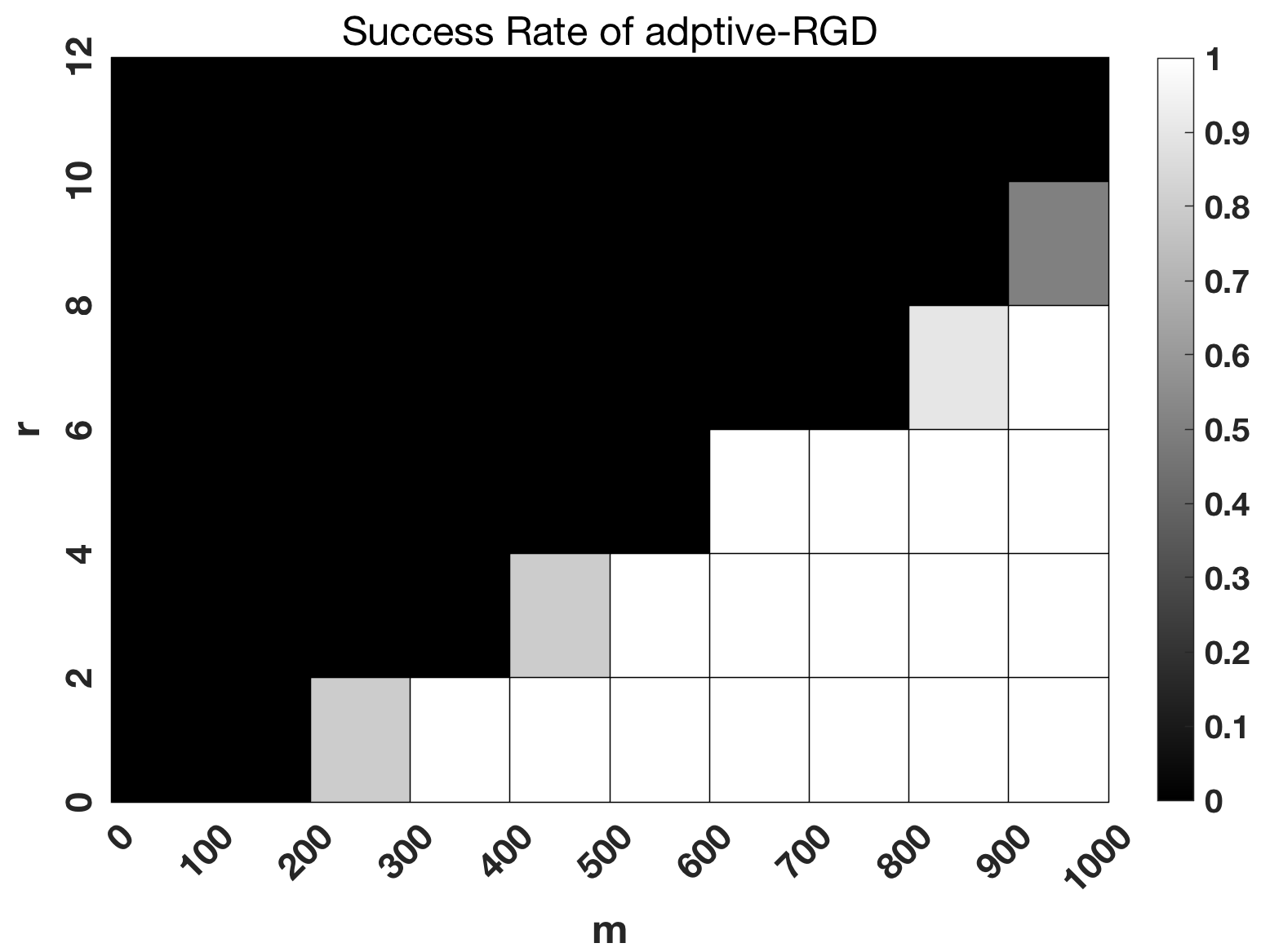} } 	
	\subfloat[{\tt PRGD }]{ \includegraphics[width=0.3\linewidth]{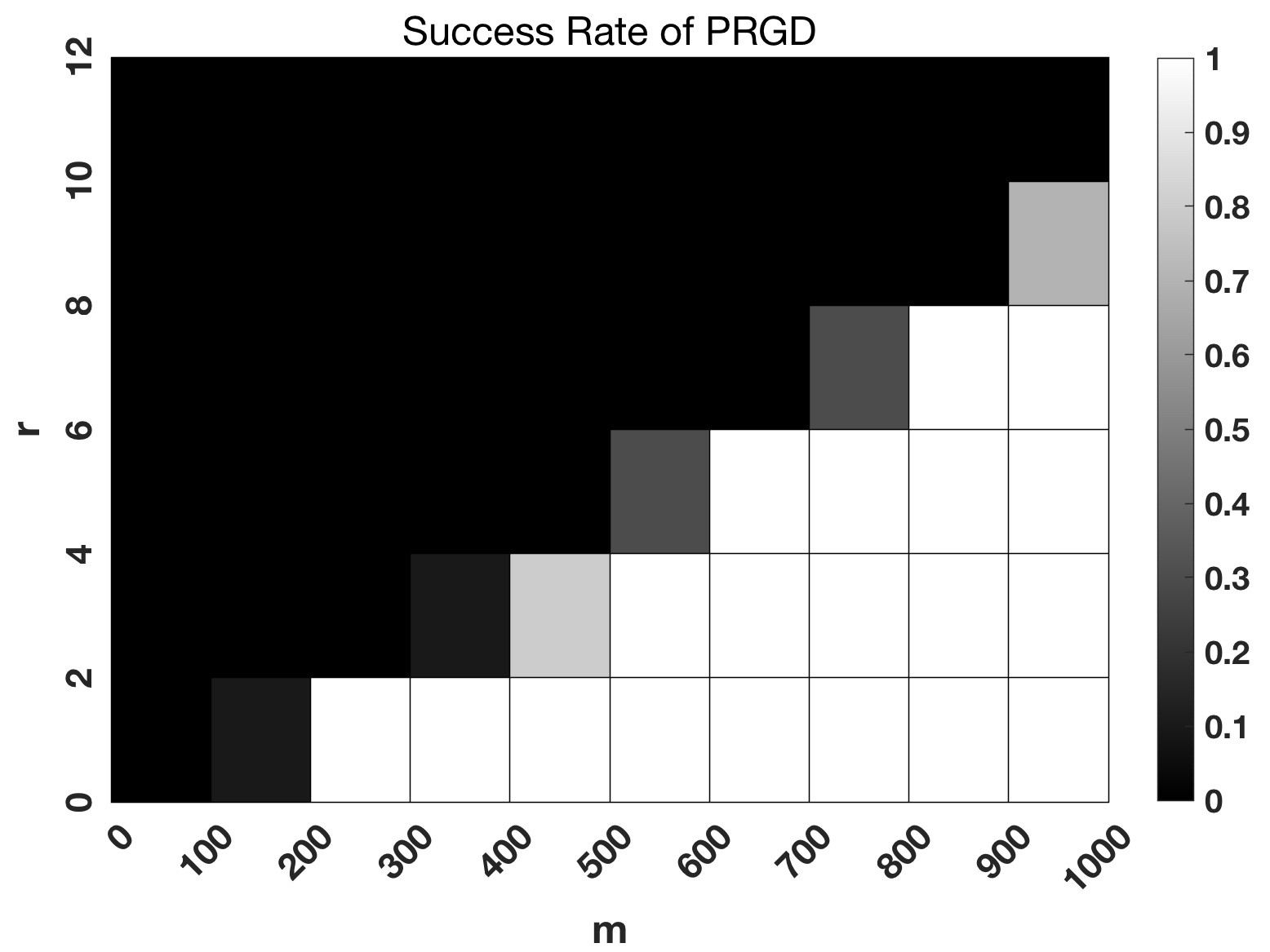} } 	
        \subfloat[{\tt adaptive-PRGD}]{ \includegraphics[width=0.3\linewidth]{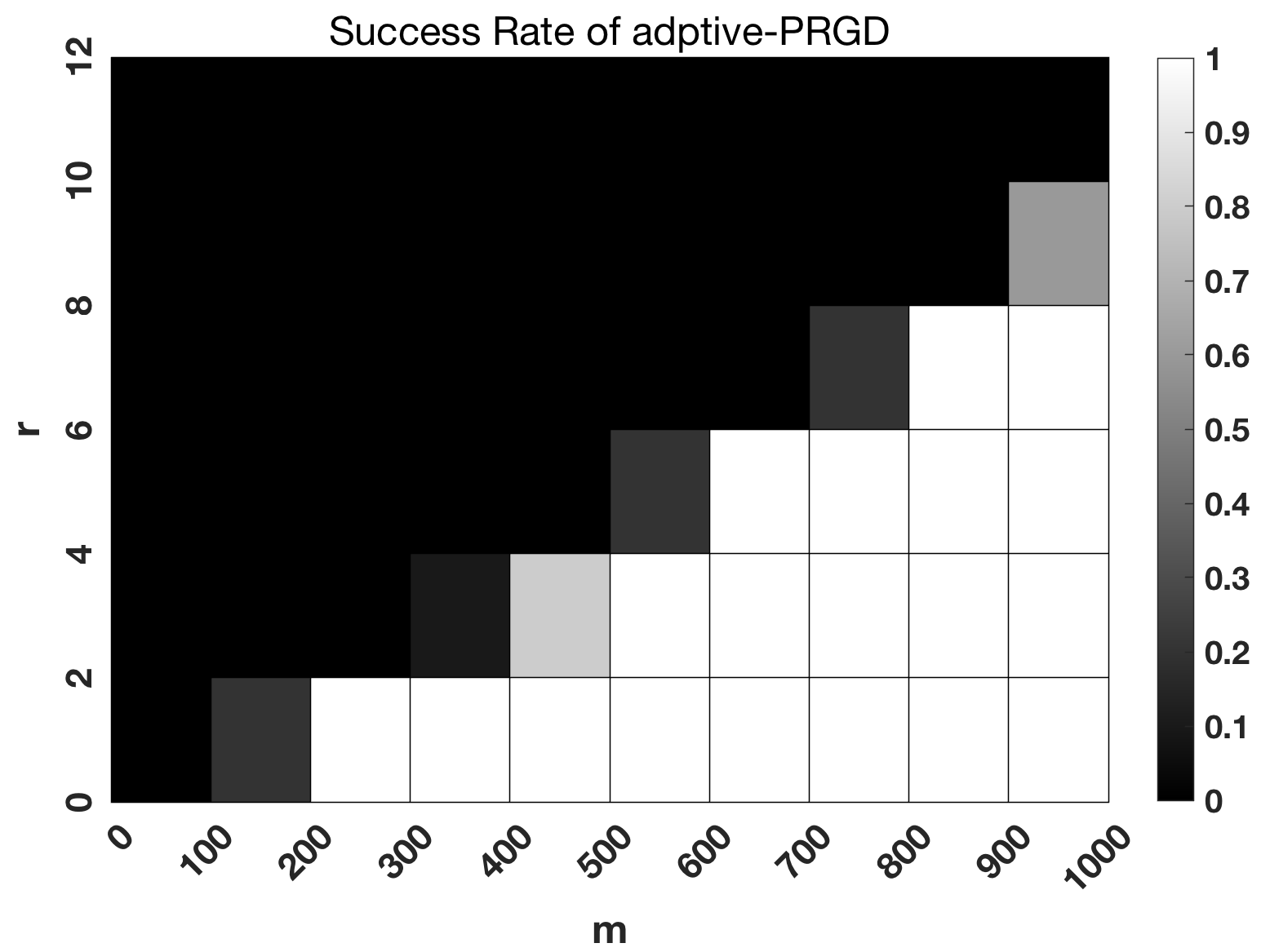} } 	
	\caption{Results of success rate for matrix sensing on simulated data. The unknown matrix has size $50 \times 50$.  The success rate is indicated by the grayscale of the box, where white indicates $100\%$ success rate and black indicates $0\%$ success rate.}
    \label{lmr_3}
\end{figure}

\paragraph{Success Rates.} We test the recovery ability of the PRGD algorithm. We compare the success rates of adaptive-PRGD with adaptive-RGD and PRGD for recovering the test matrix $\bm{X}$. An algorithm is considered to successfully recover the test matrix $\bm{X}$ of rank $r$ if the output value $\bm{X}_t$ satisfies 
$$
\frac{\left\| \bm{X}_t - \bm{X} \right\|_{F}}{\left\| \bm{X} \right\|_{F}} \leq 10^{-4}
$$
within $500$ iterations. We fix the matrix size to $n_1= n_2 =50$ and vary the rank $r \in\{2, 4, 6, 8, 10, 12\}$ and the number of measurements $m \in \{100, 200, 300, 400, 500, 600,  700, 800, 900, 1000\}$. We present the results in Figure \ref{lmr_3}, where we obtain the success rates of all the algorithms through 10 independent trials.

The results in Figure \ref{lmr_3} indicate that the success rate of adaptive-PRGD is comparable to that of PRGD and is better than adaptive-RGD. When the rank $r$ and the number $m$ of measurements are small, the success rate of adaptive-PRGD is higher than that of adaptive-RGD. For instance, when $r=2$ and $m = 300$, PRGD and adaptive-PRGD are $100\%$ successful, while the success rate of adaptive-RGD is $80 \%$. Adaptive-RGD requires more measurements to be $100\%$ successful than adaptive-PRGD and PRGD when the rank $r$ is relatively large. For example, when $r=8$, adaptive-PRGD and PRGD can be $100\%$ successful for $m=900$, while the success rate of adaptive-RGD is $90 \%$. When $m=800$, the success rate of adaptive-PRGD and PRGD is $30 \%$ and $20 \%$, respectively, while the success rate of adaptive-RGD is $0 \%$.

\subsection{Phase Retrieval}
\label{sec:PR}
Finally, we evaluate the performance of PRGD on phase retrieval problems. For simplicity, we consider the real case, but all approaches can be applied to the complex case. In phase retrieval, the goal is to find a vector $\bm{x} \in \mathbb{R}^{n}$ that satisfies the following system of phaseless equations:
\begin{equation}\label{PR}
| \bm{A}x|^2 = \bm{y},
\end{equation}
where $\bm{A} \in \mathbb{R}^{m \times n}$ and $\bm{y} \in \mathbb{R}^{m}$ are known.  We denote $\bm{a}_i^T$ as the $i$-th row of $\bm{A}$. Let $\mathcal{A}$ be a linear operator that maps an $n \times n$ matrix to a vector of length $m$, defined as follows:
$$
\mathcal{A} \bm{Z} = \left\{ \langle \bm{Z}, \bm{a}_i\bm{a}_i^T  \rangle \right\}_{i=1}^{m},  \quad \forall ~\bm{Z} \in \mathbb{R}^{n \times n}.
$$ 
Then, \eqref{PR} is equivalent to
$$
\mathcal{A}\bm{X} = \bm{y},
$$
where $\bm{X} = \bm{x}\bm{x}^T$. Since there is a one-to-one correspondence between $\bm{X}$ and $\bm{x}$, we can seek to reconstruct the rank-$1$ positive semidefinite matrix $\bm{X}$ from $\mathcal{A}\bm{X} = \bm{y}$ instead of reconstructing $\bm{x}$. Therefore, phase retrieval can be formulated as a low-rank matrix recovery problem \eqref{eq:y=AX} and \eqref{eq:defA} with $n_1=n_2=n$, $r=1$, and $\bm{A}_i = \bm{a}_i\bm{a}_i^T$.

We generate the underlying signal $\bm{x}$ with length $n=128$ and i.i.d. Gaussian entries, and the measurement matrix $\bm{A}$, with i.i.d. Gaussian entries. For all experiments, we set $m= 6n$. Since the degree of freedom in the rank-1 matrix $\bm{X}$ is $2n -1$, the oversampling ratio is roughly $1/3$. We use the stopping criterion $\frac{\left\|  \mathcal{A}\bm{X}_t - \bm{y} \right\| }{\left\| \bm{y} \right\| } \leq 10^{-8}$ for all algorithms. 

The results, presented in Figure \ref{PR}(a), show that the adaptive-PRGD algorithm requires significantly fewer iterations than the adaptive-RGD, RGD, and PRGD algorithms. Specifically, the number of iterations for adaptive-PRGD is approximately one-third that of adaptive-RGD and RGD, and the number of iterations for PRGD is comparable to that of adaptive-RGD. This experimental result demonstrates that our data-driven metric is highly effective for the phase retrieval problem. However, in our current implementation, the computational time of adaptive-RGD is less than that of adaptive-PRGD because adaptive-RGD does not require the explicit calculation of the gradient $\bm{G}_t$, while adaptive-PRGD needs the explicit expression of $\bm{G}_t$, which is time-consuming. Future research will explore methods to efficiently implement PRGD for phase retrieval.

\begin{figure}[htbp]
	\centering
	\subfloat[{\tt Gaussian(real)}]{ \includegraphics[width=0.45\linewidth]{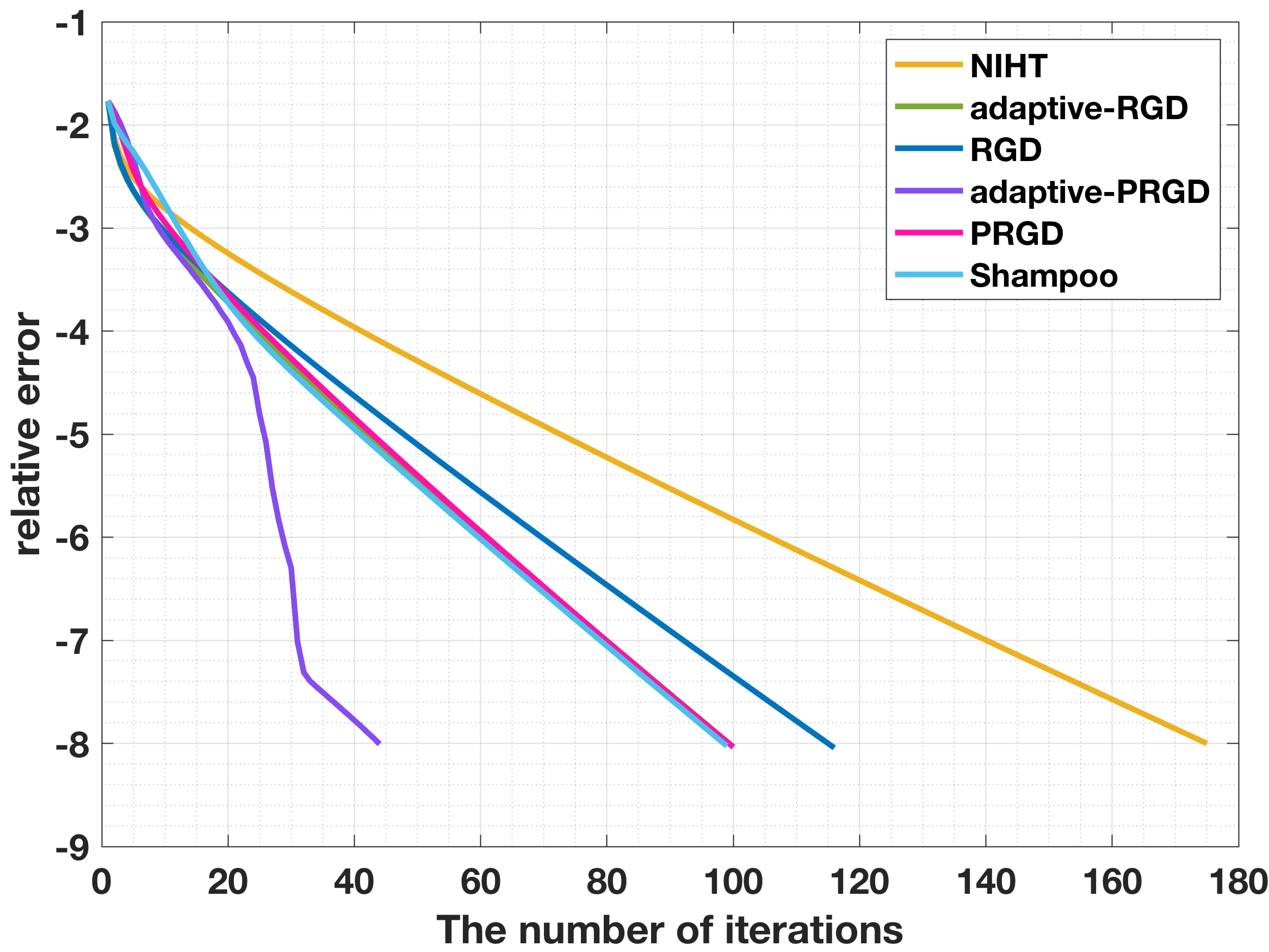} }  	
	\subfloat[{\tt Gaussian(complex)}]{ \includegraphics[width=0.45\linewidth]{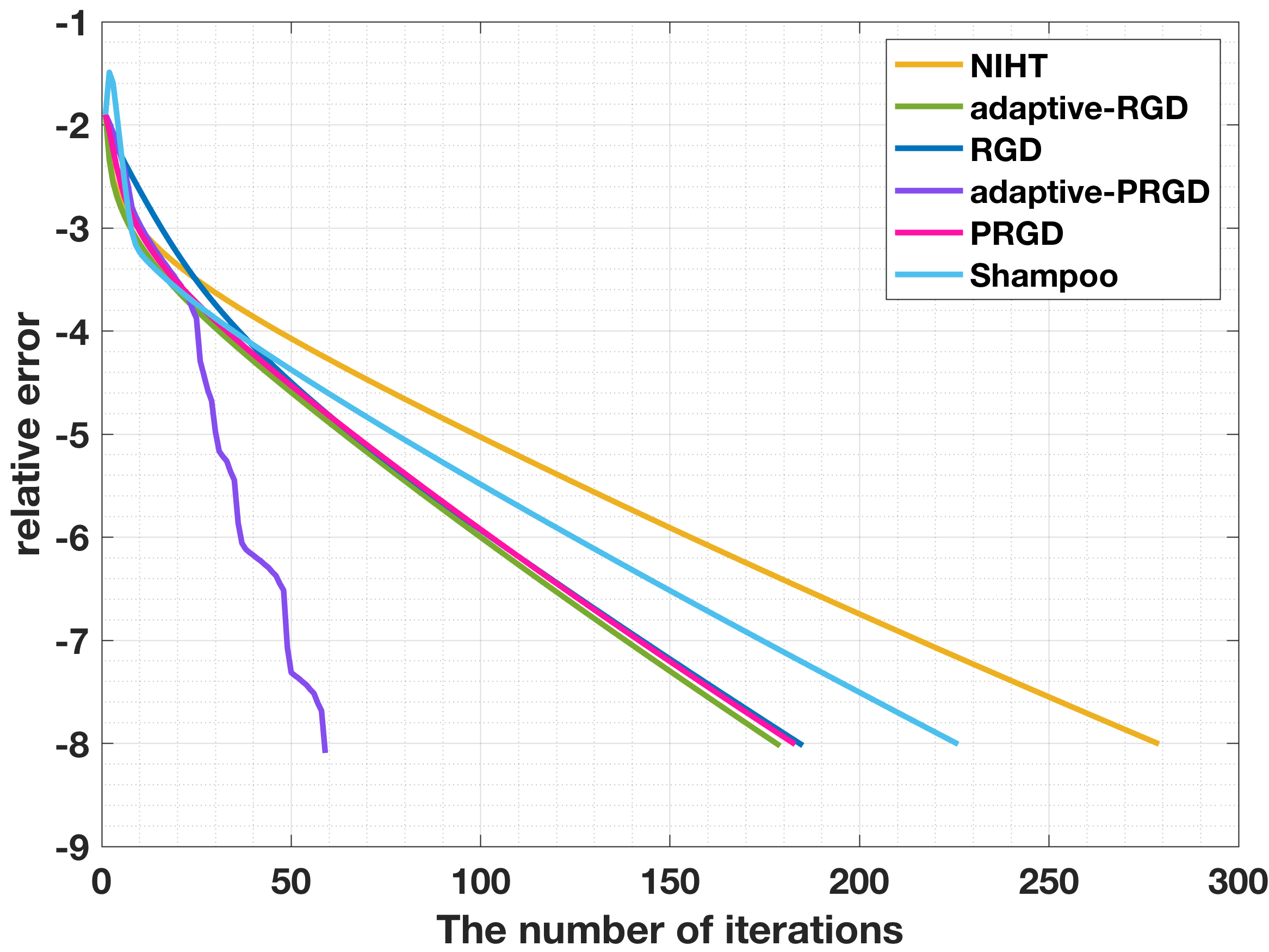} } 	

	\caption{Results of the number of iterations for phase retrieval on simulated data. The unknown signal has a size of $128$.}
    \label{PR}
\end{figure}

The phase retrieval problem and algorithms can be extended from the real case to the complex case, where both the unknown vector $\bm{x}$ and the measurement matrix $\bm{A}$ are complex-valued. We conduct experiments on complex phase retrieval and compare all algorithms. The results are presented in Figure \ref{PR}(b). The findings are consistent with the results obtained in the real case shown in Figure \ref{PR}(a).

\section{Conclusion and Future Direction}
\label{sec:conclusion}
 In this paper, we proposed a preconditioned Riemannian gradient descent (PRGD) algorithm for low-rank matrix recovery problems.  The preconditioner is constructed from the measurement and iteration data, and it is easy to compute.  We proved PRGD converges linearly to the underlying low-rank matrix under the restricted isometry property (RIP). We evaluated the performance of PRGD on various low-rank matrix recovery problems, including low-rank matrix completion, low-rank matrix sensing, and phase retrieval. The experiment results demonstrated that  PRGD outperforms RGD, reducing the number of iterations and being up to 10 times faster than the RGD algorithm for matrix completion, making it a more efficient option for low-rank matrix recovery. Overall, we believe that our proposed preconditioner and PRGD algorithm are significantly efficient for low-rank matrix recovery. 

 However, as demonstrated in Section \ref{sec:PR} of this paper, for the phase retrieval problem, although PRGD can reduce the number of iterations compared to RGD,  the computational time in the current implementation of the algorithm has not been further reduced yet. In the future, we will explore ways to address this issue. Additionally, our proposed new preconditioner and PRGD algorithm can be extended to solve other recovery problems, such as low-rank tensor completion. Furthermore, we can also construct new and more effective preconditioners tailored to different problems. We plan to pursue this line of research in the future.

 \section*{Acknowledgments}
Jian-Feng Cai is partially supported by Hong Kong Research Grant Council GRF 16306821 and GRF 16310620, and Hong Kong Innovation and Technology Fund MHP/009/20. Fengmiao Bian is also partially supported by an outstanding Ph.D. graduate development scholarship from Shanghai Jiao Tong University. 

\begin{small}
\bibliographystyle{plain}
\bibliography{bib}
\end{small}

\end{document}